\documentclass[leqno,11pt]{amsart}

\usepackage{geometry}

\usepackage[all]{xy} % diagrams
\usepackage{amsmath, amssymb, amsfonts, latexsym, mdwlist, amsthm}
\usepackage{subfig}
\usepackage{graphicx}
\usepackage{wrapfig}

\usepackage[bookmarks, colorlinks, breaklinks, pdftitle={Projectivity and Birational Geometry of Bridgeland Moduli spaces},
pdfauthor={Arend Bayer, Emanuele Macri}]{hyperref}
\hypersetup{linkcolor=blue,citecolor=blue,filecolor=black,urlcolor=blue}

\usepackage{tikz}
\usetikzlibrary{calc,trees,positioning,arrows,chains,shapes.geometric,%
    decorations.pathreplacing,decorations.pathmorphing,shapes,%
    matrix,shapes.symbols}

\tikzset{
>=stealth',
  punktchain/.style={
    rectangle,
    rounded corners,
    draw=black, thick,
    minimum height=3em,
    text centered,
    on chain},
  line/.style={draw, thick, <-},
  element/.style={
    tape,
    top color=white,
    bottom color=blue!50!black!60!,
    minimum width=8em,
    draw=blue!40!black!90, very thick,
    text width=10em,
    minimum height=3.5em,
    text centered,
    on chain},
  every join/.style={->, thick,shorten >=1pt},
  decoration={brace},
  tuborg/.style={decorate},
  tubnode/.style={midway, right=2pt},
}

\usepackage{paralist}
\setdefaultenum{(a)}{(i)}{}{}
\usepackage{enumitem} % for space-saving description environments}

\def\C{\ensuremath{\mathbb{C}}}

\def\H{\ensuremath{\mathbb{H}}}

\def\P{\ensuremath{\mathbb{P}}}
\def\Q{\ensuremath{\mathbb{Q}}}
\def\R{\ensuremath{\mathbb{R}}}
\def\Z{\ensuremath{\mathbb{Z}}}

\def\alg{\mathrm{alg}}
\def\Amp{\mathrm{Amp}}
\def\Aut{\mathop{\mathrm{Aut}}\nolimits}

\def\ch{\mathop{\mathrm{ch}}\nolimits}

\def\Coh{\mathop{\mathrm{Coh}}\nolimits}

\def\cone{\mathop{\mathrm{cone}}}

\def\deg{\mathop{\mathrm{deg}}}

\def\dim{\mathop{\mathrm{dim}}\nolimits}
\def\ev{\mathop{\mathrm{ev}}\nolimits}

\def\Ext{\mathop{\mathrm{Ext}}\nolimits}
\def\GL{\mathop{\mathrm{GL}}}
\def\Hilb{\mathop{\mathrm{Hilb}}\nolimits}
\def\Hom{\mathop{\mathrm{Hom}}\nolimits}

\def\RlHom{\mathop{\mathbf{R}\mathcal Hom}\nolimits}

\def\id{\mathop{\mathrm{id}}\nolimits}

\def\Nef{\mathrm{Nef}}
\def\num{\mathop{\mathrm{num}}\nolimits}

\def\NS{\mathop{\mathrm{NS}}\nolimits}

\def\perf{\mathop{\mathrm{perf}}}
\def\Pic{\mathop{\mathrm{Pic}}}

\def\rk{\mathop{\mathrm{rk}}}

\def\ST{\mathop{\mathrm{ST}}\nolimits}

\def\td{\mathop{\mathrm{td}}\nolimits}

\newenvironment{Prf}{\textit{Proof.}\/}{\hfill$\Box$}

\def\MG13{\ensuremath{{\mathcal M}_{\Gamma_1(3)}}}
\def\tildeMG13{\ensuremath{\widetilde{\mathcal M}_{\Gamma_1(3)}}}
\def\Stab{\mathop{\mathrm{Stab}}}
\def\into{\ensuremath{\hookrightarrow}}
\def\onto{\ensuremath{\twoheadrightarrow}}

\def\blank{\underline{\hphantom{A}}}

\def\Db{\mathrm{D}^{\mathrm{b}}}

\def\abs#1{\left\lvert#1\right\rvert}

\newcommand\stv[2]{\left\{#1\,\colon\,#2\right\}}

\makeatletter
\newtheorem*{rep@theorem}{\rep@title}
\newcommand{\newreptheorem}[2]{%
\newenvironment{rep#1}[1]{%
 \def\rep@title{#2 \ref{##1}}%
 \begin{rep@theorem}}%
 {\end{rep@theorem}}}
\makeatother

\newtheorem{Thm}{Theorem}[section]
\newreptheorem{Thm}{Theorem}
\newtheorem{Prop}[Thm]{Proposition}
\newtheorem{PropDef}[Thm]{Proposition and Definition}
\newtheorem{Lem}[Thm]{Lemma}
\newtheorem{PosLem}[Thm]{Positivity Lemma}
\newtheorem{Cor}[Thm]{Corollary}
\newreptheorem{Cor}{Corollary}
\newtheorem{Con}[Thm]{Conjecture}
\newreptheorem{Con}{Conjecture}

\newtheorem{thm-int}{Theorem}

\theoremstyle{definition}
\newtheorem{Def-s}[Thm]{Definition}
\newtheorem{Def}[Thm]{Definition}
\newtheorem{Rem}[Thm]{Remark}

\newtheorem{Ex}[Thm]{Example}

\def\C{\ensuremath{\mathbb{C}}}

\def\H{\ensuremath{\mathbb{H}}}

\def\P{\ensuremath{\mathbb{P}}}
\def\Q{\ensuremath{\mathbb{Q}}}
\def\R{\ensuremath{\mathbb{R}}}
\def\Z{\ensuremath{\mathbb{Z}}}

\def\AA{\ensuremath{\mathcal A}}

\def\CC{\ensuremath{\mathcal C}}

\def\EE{\ensuremath{\mathcal E}}
\def\FF{\ensuremath{\mathcal F}}

\def\HH{\ensuremath{\mathcal H}}
\def\II{\ensuremath{\mathcal I}}

\def\LL{\ensuremath{\mathcal L}}

\def\NN{\ensuremath{\mathcal N}}
\def\OO{\ensuremath{\mathcal O}}
\def\PP{\ensuremath{\mathcal P}}

\def\TT{\ensuremath{\mathcal T}}

\def\ZZ{\ensuremath{\mathcal Z}}

\def\LLL{\mathfrak L}
\def\MMM{\mathfrak M}

\def\simpos{\sim_{\R^+}}

\newcommand{\ignore}[1]{}

\begin{document}

\title{Projectivity and Birational Geometry of Bridgeland moduli spaces}

\author{Arend Bayer}
\address{Department of Mathematics, University of Connecticut U-3009, 196 Auditorium Road, Storrs, CT 06269-3009, USA}
\curraddr{School of Mathematics,
The University of Edinburgh,
James Clerk Maxwell Building,
The King's Buildings, Mayfield Road, Edinburgh, Scotland EH9 3JZ,
United Kingdom}
\email{arend.bayer@ed.ac.uk}
\urladdr{http://www.maths.ed.ac.uk/~abayer/}

\author{Emanuele Macr\`i}
\address{Department of Mathematics, The Ohio State University, 231 W 18th Avenue, Columbus, OH 43210-1174, USA}
\email{macri.6@math.osu.edu}
\urladdr{http://www.math.osu.edu/~macri.6/}



\begin{abstract}
We construct a family of nef divisor classes on every moduli space of stable complexes in the sense of Bridgeland.
This divisor class varies naturally with the Bridgeland stability condition. 
For a generic stability condition on a K3 surface, we prove that this class is ample, thereby generalizing a result of Minamide, Yanagida, and Yoshioka.
Our result also gives a systematic explanation of the relation between wall-crossing for
Bridgeland-stability and the minimal model program for the moduli space.

We give three applications of our method for classical moduli spaces of sheaves
on a K3 surface:

1. We obtain a region in the ample cone in the moduli space of Gieseker-stable sheaves
only depending on the lattice of the K3.

2. We determine the nef cone of the Hilbert scheme of $n$ points on a K3 surface of Picard rank one when $n$ is large compared to the genus.

3. We verify the ``Hassett-Tschinkel/Huybrechts/Sawon'' conjecture on 
the existence of a birational Lagrangian fibration
for the Hilbert scheme in a new family of cases.
\end{abstract}

\vspace{-1em}

\maketitle

\setcounter{tocdepth}{1}
\tableofcontents

\section{Introduction}\label{sec:intro}

In this paper, we give a canonical construction of determinant line bundles on any moduli space $M$ of
Bridgeland-semistable objects. Our construction has two advantages over the classical construction
for semistable sheaves: our divisor class varies naturally with the stability condition, and we can
show that our divisor is automatically nef.

This also explains a picture envisioned by Bridgeland, and observed in examples by
Arcara-Bertram and others, that relates wall-crossing under a change of stability condition to the
birational geometry and the minimal model program of $M$.  As a result, we can \emph{deduce}
properties of the birational geometry of $M$ from wall-crossing; this leads to new results even when
$M$ coincides with a classical moduli space of Gieseker-stable sheaves. 

\subsection*{Moduli spaces of complexes}%\label{subsec:IntroModuli}
Moduli spaces of complexes first appeared in \cite{Bridgeland:Flop}:
the flop of a smooth threefold $X$ can be constructed as a moduli space parameterizing
\emph{perverse ideal sheaves} in the derived category of $X$.
Recently, they have turned out to be extremely useful 
in Donaldson-Thomas theory; see \cite{Toda:Survey} for a survey.

%and \cite{PT1, Joyce-Song, Kontsevich-Soibelman:stability, large-volume,Toda:limit-stable, Toda:PTDT, Bridgeland:PTDT} for related results.

Ideally, the necessary notion of stability of complexes can be given in terms of
Bridgeland's notion of a stability condition on the derived category, introduced in
\cite{Bridgeland:Stab}. Unlike other constructions 
(as in \cite{Toda:limit-stable, large-volume}), the space of Bridgeland
stability conditions admits a well-behaved wall and chamber structure: the moduli
space of stable objects with given invariants remains unchanged unless the stability
condition crosses a wall.
However, unlike Gieseker-stability for sheaves, Bridgeland stability is not a priori
connected to a GIT problem. As a consequence, while established methods
(\cite{Inaba,Lieblich:mother-of-all,Toda:K3, Abramovich-Polishchuk:t-structures})
can prove existence of moduli spaces as algebraic spaces or Artin stacks,
there are so far only ad-hoc methods to prove that they are projective, or to construct coarse moduli spaces.

%In particular, it is not clear whether coarse moduli spaces of semistable complexes exist, whether
%they are projective, and what is their birational geometry.

In this paper, we propose a solution to this problem by constructing a family of numerically
positive divisor classes on any moduli space of Bridgeland-stable complexes.

%In particular, by varying a stability condition, it is expected (\cite{Bridgeland:K3,Aaron-Daniele}) that, when moduli spaces cross a wall, they undergo birational transformations in a similar fashion as for the Mumford-Thaddeus principle for (twisted) Gieseker semistable sheaves (\cite{MatsukiWenthworth:TwistedVariation}).

\subsection*{A family of nef divisors on Bridgeland-moduli spaces}%\label{subsec:IntroNef}

Let $X$ be a smooth projective variety over $\C$.
We denote by $\Db(X)$ its bounded derived category of coherent sheaves, and by $\Stab(X)$ the
space of Bridgeland stability conditions on $\Db(X)$, see Section \ref{sec:Bridgeland}. We refer to
p.~\pageref{subsec:notation} for an overview of notations.

Let $\sigma = (Z, \AA) \in \Stab(X)$ be a stability condition,
and $v$ a choice of numerical invariants.
Assume that we are given a family $\EE \in \Db(S \times X)$
of $\sigma$-semistable objects of class $v$ parameterized by
a proper algebraic space $S$ of finite type over $\C$; for example, $S$ could be a fine moduli space
$M_{\sigma}(v)$ of stable objects.
We define a numerical Cartier divisor class
$\ell_{\sigma, \EE}\in N^1(S) = \Hom(N_1(S), \R)$ as follows:
for any projective integral curve $C\subset S$, we set
\begin{equation} \label{eq:firstellDef}
\ell_{\sigma,\EE}([C]) =  \ell_{\sigma,\EE}.C := \Im \left(-\frac{Z\bigl(\Phi_{\EE}(\OO_C)\bigr)}{Z(v)}\right)
= \Im \left(-\frac{Z\bigl( (p_X)_* \EE|_{C \times X}\bigr)}{Z(v)} \right)
\end{equation}
where $\Phi_\EE\colon \Db(S)\to\Db(X)$ is the Fourier-Mukai functor with kernel $\EE$,
and $\OO_C$ is the structure sheaf of $C$. It is easy to prove that \eqref{eq:firstellDef} defines a
numerical divisor class $\ell_{\sigma, \EE} \in N^1(S)$.
Our main result, the Positivity Lemma \ref{Positivity}, implies the positivity of this divisor:
\begin{Thm}\label{thm:main1}
The divisor class $\ell_{\sigma, \EE}$ is nef.
Additionally, we have $\ell_{\sigma, \EE}.C = 0$ if and only if for two general points $c, c' \in C$, the corresponding objects $\EE_c, \EE_{c'}$ are $S$-equivalent.
\end{Thm}
(Two semistable objects are $S$-equivalent if their Jordan-H\"older filtrations 
into stable factors of the same phase have identical stable factors.)
The class $\ell_{\sigma, \EE}$ can also be given as a determinant line bundle. The
main advantage of our construction 
is that we can show its positivity property directly, without using GIT; instead,
the proof is based on a categorical construction by Abramovich and Polishchuk
\cite{Abramovich-Polishchuk:t-structures, Polishchuk:families-of-t-structures}.
Our construction also avoids any additional choices: it depends only on $\sigma$.

%When $\sigma$ varies within a chamber of $\Stab(X)$ for which the moduli
%$M_{\sigma}(v)$ remains constant, we obtain a family of nef divisors on $M_\sigma(v)$; in examples,
%this family covers the entire nef cone of $M_{\sigma}(v)$.

\subsection*{Chambers in $\Stab(X)$ and the nef cone of the moduli spaces}
Consider a chamber $\CC$ for the wall-and-chamber decomposition with respect to $v$; then (assuming
its existence) the moduli space $M_{\CC}(v):=M_{\sigma}(v)$ of $\sigma$-stable objects of class $v$
is constant for $\sigma \in \CC$. Also
assume for simplicity that it admits a universal family $\EE$. Theorem \ref{thm:main1}
yields an essentially linear map 
\begin{equation} \label{eq:lasmap}
l \colon \overline{\CC} \to \Nef(M_{\CC}(v)), \quad l(\sigma) = \ell_{\sigma, \EE}.
\end{equation}
This immediately begs for the following two questions:
\begin{description}
\item[Question 1] Do we actually have $l(\CC) \subset \Amp(M_{\CC}(v))$?
\item[Question 2] What will happen at the walls of $\CC$?
\end{description}

\subsection*{K3 surfaces: Overview}
While our above approach is very general, we now restrict to the case where $X$ is a 
smooth projective K3 surface.
In this situation, Bridgeland described (a connected component of) the space of stability conditions in
\cite{Bridgeland:K3}, and Toda proved existence results for moduli spaces in
\cite{Toda:K3}; see Section \ref{sec:reviewK3}. The following paraphrases a conjecture proposed by Bridgeland
in Section 16 of the arXiv-version of \cite{Bridgeland:K3}:

\begin{Con}[Bridgeland]\label{conj:Bridgeland}
Given a stability condition $\sigma$ on a K3 surface, and a numerical class $v$, there exists a coarse
moduli space $M_{\sigma}(v)$ of $\sigma$-semistable complexes with class $v$.
Changing the stability condition produces birational maps between the coarse moduli spaces.
\end{Con}

Our main results give a partial proof of this conjecture, and answers to the above questions:
Theorem \ref{thm:ProjK3} answers Question 1 and proves existence of coarse moduli spaces;
Theorem \ref{thm:MMP} partially answers Question 2 and partially proves the second statement
of the conjecture.
They also give a close relation between walls in $\Stab(X)$, and walls in
the movable cone of the moduli space separating nef cones of different birational models.

\subsection*{Projectivity of the moduli spaces}\label{subsec:IntroK3Surfaces}
%We show that the assignment \eqref{eq:firstellDef} induces an ample divisor, denoted
%$\ell_\sigma$, on the moduli space.
Assume that $\sigma$ is \emph{generic}, which means that it does not lie on a wall
with respect to $v$.

\begin{Thm}\label{thm:ProjK3}
Let $X$ be a smooth projective K3 surface, and let $v\in H^*_{\alg}(X, \Z)$.
Assume that the stability condition $\sigma$ is generic with respect to $v$.
Then:
\begin{enumerate}
\item The coarse moduli space $M_{\sigma}(v)$ of $\sigma$-semistable objects with Mukai
vector $v$ exists as a normal projective irreducible variety with $\Q$-factorial singularities.
\item \label{enum:ample}
The assignment \eqref{eq:firstellDef} induces an ample divisor class $\ell_{\sigma}$ on $M_{\sigma}(v)$.
\end{enumerate}
\end{Thm}
This generalizes \cite[Theorem 0.0.2]{MYY2}, which shows projectivity of $M_{\sigma}(v)$
in the case where $X$ has Picard rank one.

\subsection*{Wall-crossing and birational geometry of the moduli spaces}
We also use Theorem \ref{thm:main1} to study the wall-crossing behavior of the moduli space
under deformations of $\sigma$.

Assume that $v$ is a primitive class.
Let $W$ be a wall of the chamber decomposition for $v$.  Let $\sigma_0$ be a generic point of $W$,
and let $\sigma_{+}, \sigma_{-}$ be two stability conditions nearby on each side of the wall.
By Theorem \ref{thm:ProjK3} and its proof, they are smooth projective Hyperk\"ahler varieties.
Since being semistable is a closed condition in $\Stab(X)$, the (quasi-)universal families $\EE_{\pm}$ on
$M_{\sigma_{\pm}}(v)$ are also families of $\sigma_0$-semistable objects. Theorem \ref{thm:main1} also
applies in this situation, and thus $\sigma_0$ produces nef divisor classes $\ell_{\sigma_0,\EE_\pm}$ on
$M_{\sigma_{\pm}}(v)$.  In Section \ref{sec:flops}, we prove:

\begin{Thm}\label{thm:MMP}
Let $X$ be a smooth projective K3 surface, and $v\in H^*_{\alg}(X, \Z)$ be a primitive class.
\begin{enumerate*}
\item \label{enum:contraction}
The classes $\ell_{\sigma_0,\EE_\pm}$ are big and nef, and induce birational contraction morphisms
\[
\pi_{\sigma_\pm}\colon M_{\sigma_{\pm}}(v) \to Y_{\pm},
\]
where $Y_{\pm}$ are normal irreducible projective varieties.
\item \label{enum:flop}
If there exist $\sigma_0$-stable objects, and if their complement in $M_{\sigma_{\pm}}(v)$ has
codimension at least two $2$, then one of the following two possibilities holds:
\begin{itemize*}
\item Both $\ell_{\sigma_0,\EE_+}$ and $\ell_{\sigma_0,\EE_-}$ are ample, and the birational map
\[
f_{\sigma_0} \colon M_{\sigma_{+}}(v) \dashrightarrow M_{\sigma_{-}}(v),
\]
obtained by crossing the wall in $\sigma_0$, extends to an isomorphism.
\item Neither $\ell_{\sigma_0,\EE_+}$ nor $\ell_{\sigma_0, \EE_-}$ is ample, and $f_{\sigma_0} \colon M_{\sigma_{+}}(v) \dashrightarrow M_{\sigma_{-}}(v)$
is the flop induced by $\ell_{\sigma_0,\EE_+}$: we have a commutative diagram of birational maps
\begin{equation*}
\xymatrix{ M_{\sigma_{+}}(v)\ar@{-->}[rr]^{f_{\sigma_0}}\ar[dr]_{\pi_{\sigma_+}} && M_{\sigma_{-}}(v)\ar[dl]^{\pi_{\sigma_-}}\\
& Y_+=Y_{-} &
},
\end{equation*}
and $f_{\sigma_0}^*\ell_{\sigma_0,\EE_-}=\ell_{\sigma_0,\EE_+}$.
\end{itemize*}
\end{enumerate*}
\end{Thm}

Note that our theorem does not cover the cases where there are no
$\sigma_0$-stable objects, or the case where their complement has codimension one\footnote{Our only
result in the latter case is Lemma \ref{lem:continuous}, which shows 
$f_{\sigma_0}^* \ell_{\sigma_0,\EE_-} = \ell_{\sigma_0,\EE_+}$.}

In some examples (including those considered in \cite{Aaron-Daniele, MaciociaMeachan}), we can show that
$Y_{+} = Y_-$ is a connected component of the coarse moduli space of $\sigma_0$-semistable objects of
class $v$. The contraction from the moduli space of 
Gieseker-stable sheaves to the Uhlenbeck moduli space of $\mu$-semistable sheaves
(see \cite{LePotier:Determinant,JunLi:Uhlenbeck}) is another particular example of the contraction morphism
$\pi_{\sigma_\pm}$; this was observed on the level of sets of semistable objects in
\cite{LoQin:miniwalls}, and on the level of moduli spaces in the recent preprint
\cite{Jason:Uhlenbeck}. 

We study many more examples of wall-crossing behavior in Sections \ref{sec:K3sheaves} and \ref{sec:Hilbert}.

\subsection*{Nef cones of moduli spaces of stable sheaves}
Our Theorem \ref{thm:main1} can give new results on the nef cone
of $M_\sigma(v)$ even in the situation where it agrees with a classical moduli space of 
sheaves. We give two examples of such applications:
\begin{itemize}
\item In \textbf{Corollary \ref{Cor:amplecone}}, we determine a region 
of the ample cone of the moduli space of Gieseker-stable sheaves on $X$ that depends only on the lattice 
of $X$.
\item In \textbf{Proposition \ref{prop:Hilbnefcone}}, we determine the nef cone of the Hilbert scheme
of $n$ points on a K3 surface of Picard rank one and genus $g$ for $n \ge \frac g2 + 1$.
Our result shows in particular that \cite[Conjecture 1.2]{HassettTschinkel:ExtremalRays}
will need to be modified (see Remark \ref{rem:HT_ER_conjecture}).
\end{itemize}
The strength of our approach is that it simultaneously produces nef divisors (by Theorem \ref{thm:main1})
and extremal rays of the Mori cone (as curves of $S$-equivalent objects on the wall).

\subsection*{Lagrangian fibrations for the Hilbert scheme} 
Let $X$ be a K3 surface with $\Pic X = \Z \cdot H$, and of degree $d = \frac 12 H^2$. 
We consider the Hilbert scheme $\Hilb^n(X)$ of $n$ points on $X$.
According to a long-standing conjecture for Hyperk\"ahler varieties, it has a birational model
admitting a Lagrangian fibration if and only if $d = \frac{k^2}{h^2}(n-1)$ for some integers $k, h$.
The following theorem solves the conjecture in the case $h = 2$, and interprets all birational models as moduli
spaces of Bridgeland-stable objects.

We denote by $\tilde H \subset \Hilb^n(X)$ the divisor of subschemes intersecting a given curve in
the linear system $\abs{H}$, and by $B \subset \Hilb^n(X)$ the reduced divisor of non-reduced
subschemes.
\begin{repThm}{thm:TBHTHS}
Let $X$ be a K3 surface with $\Pic X = \Z \cdot H$ and
$H^2 = 2d$. Assume that there is an odd integer $k$ with $d = \frac{k^2}4(n-1)$ for some
integer $n$. Then:
\begin{enumerate}
\item
The movable cone $\mathrm{Mov}(\mathrm{Hilb}^n(X))$ is generated by $\tilde H$ and $2 \tilde H - k B$.
\item 
The morphism induced by $\tilde H$ is the Hilbert-to-Chow morphism, while the one induced by $2 \tilde H - k B$ is a Lagrangian fibration on a minimal model for $\Hilb^n(X)$.
\item
All minimal models for $\mathrm{Hilb}^n(X)$ arise as moduli spaces of stable objects in $\Db(X)$ and their birational transformations are induced by crossing a wall in $\Stab^\dagger(X)$.
\end{enumerate}
\end{repThm}

\subsection*{Some relations to existing work}
\subsubsection*{Wall-crossing}
Our construction was directly inspired by the results of \cite{ABCH:MMP}. The authors studied
wall-crossing for the Hilbert scheme of points on $\P^2$, and found a surprisingly direct relation
between walls $\Stab(\P^2)$ and walls in the movable cone of the Hilbert scheme separating nef
cones of different birational models.

In their case, the variation of moduli spaces can also be seen as a variation of GIT parameters,
via the classical monad construction.  More precisely, stable complexes with respect to a Bridgeland
stability condition can be seen as $\theta$-stable representations of a Beilinson quiver for $\P^2$,
in the sense of King \cite{King:QuiverStability}.  In this case, our divisor class $\ell_\sigma$
agrees with the ample divisor coming from the affine GIT construction of these moduli spaces. More
generally, our family of nef divisors generalizes a construction by Craw and Ishii in
\cite{Alastair-Ishii} that produces a family of nef divisors on moduli spaces of $\theta$-stable
quiver representations.

The analogue of Theorem \ref{thm:ProjK3} for abelian surfaces has been proved in \cite{MYY2}; 
a different method to prove projectivity was established in \cite{MaciociaMeachan, Maciocia:walls}.

Our main result implies that knowing the precise location of walls in $\Stab(X)$ has immediate
applications to the geometry of the nef cone and the movable cone of the moduli spaces of stable
objects. Various authors have considered the geometry of walls in $\Stab(X)$ explicitly: first
examples were found by Arcara and Bertram in \cite{Aaron-Daniele, AB:Reider}; the case where $X$ is
an abelian surface has been studied by Minamide, Yanagida and Yoshioka in
\cite{Minamide-Yanagida-Yoshioka:wall-crossing, MYY2,
YY:abeliansurfaces}, and  by Meachan and Maciocia in \cite{MaciociaMeachan, Maciocia:walls};  Lo and
Qin studied the case of arbitrary surfaces in \cite{LoQin:miniwalls}. As an example, our Corollary
\ref{Cor:amplecone} is a straightforward combination of our main result and the main result
of Kawatani in \cite{Kawatani:Gieseker_vs_Bridgeland}; the above-mentioned authors had all
used similar methods.

\subsubsection*{Proof of projectivity} 
In \cite{Minamide-Yanagida-Yoshioka:wall-crossing,MYY2}, the authors use a beautiful observation
to reinterpret any moduli space of Bridgeland-semistable complexes on a K3 surface with
Picard rank one as a moduli space of Gieseker-semistable sheaves on a Fourier-Mukai partner $Y$ of $X$.
Our proof of Theorem \ref{thm:ProjK3} is directly based on a generalization of their idea to
K3 surfaces of arbitrary Picard rank.

\subsubsection*{Strange duality}
Our construction, and specifically Corollary \ref{Cor:amplecone} may prove useful for Le Potier's
Strange Duality Conjecture, see \cite{LePotier:StrangeDuality} and
\cite{Alina-Dragos-Yoshioka:Genericstrangeduality}. While Le Potier's construction produces line
bundles with sections, it is more difficult to show that these line bundles are nef; the Positivity
Lemma can fill this gap.

\subsubsection*{Markman's monodromy operators}
It would be very interesting to relate our picture to results by Markman on the movable cone in
\cite{Eyal:survey}. Markman proves that the closure of the movable cone is a fundamental domain for
a natural group action on the cone of big divisors. The group is generated by reflections, which
presumably correspond to walls where $\ell_{\sigma_0, \EE_\pm}$ induce divisorial contractions; we expect them to
behave similarly to the ``bouncing walls'' appearing in Sections \ref{sec:K3sheaves} and
\ref{sec:Hilbert}. The two maps $l_{\pm}$ of equation \eqref{eq:lasmap} for the two adjacent
chambers can likely be identified via the monodromy operators introduced in \cite{Eyal:monodromy,
Eyal:monodromy2}.

\subsubsection*{Variation of GIT}
The idea of varying a Bridgeland stability condition is a direct generalization of varying the polarization for Gieseker stability on sheaves.
The latter was studied for surfaces in \cite{EllingsrudGottsche:Variation, FriedmanQin:Variation,
MatsukiWenthworth:TwistedVariation}, using variation of GIT \cite{Thaddeus:GIT-flips, DolgachevHu:Variation}.
For K3 surfaces, the advantage of Bridgeland stability arises since the space of stability conditions has the correct dimension to span an open subset of the movable cone of every moduli space.

\subsubsection*{Lagrangian fibrations}
As indicated above, Theorem \ref{thm:TBHTHS} solves a special case of a 
conjecture on Lagrangian fibrations for arbitrary Hyperk\"ahler varieties. 
This appeared in print in articles
by Hassett-Tschinkel \cite{HassettTschinkel:RationalCurves}, Huybrechts \cite{GrossHuybrechtsJoyce}
and Sawon \cite{Sawon:AbelianFibred}, and had been proposed earlier (see
\cite{Verbitsky:HyperkaehlerSYZ}).
In the case of the  Hilbert scheme on a K3 surface with Picard rank one, it was
proved for $d = k^2(n-1)$ independently 
by Markushevich \cite{Markushevich:Lagrangian} and Sawon \cite{Sawon:LagrangianFibrations}, and
the case of $d = \frac 1{h^2} (n-1)$ by Kimura
and Yoshioka \cite{KY:birationalmapsVB}. Our proof is based on a Fourier-Mukai transform just as the
proofs of Sawon and Markushevich; Bridgeland stability gives a more systematic tool to control
generic preservation of stability under the transform.

%\footnote{Their construction can be described in quite
%classical terms: in this case, a generic element $I_Y \in \Hilb^n(X)$ admits a unique map
%$\OO(-h\cdot H) \into I_Y$; generically, the cokernel will be stable and thus an element of
%$M(0, h\cdot H, 1)$, which admits a Lagrangian fibration.}

\subsection*{Open questions} Theorem \ref{thm:MMP}, and in particular case \eqref{enum:flop}, does
not treat the case of ``totally semistable walls'', which is the case where there is no
$\sigma_0$-stable complex.  Proving a similar result in general would lead to further progress
towards determining the movable cone for moduli spaces of stable sheaves (for general results and
conjectures on the movable cone for an Hyperk\"ahler manifold, see
\cite{HassettTschinkel:MovingCone, HassettTschinkel:ExtremalRays}); in particular, it would likely
imply the above-mentioned conjecture on Lagrangian fibrations for any moduli space of
Gieseker-stable sheaves on a K3 surface.

We will treat this case in \cite{BM:walls}.

\subsection*{Outline of the paper}
Sections \ref{sec:Bridgeland}---\ref{sec:DivisorClassModuliSpace} treat the case of an arbitrary 
smooth projective variety $X$, while Sections \ref{sec:ReviewGieseker}---\ref{sec:Hilbert} are devoted
to the case of K3 surfaces.

Section \ref{sec:Bridgeland} is a brief review of the notion of Bridgeland stability condition, and
sections \ref{sec:MainConstr} and \ref{sec:DivisorClassModuliSpace} are devoted to the proof of 
Theorem \ref{thm:main1}. 
The key ingredient is a construction by Abramovich and Polishchuk
in \cite{Abramovich-Polishchuk:t-structures, Polishchuk:families-of-t-structures}:
Given the t-structure on $X$ associated to a Bridgeland stability condition, their construction produces a
t-structure on $\Db(S\times X)$, for any scheme $S$.  This categorical ingredient allows us to transfer the basic
positivity of the ``central charge'', see equation \eqref{eq:Zpositivity}, to the positivity of
$\ell_{\sigma,\EE}$ as a divisor.

As indicated above, the second part of the paper is devoted to the case of K3 surfaces.
Sections \ref{sec:ReviewGieseker} and \ref{sec:reviewK3} recall background about 
Gieseker-stability and Bridgeland stability conditions, respectively.

We prove Theorem \ref{thm:ProjK3} in Section \ref{sec:ProjK3}. We show projectivity of the moduli space
by identifying it with a moduli space of semistable sheaves via a Fourier-Mukai transform,
generalizing an idea in \cite{MYY2}. 
Section \ref{sec:flops} contains the proof of Theorem \ref{thm:MMP}. 
Since the moduli spaces are $K$-trivial and $\ell_{\sigma_0,\EE_\pm}$ are nef by Theorem
\ref{thm:main1}, it only remains to prove that they are big; the proof is based on Yoshioka's
description of the Beauville-Bogomolov form on the moduli space in terms of the Mukai pairing
of the K3 surface in \cite{Yoshioka:Abelian}.

The final two sections \ref{sec:K3sheaves} and \ref{sec:Hilbert} are devoted to applications to moduli
spaces of sheaves and to the Hilbert scheme, respectively.

\subsection*{Acknowledgements}
The authors would like to thank in particular Aaron Bertram and Daniel Huybrechts for many
insightful discussions related to this article; the first author would also like to thank
Alastair Craw for very useful discussions on our main construction of nef divisors in different
context. 
We are also grateful to Izzet Coskun, Alina Marian, Eyal Markman, Dragos Oprea,
Paolo Stellari, and Jenia Tevelev for comments and discussions, to Eyal Markman and K\=ota
Yoshioka for very useful comments on an earlier version of this article, and to the referee for very helpful comments which greatly improved the exposition.

This project got started while the first author was visiting the programme on moduli spaces at the
Isaac Newton Institute in Cambridge, England, and he would like to thank the institute for its warm
hospitality and stimulating environment. The collaboration continued during a visit of both authors 
to the Hausdorff Center for Mathematics, Bonn, and we would like to thank the HCM for its support.
A.~B.~ is partially supported by NSF grant DMS-1101377.
E.~M.~ is partially supported by NSF grant DMS-1001482/DMS-1160466, Hausdorff Center for Mathematics, Bonn, and by SFB/TR 45.

\subsection*{Notation and Convention} \label{subsec:notation}

For an abelian group $G$ and a field $k(=\Q,\R,\C)$, we denote by $G_k$ the $k$-vector space $G\otimes k$.

Throughout the paper, $X$ will be a smooth projective variety over the complex numbers.
For a (locally-noetherian) scheme (or algebraic space) $S$, we will use the notation $\Db(S)$ for its bounded derived category of coherent sheaves, $\mathrm{D}_{qc}(S)$ for the unbounded derived category of quasi-coherent sheaves, and $\mathrm{D}_{S\text{-}\mathrm{perf}}(S\times X)$ for the category of $S$-\emph{perfect complexes}.
(An $S$-perfect complex is a complex of $\OO_{S\times X}$-modules which locally, over $S$, is quasi-isomorphic to a bounded complex of coherent shaves which are flat over $S$.)

We will abuse notation and denote all derived functors as if they were underived.
We denote by $p_S$ and $p_X$ the two projections from $S\times X$ to $S$ and $X$, respectively.
Given $\EE\in\mathrm{D}_{qc}(S\times X)$, we denote the Fourier-Mukai functor associated to $\EE$ by
\[
\Phi_\EE (\blank):= (p_X)_*\left( \EE\otimes p_S^*(\blank)\right).
\]

We let $K_{\num}(X)$ be the numerical Grothendieck group of $X$ and denote by $\chi(-)$ (resp., $\chi(-,-)$) the Euler characteristic on $K_{\num}(X)$: for $E,F\in\Db(X)$,
\[ \chi(E)=\sum_p (-1)^p\, h^p(X,E) \quad \text{and} \quad
\chi(E,F)=\sum_p (-1)^p\, \mathrm{ext}^p(E,F).
\]

We denote by $\mathrm{NS}(X)$ the N\'eron-Severi group of $X$, and write $N^1(X):=\mathrm{NS}(X)_\R$.
The space of full numerical stability conditions on $\Db(X)$ will be denoted by $\Stab(X)$.

Given a complex $E\in\Db(X)$, we denote its cohomology sheaves by $\HH^*(E)$.
The skyscraper sheaf at a point $x\in X$ is denoted by $k(x)$.
For a complex number $z\in\mathbb{C}$, we denote its real and imaginary part by $\Re z$ and $\Im z$, respectively.

For a K3 surface $X$, we denote the Mukai vector of an object $E \in \Db(X)$ by
$v(E)$. We will often write it as $v(E) = (r, c, s)$, where
$r$ is the rank of $E$, $c \in \NS(X)$, and $s$ the degree of $v(E)$. For a spherical object
$S$ we denote the spherical twist at $S$ by $\ST_S(\blank)$, defined in
\cite{Seidel-Thomas:braid} by the exact triangle, for all $E\in\Db(X)$,
\[
\Hom^\bullet(S, E) \otimes S \to E \to \ST_S(E).
\]

\section{Review: Bridgeland stability conditions}\label{sec:Bridgeland}

In this section, we give a brief review of stability conditions on derived categories,
as introduced in \cite{Bridgeland:Stab}.

Let $X$ be a smooth projective variety, and denote by $\Db(X)$ its bounded derived category of coherent sheaves.
A \emph{full numerical stability condition} $\sigma$ on $\Db(X)$ consists of a pair $(Z,\AA)$, where
$Z\colon K_{\num}(X)\to\C$ 
is a group homomorphism (called \emph{central charge}) and $\AA\subset\Db(X)$ 
is the \emph{heart of a bounded t-structure}, satisfying the following three properties:
\begin{enumerate}
\item For any $0 \neq E\in\AA$ the central charge $Z(E)$ lies in the following semi-closed
upper half-plane:
\begin{equation} \label{eq:Zpositivity}
Z(E) \in \H := \HH \cup \R_{<0} = \R_{>0} \cdot e^{(0,1]\cdot i\pi}
\end{equation}
\suspend{enumerate}
This positivity condition is the essential ingredient for our positivity result.
One could think of it as two separate positivity conditions: $\Im Z$ defines a rank function
on the abelian category $\AA$, i.e., a non-negative function $\rk \colon \AA \to \R_{\ge 0}$
that is additive on short exact sequences.
Similarly, $-\Re Z$ defines a degree function $\deg \colon \AA \to \R$, which has the property that
$\rk(E) = 0 \Rightarrow \deg(E) > 0$. 
We can use them to define a notion of slope-stability with respect
to $Z$ on the abelian category $\AA$ via the slope $\mu(E) = \frac{\deg(E)}{\rk(E)}$.
\resume{enumerate}
\item With this notion of slope-stability, every object in $E \in \AA$ has a Harder-Narasimhan
filtration $0 = E_0 \into E_1 \into \dots \into E_n = E$ such that the $E_i/E_{i-1}$'s are $Z$-semistable,
with $\mu(E_1/E_0) > \mu(E_2/E_1) > \dots > \mu(E_n/E_{n-1})$.
\item There is a constant $C>0$ such that, for all $Z$-semistable object $E\in \AA$, we have
\begin{align*}
\lVert E \rVert \le C \lvert Z(E) \rvert,
\end{align*}
where $\lVert \ast \rVert$ is a fixed norm on $K_{\num}(X)\otimes\R$.
\end{enumerate}
The last condition was called the \emph{support property} in \cite{Kontsevich-Soibelman:stability},
and is equivalent (see \cite[Proposition B.4]{localP2})
to Bridgeland's notion of a \emph{full} stability condition.

\begin{Def} \label{def:algebraic}
A stability condition is called \emph{algebraic} if its central charge takes values in $\Q\oplus\Q
\sqrt{-1}$.
\end{Def}
As $K_{\num}(X)$ is finitely generated, for an algebraic stability condition the 
image of $Z$ is a discrete lattice in $\C$.

Given $(Z, \AA)$ as above, one can extend the notion of stability to $\Db(X)$ as follows:
for $\phi \in (0, 1]$,
we let $\PP(\phi) \subset \AA$ be the full subcategory $Z$-semistable objects
with $Z(E) \in \R_{>0} e^{i\phi\pi}$; for general $\phi$, it is defined
by the compatibility $\PP(\phi + n) = \PP(\phi)[n]$.
Each subcategory $\PP(\phi)$ is extension-closed and abelian. Its nonzero objects are called
$\sigma$-\emph{semistable} of phase $\phi$, and its simple objects are called
$\sigma$-\emph{stable}. Then each object $E \in \Db(X)$ has a 
\emph{Harder-Narasimhan filtration}, where the inclusions $E_{i-1} \subset E_i$ are replaced by
exact triangles $E_{i-1} \to E_i \to A_i$, and where the $A_i$'s are $\sigma$-semistable of decreasing phases
$\phi_i$.  The
category $\PP(\phi)$ necessarily has finite length. Hence every object in $\PP(\phi)$ has a finite
Jordan-H\"older filtration, whose filtration quotients are $\sigma$-stable objects of the phase
$\phi$.  Two objects $A,B\in\PP(\phi)$ are called $S$-\emph{equivalent} if their Jordan-H\"older
factors are the same (up to reordering).

The set of stability conditions will be denoted by $\Stab(X)$.
It has a natural metric topology (see \cite[Prop.\ 8.1]{Bridgeland:Stab} for the explicit form of
the metric). Bridgeland's main theorem is the following:

\begin{Thm}[Bridgeland] \label{thm:Bridgeland-deform}
The map
\begin{align*}
\ZZ \colon \Stab(X) \to \Hom(K_{\num}(X), \C), \qquad (Z, \AA)\mapsto Z, 
\end{align*}
is a local homeomorphism.
In particular, $\Stab(X)$ is a complex manifold of finite dimension equal to the rank of $K_{\num}(X)$.
\end{Thm}
In other words, a stability condition $(Z, \AA)$ can be deformed uniquely given a small deformation
of its central charge $Z$.

Let us now fix a class $v \in K_{\num}(X)$, and consider the set of
$\sigma$-semistable objects $E \in \Db(X)$ of class $v$ as $\sigma$ varies. The proof of the
following statement is essentially contained in \cite[Section 9]{Bridgeland:K3}; see
also \cite[Proposition 3.3]{localP2} and \cite[Prop 2.8]{Toda:K3}:

\begin{Prop} \label{prop:chambers}
There exists a locally finite set of \emph{walls} (real codimension one submanifolds with boundary) 
in $\Stab(X)$, depending only on $v$, with the following properties: 
\begin{enumerate}
\item When $\sigma$ varies within a chamber, the sets of $\sigma$-semistable and
$\sigma$-stable objects of class $v$ does not change.
\item When $\sigma$ lies on a single wall $W \subset \Stab(X)$,
then there is a $\sigma$-semistable object
that is unstable in one of the adjacent chambers, and semistable in the other adjacent chamber.
\item When we restrict to an intersection of finitely many walls $W_1, \dots, W_k$, we
obtain a wall-and-chamber decomposition on $W_1 \cap \dots \cap W_k$ with the same properties,
where the walls are given by the intersections $W \cap W_1 \cap \dots \cap W_k$ for any
of the walls $W \subset \Stab(X)$ with respect to $v$.
\end{enumerate}
\end{Prop}
If $v$ is primitive, then $\sigma$ lies on a wall if and only if there exists a
strictly $\sigma$-semistable object of class $v$. 
The Jordan-H\"older filtration of $\sigma$-semistable objects
does not change when $\sigma$ varies within a chamber.

\begin{Def}\label{def:generic}
Let $v\in K_{\num}(X)$.
A stability condition is called \emph{generic} with respect to $v$ if it does not lie on a wall
in the sense of Proposition \ref{prop:chambers}.
\end{Def}

We will also need the following variant of \cite[Lemma 2.9]{Toda:K3}:
\begin{Lem} \label{lem:Yukinobualgebraic}
Consider a stability condition $\sigma = (Z, \AA)$ with $Z(v) = -1$. Then there are algebraic
stability conditions $\sigma_i = (Z_i, \AA_i)$ for $i = 1, \dots, m$ nearby $\sigma$
with $Z_i(v) = -1$ such that:
\begin{enumerate}
\item For every $i$ the following statement holds: an object of class $v$ is $\sigma_i$-stable (or
$\sigma_i$-semistable) if and only if it is $\sigma$-stable (or $\sigma$-semistable, respectively).
\item The central charge $Z$ is in the convex hull of $\{Z_1, \dots, Z_n\}$.
\end{enumerate}
\end{Lem}
\begin{Prf}
If $v$ is generic, this follows immediately from Theorem \ref{thm:Bridgeland-deform} and Proposition
\ref{prop:chambers}, and the density of algebraic central charges
$\Hom(K_{\num}(X), \Q \oplus i\Q)$ inside the vector space $\Hom(K_{\num}(X), \C)$.
Once we restrict to the subset $Z(v) = -1$, any wall is locally defined by a
linear rational equation of the form $\Im Z(w) = 0$, where $w \in K_{\num}(X)$ is the class of a destabilizing
subobject, and thus the claim follows similarly.
\end{Prf}

\begin{Rem}\label{rmk:GroupAction}
There are two group actions on $\Stab(X)$, see \cite[Lemma 8.2]{Bridgeland:Stab}:
the group of autoequivalences $\Aut(\Db(X))$ acts on the left via
$\Phi(Z,\AA)=(Z\circ\Phi_*^{-1},\Phi(\AA))$, where $\Phi \in \Aut(\Db(X))$ and
$\Phi_*$ is the automorphism induced by $\Phi$ at the level of numerical Grothendieck groups.
We will often abuse notation and denote $\Phi_*$ by $\Phi$, when no confusion arises.
The universal cover $\widetilde\GL^+_2(\R)$ of the group 
$\GL^+_2(\R)$ of matrices with positive determinant acts on the right as a lift of the
action of $\GL^+_2(\R)$ on
$\Hom(K_{\num}(X), \C) \cong \Hom(K_{\num}(X), \R^2)$. We typically only use the action of the
subgroup $\C \subset \widetilde\GL^+_2(\R)$ given as the universal cover of $\C^* \subset \GL^+_2(\R)$: given
$z \in \C$, it acts on $(Z, \AA)$ by $Z \mapsto e^{2\pi i z}\cdot Z$, and by modifying
$\AA$ accordingly.
\end{Rem}

\section{Positivity}\label{sec:MainConstr}

In this section we prove our main result, Positivity Lemma \ref{Positivity}.

We consider any smooth projective variety $X$ with a numerical stability condition $\sigma=(Z,\AA)$
on $\Db(X)$. Let us first recall the definition of flat families, due to Bridgeland:
\begin{Def}\label{def:flatness}
Let $\AA\subset\Db(X)$ be the heart of a bounded t-structure on $\Db(X)$.
Let $S$ be an algebraic space of finite type over $\C$, and let $\EE\in\mathrm{D}_{S\text{-}\mathrm{perf}}(S\times X)$.
We say that $\EE$ is \emph{flat} with respect to $\AA$ if, for every closed point $s\in S$, the derived restriction $\EE_s$ belongs $\AA$.
\end{Def}

Let $v \in K_{\num}(X)$; we use the action of $\C$ on $\Stab(X)$ described in Remark
\ref{rmk:GroupAction} to assume $Z(v)=-1$.
We denote by $\MMM_{\sigma}(v)$ be the moduli stack of flat
families of $\sigma$-semistable objects of class $v$ and phase $1$.  Our construction in this
section gives a version of Theorem \ref{thm:main1} for the stack $\MMM_{\sigma}(v)$; we will discuss how it
extends to the coarse moduli space (when it exists) in Section \ref{sec:DivisorClassModuliSpace}.

\begin{PropDef} \label{def:basic}
Let $C \to \MMM_{\sigma}(v)$ be an integral projective curve over $\MMM_{\sigma}(v)$, with induced universal family
$\EE \in \Db(C \times X)$, and associated Fourier-Mukai transform
$\Phi_{\EE} \colon \Db(C) \to \Db(X)$. To such a $C$ we associate
a number $\LLL_\sigma.C \in \R$ by
\begin{equation} \label{eq:DEF}
\LLL_\sigma.C := \Im Z(\Phi_\EE(\OO_C)).
\end{equation}
This has the following properties:
\begin{enumerate}
\item \label{enum:family}
Modifying the universal family by tensoring with a pull-back of a line bundle
on $C$ does not modify $\LLL_\sigma.C$.
\item \label{enum:twist}
We can replace $\OO_C$ in equation \eqref{eq:DEF} by any line bundle on $C$, without 
changing $\LLL_\sigma.C$.
\end{enumerate}
\end{PropDef}

We will think of $\LLL_\sigma$ as a divisor class in $N^1(\MMM_{\sigma}(v))$.

\begin{Prf}
If $c \in C$, then replacing the universal family by $\EE' = \EE \otimes p^*\OO_C(c)$ effects
the Fourier-Mukai transform by
\[ [\Phi_{\EE'}(\OO_C)] = [\Phi_\EE(\OO_C)] + [\Phi_\EE(k(c))] = [\Phi_\EE(\OO_C)] + v.
\]
As $\Im Z(v) = 0$, this proves claim \eqref{enum:family}, and similarly \eqref{enum:twist}.
\end{Prf}

\begin{PosLem} \label{Positivity}
The divisor class $\LLL_\sigma$ is nef: $\LLL_\sigma.C \ge 0$.
Further, we have $\LLL_\sigma.C > 0$ if and only
if for two general closed points $c, c' \in C$, the corresponding objects $\EE_c, \EE_{c'}\in\Db(X)$ are not
$S$-equivalent.
\end{PosLem}
We first point out that by Lemma \ref{lem:Yukinobualgebraic}, we can immediately restrict to the
case where $\sigma$ is an algebraic stability condition. This implies that the heart
$\AA$ is Noetherian, by \cite[Proposition 5.0.1]{Abramovich-Polishchuk:t-structures}.

The essential ingredient in the proof is the construction and description by Abramovich and Polishchuk of a constant family of t-structures on $S \times X$ induced by $\AA$, for smooth $S$ given in \cite{Abramovich-Polishchuk:t-structures} and extended to singular $S$ in \cite{Polishchuk:families-of-t-structures}.
For any scheme $S$ of finite type over $\C$, we denote by $\AA_S$ the heart of the ``constant t-structure'' on $\Db(S \times X)$ given by \cite[Theorem 3.3.6]{Polishchuk:families-of-t-structures}.
The heart $\AA_S$ could be thought of as $\Coh S \boxtimes \AA$, since it behaves like $\AA$ with respect to $X$, and like $\Coh S$ with respect to $S$.
For example, whenever $F \in \Coh S$ and $E \in \AA$, we have $F \boxtimes E \in \AA_S$; also, $\AA_S$ is invariant under tensoring with line bundles pulled back from $S$.
It is characterized by the following statements (which paraphrase \cite[Theorem 3.3.6]{Polishchuk:families-of-t-structures}):

\begin{Thm} \label{thm:APP}
Let $\AA$ be the heart of a Noetherian bounded t-structure on $\Db(X)$. 
Denote by $\AA^{qc} \subset \mathrm{D}_{qc}(X)$ the closure of $\AA$ under infinite coproducts 
in the derived category of quasi-coherent sheaves.
\begin{enumerate}
\item \label{enum:deftstruct}
For any scheme $S$ of finite type of $\C$ there is a Noetherian bounded t-structure on $\Db(S\times X)$, whose
heart $\AA_S$ is characterized by 
\[ \EE \in \AA_S \Leftrightarrow p_* \EE |_{X \times U} \in \AA^{qc} \quad \text{for every open affine
$U \subset S$}
\]
\item \label{enum:sheaf}
The above construction defines a sheaf of t-structures over $S$: when $S = \bigcup_i U_i$
is an open covering of $S$, then $\EE \in \AA_S$ if and only if
$\EE |_{X \times U_i} \in \AA_{U_i}$ for every $i$.
\item \label{enum:Sprojective}
When $S$ is projective and $\OO_S(1)$ denotes an ample divisor, then
\[ \EE \in \AA_S \Leftrightarrow (p_X)_* (\EE \otimes p_S^* \OO_S(n)) \in \AA \quad \text{for 
all $n \gg 0$}. \]
\end{enumerate}
\end{Thm}

The following lemma is essentially \cite[Proposition 2.3.7]{Polishchuk:families-of-t-structures} (see also \cite[Corollary 3.3.3]{Abramovich-Polishchuk:t-structures} for the smooth case):

\begin{Lem}
Let $\EE \in \Db(S \times X)$ be a flat family of objects in $\AA$.
Then $\EE \in \AA_S$.
\end{Lem}

\begin{Prf}
We first claim the statement when $S$ is a zero-dimensional scheme of finite length $l>0$,
with a unique closed point $s \in S$.
Choose a filtration
\[
0=F_0\subset F_1\subset\ldots \subset F_l=\OO_S,
\]
of the structure sheaf in $\Coh S$ with $F_i/F_{i-1}\cong k(s)$ for all $i$.
After pull-back to $S \times X$ and tensoring with $\EE$ 
we get a sequence of morphisms in $\Db(S\times X)$
\[
0=G_0\to G_1\to \ldots \to G_l=\EE,
\]
such that $\mathrm{cone}(G_{i-1}\to G_i)\cong \EE_s$ for all $i$.
By induction on $i$ we obtain $(p_X)_*G_i\in\AA$; then part \eqref{enum:Sprojective} of Theorem
\ref{thm:APP} implies $\EE\in\AA_S$.

For general $S$, any closed point $s \in S$ is contained in a local zero-dimensional subscheme $T \subset S$
that is a local complete intersection in $S$.
The previous case shows $\EE_T \in \AA_T$, and by \cite[Proposition 2.3.7]{Polishchuk:families-of-t-structures} we can cover $S$ by open sets $U$ with $\EE_U \in \AA_U$.
By the sheaf property \eqref{enum:sheaf}, this shows $\EE \in \AA_S$.
\end{Prf}

\begin{Lem} \label{lem:OCninAA}
Let $C$ be an integral projective curve, and $\EE \in D^b(C \times X)$ be a family of $\sigma$-semistable
objects in $\PP(1)$. Then, there exists $n_0>0$ such that
\[
\Phi_\EE \left(L \right) \in \AA,
\]
for all line bundles $L$ on $C$ with degree $\deg(L)\geq n_0$.
\end{Lem}

\begin{Prf}
By the previous lemma, we have $\EE \in \AA_C$.
Fix an ample divisor $\OO_C(1)$ on $C$.
By the statement \eqref{enum:Sprojective} of Theorem \ref{thm:APP}, there exists $m_0>0$ such that
\[
\Phi_\EE \left(\OO_C(n)\right) = (p_X)_* (\EE \otimes p_C^* \OO_C(n))\in\AA,
\]
for all $n\geq m_0$.
Fix $n_0>0$ such that, for a line bundle $L$ with $\deg(L)\geq n_0$, we have $H^0(C,L(-m_0))\neq0$.
Then, consider the exact sequence
\[
0\to \OO(m_0)\to L \to T \to 0,
\]
where $T$ has zero-dimensional support.
By applying $\Phi_\EE$ we get our claim.
\end{Prf}

Lemma \ref{lem:OCninAA} directly implies the first claim of Positivity Lemma \ref{Positivity}:
For a curve $C \to \MMM_{\sigma}(v)$ with universal family $\EE \in \Db(X \times C)$, we have
\[
\LLL_\sigma.C = \Im Z \left(\Phi_\EE(\OO_C)\right) = \Im Z \left(\Phi_\EE (\OO_C(n))\right) \ge 0,
\]
by the basic positivity property of the central charge $Z$ in equation
\eqref{eq:Zpositivity}.

It remains to prove the second claim.

\begin{Lem} \label{lem:finiteJH} %(Jordan-H\"older filtration in families.) 
Let $\EE \in \Db(S \times X)$ be a flat family of semistable objects in $\AA$ over an irreducible
scheme $S$ of finite type over $\C$. Assume that the union of all Jordan-H\"older factors of
$\EE_s$ over all closed points $s \in S$ is finite. Then all the objects $\EE_s$ are $S$-equivalent
to each other, and we can choose a Jordan-H\"older filtration for every $\EE_s$ such that the order
of their stable filtration quotients does not depend on $s$.
%Then there is filtration 
%$\EE^0 \subset \EE^1 \subset \dots \subset \EE^n = \EE$ in $\AA_S$ that by restriction
%induces the Jordan-H\"older filtration of $\EE_s$ for every closed point $s \in S$.
\end{Lem}

\begin{Prf}
If we choose a Jordan-H\"older filtration of $\EE_s$ for every closed point $s$, then there will
be a stable object $F \in \AA$ that appears as the final quotient $\EE_s \onto F$ of the filtration
for infinitely many $s \in S$. In particular, $\Hom(\EE_s, F)$ is non-zero for infinitely many
$s \in S$; by semi-continuity, this implies that for \emph{every} $s \in S$ there is a (necessarily
surjective) morphism $\EE_s \onto F$ in $\AA$. The same argument applied to the kernel
of $\EE_s \onto F$ implies the claim by induction on the length of $\EE_{s_0}$ for a fixed chosen
point $s_0 \in S$.
\end{Prf}

\begin{Lem} \label{lem:simple}
Let $F \in \AA$ be a simple object. Then any subobject of $p_X^* F$ in
$\AA_S$ is of the form $I \boxtimes F$ for some ideal sheaf $I \subset \OO_S$ on $S$.
\end{Lem}
\begin{Prf}
By the sheaf property of $\AA_S$, it is sufficient to treat the case where
$S$ is affine. By the characterization \eqref{enum:deftstruct} in Theorem
\ref{thm:APP} of $\AA_S$, a subobject $G \subset p_X^* F$ in $\AA_S$ gives a subobject
$(p_X)_* G$ of $(p_X)_* p_X^* F = \OO_S \otimes_\C F$ in $\AA^{qc}$ that is compatible
with the $\OO_S$-module structure (see also \cite[Proposition 3.3.7]{Polishchuk:families-of-t-structures}).
Since $F$ is simple, such a subobject must be of
the form $I \otimes_\C F$ for some ideal $I \subset \OO_S$.
\end{Prf}

\begin{Lem} \label{lem:JH-in-families}
Let $\EE \in \Db(C \times X)$ be a family of semistable objects over an integral
curve $C$. Assume that for general $c, c' \in C$, the objects
$\EE_c, \EE_{c'}$ are $S$-equivalent to each other.
Then there exist line bundles $\LL_1, \dots, \LL_n $ 
on $C$ and a filtration
\[
0=\Gamma_0\subset\Gamma_1\subset\ldots\subset\Gamma_n=\EE
\]
in $\AA_S$ such that, for all $i=1,\ldots,n$,
\[
\Gamma_i/\Gamma_{i-1} \cong F_i\boxtimes\LL_i
\]
and such the restrictions of the $\Gamma_i$ to the fibers $\{c\} \times X$ induces the
Jordan-H\"older filtration of $\EE_c$ for all but finitely many $c \in C$.
\end{Lem}

\begin{Prf}
The same arguments as in the previous lemma show that any two $\EE_c, \EE_{c'}$ are
$S$-equivalent to each other, and that there is a common stable subobject
$F:= F_1 \subset \EE_s$ for all $s$. 

We first claim that there exists a line bundle $\LL_1$ on $C$
with a non-zero morphism
$\phi \colon F \boxtimes \LL_1 \to \EE$ on $C \times X$; equivalently, we need to show that for 
$\LL := \LL_1^*$, we have
\[
0\neq \Hom_X\left(F,(p_X)_*(\EE\otimes p_C^*\LL)\right)
= \Hom_{X}\left(F,\Phi_{\EE}(\LL)\right).
\]

Let $n_0$ be as in Lemma \ref{lem:OCninAA}, and fix a line bundle $\LL_0$ on $C$ of degree $n_0$.
Set $r := \dim \Ext^1(F, \Phi_{\EE}(\LL_0))$. Pick $r+1$ distinct smooth points
$c_1, \dots, c_{r+1} \in C$, and set $\LL:= \LL_0(c_1 + \dots + c_{r+1})$.
Consider the short exact sequence
\[ 
0 \to \LL_0 \to \LL \to \OO_{c_1} \oplus \dots \oplus \OO_{c_{r+1}} \to 0
\]
in $\Coh C$.  By Lemma \ref{lem:OCninAA}, it induces a short sequence
\[
0 \to \Phi_\EE(\LL_0) \to \Phi_\EE(\LL) \to
\EE_{c_1} \oplus \dots \oplus \EE_{c_{r+1}} \to 0
\]
in $\AA$. Since
$\dim \Hom(F, \EE_{c_1} \oplus \dots \oplus \EE_{c_{r+1}}) \ge r+1
> r = \dim \Ext^1(F, \Phi_\EE(\LL_0))$, there exists a non-zero morphism from $F$ to the direct sum that
factors via $\Phi_\EE(\LL)$, which proves the existence of the morphism $\phi$ as claimed.

It follows from \cite[Proposition 2.2.3]{Lieblich:mother-of-all} that the
restriction $\phi_c \colon F_1 \to \EE_c$ is non-zero for all but finitely many closed
points $c \in C$.
Since $F_1$ is stable, this morphism is necessarily injective. To proceed by induction,
it remains to show that $\phi$ is an injective morphism in $\AA_C$. Otherwise, by Lemma \ref{lem:simple}
the kernel of $\phi$ is of the form $I \otimes \LL_1 \boxtimes F_1$ for some ideal sheaf $I \subset \OO_C$.
In particular, the derived restriction
of $\ker \phi \into \LL_1 \boxtimes F_1$ to $\{c\} \times X$ is an isomorphism for all
but finitely many $c$, in contradiction to the injectivity of $\phi_c$.
\end{Prf}

\begin{Prf} (Positivity Lemma \ref{Positivity})
Let $C$ be an integral projective curve, and let $\EE\in\MMM_{\sigma}(v)(C)$.
As we observed before, we only need to prove the second claim.

``$\Leftarrow$'': Assume that $\LLL_\sigma.C = 0$.
We will show that all objects $\EE_c$, for smooth points $c \in C$, are $S$-equivalent to each other.

By Lemma \ref{lem:OCninAA}, for a line bundle $L$ of large degree on $C$, and for $c\in C$ a smooth point, the short exact sequence
\[
0\to L(-c) \to L \to k(c) \to 0
\]
induces a short exact sequence in $\AA$
\[
0\to \Phi_\EE(L(-c)) \to \Phi_\EE(L) \to \EE_c \to 0.
\]
Since $\LLL_\sigma. C = 0$, we have $Z(\Phi_\EE(L)) \in \R_{<0}$, and so $\Phi_\EE(L) \in \PP(1)$ is semistable, of the same phase as $\EE_c$.
It follows that the Jordan-H\"older factors of $\EE_c$ are a subset of the Jordan-H\"older factors
of $\Phi_\EE(L)$, which of course do not depend on $c$.
Lemma \ref{lem:finiteJH} implies the claim.

``$\Rightarrow$'': Assume that, for general $c \in C$, all objects $\EE_c$ are $S$-equivalent to
each other. Using Lemma \ref{lem:JH-in-families} and
the projection formula, we obtain:
\begin{align*}
\LLL_{\sigma}.C
&= \Im Z\left([\Phi_\EE(\OO_C)] \right)
= \sum_{i=1}^n \Im Z \left( [(p_X)_* F_i \boxtimes \LL_i]\right) \\
&= \sum_{i=1}^n \Im Z \left( [F_i \otimes H^\bullet(C, \LL_i)]\right)
= \sum_{i=1}^n \chi(C, \LL_i) \cdot \Im Z \left( [F_i]\right)
= 0
\end{align*}
\end{Prf}

\section{A natural nef divisor class on the moduli space, and comparison}\label{sec:DivisorClassModuliSpace}

Let $S$ be a proper algebraic space of finite type over $\C$, let $\sigma = (Z, \AA) \in\Stab(X)$, and let
$\EE\in\mathrm{D}_{S\text{-}\mathrm{perf}}(S\times X)$ be a flat family of semistable objects of
class $v$ with $Z(v) = -1$. Note that $S$ is allowed to have arbitrary singularities.

By restriction of the family, our construction in the previous section assigns a number
$\LLL_\sigma.C$ to every curve $C \subset S$. Our first goal is to prove that this induces
a nef divisor class on $S$, as claimed in Theorem \ref{thm:main1}; the following theorem gives a more
complete statement:

\begin{Thm} \label{thm:nefdivisor}
The assignment $C \mapsto \LLL_\sigma.C$ only depends on the numerical curve class $[C] \in N_1(S)$,
and is additive on curve classes. It defines a nef divisor class
$\ell_{\sigma,\EE} \in N^1(S)$, which is invariant under tensoring the family $\EE$ with
a line bundle pulled back from $S$.
Additionally, for a curve $C\subseteq S$ we have $\ell_{\sigma,\EE}.C > 0$ if and only
if for two general closed points $c, c' \in C$, the corresponding objects $\EE_c, \EE_{c'}\in\Db(X)$ are not
$S$-equivalent.
\end{Thm}

Before the proof, let us recall that the real N\'eron-Severi group $N^1(S)$ is defined as the group of real Cartier
divisors modulo numerical equivalence; dually, $N_1(S)$ is the group of real 1-cycles modulo
numerical equivalence with respect to pairing with Cartier divisors (see, for example,
\cite[Sections 1.3 and 1.4]{Laz:Positivity1}). 

Similarly, the Euler characteristic gives a well-defined pairing
\[
\chi \colon K(\Db(S)) \times K(\Db_{\mathrm{perf}}(S)) \to \Z
\]
between the K-groups of 
the bounded derived categories of coherent sheaves $\Db(S)$ and of perfect complexes
$\Db_{\mathrm{perf}}(S)$. Taking the quotient with respect to the kernel of
$\chi$ on each side we obtain numerical Grothendieck groups
$K_{\num}(S)$ and $K_{\num}^{\mathrm{perf}}(S)$, respectively, with an induced
perfect pairing
\[
\chi \colon K_{\num}(S) \times K_{\num}^{\mathrm{perf}}(S) \to \Z.
\]

\begin{proof}[Proof of Theorem \ref{thm:nefdivisor}]
If $C_1, C_2$ are numerically equivalent, then $[\OO_{C_1}], [\OO_{C_2}] \in K_{\num}(S)$
only differ by multiples of the class of a skyscraper sheaf $k(s)$ of a closed point $s$.
Since $\EE$ is $S$-perfect, it induces a functor
$\Phi_\EE \colon \Db(S) \to \Db(X)$, and this functor preserves numerical equivalence. Together with
$\Im Z(\Phi_\EE(k(s))) = \Im Z(v) = 0$,
this proves the first claim. The additivity follows similarly, and all other claims
follow directly from the Positivity Lemma \ref{Positivity}.
\end{proof}

\begin{Ex}\label{ex:quintic}
Let $X\subset\P^4$ be a smooth quintic threefold, containing two disjoint lines $L_1, L_2$.
Consider the smooth proper algebraic space $X^+$ obtained as the flop of $X$ at the line
$L_1$.
Then, by a classical result of Bondal and Orlov \cite{BondalOrlov:Main}, we have an equivalence of
derived categories $\Db(X)\cong\Db(X^+)$. However, $X^+$ admits no numerically positive class,
as the flopped curve $\widetilde L_1$ is the negative of $L_2$ in $N_1(X^+)$.
By Theorem \ref{thm:main1}, $X^+$ cannot be isomorphic to a moduli space of (semi)stable complexes
on $\Db(X)$ with respect to a numerical stability condition on $X$.
\end{Ex}

There are many examples of non-projective flops of a projective variety, including Mukai flops
of holomorphic symplectic varieties. By the same reasoning, these non-projective flops cannot be
obtained by wall-crossing.

We will now compare our construction to the classical notion of a determinant divisor on $S$ associated to the family $\EE$ (see, e.g., \cite{Mukai:BundlesK3, Donaldson:Invariants,LePotier:Determinant,Faltings:Higgs, JunLi:Picard} and \cite[Section 8.1]{HL:Moduli}).

\begin{Def}\label{def:determinant}
We define a group homomorphism (called the \emph{Donaldson morphism})\footnote{In
the case of K3 surfaces, we will instead use a dual version, see
Definition \ref{def:MukaiHom} and Remark \ref{rem:DonaldsonMukai}.}
\[
\lambda_{\EE}:v^\sharp\to N^1(S)
\]
as the composition
\[
v^\sharp\xrightarrow{p_X^*}K_{\num}^{\perf}(S\times
X)_\R\xrightarrow{\cdot[\EE]}K_{\num}^{\perf}(S\times X)_\R\xrightarrow{(p_S)_*}K_{\num}^{\perf}(S)_\R\xrightarrow{\det}N^1(S),
\]
where
\[
v^\sharp:=\left\{w\in K_{\num}(X)_\R \colon \chi(v\cdot w)=0 \right\}.
\]
\end{Def}

Since the Euler characteristic $\chi$ gives a non-degenerate pairing, we can write
\[
\Im(Z(\blank))=\chi(w_Z\cdot \blank),
\]
for a unique vector $w_Z\in v^\sharp$.

\begin{Prop}\label{prop:comparison}
For an integral curve $C\subset S$, we have
\begin{equation}\label{eq:comparison}
\lambda_{\EE}(w_Z).C = \Im Z(\Phi_{\EE}(\OO_C)) =: \ell_{\sigma,\EE}.C.
\end{equation}
\end{Prop}

\begin{Prf}
It is enough to prove \eqref{eq:comparison} when $w_Z=[F]\in K_{\num}(X)$, for some $F\in\Db(X)$.
We define
\[
\LL(F):=(p_{S})_*\left(p_X^*F\otimes\EE\right).
\]
As in the classical case, the assumption $w_Z\in v^\sharp$ and the projection formula show
that the rank of $\LL(F)$ must be equal to $0$:
\[
\rk \LL(F)= \chi(S, \OO_s \otimes \LL(F)) = \chi(S \times X, \OO_{\{s\} \times X} \otimes \EE \otimes p_X^* F) 
= \chi(\EE_s \otimes F) = \chi(v \cdot w_Z) = 0.
\]
The Riemann-Roch Theorem gives
\[
\lambda_{\EE}(w_Z).C = \mathrm{deg}\, \LL(F)|_C = \chi(C,\LL(F)|_C).
\]
By Cohomology and Base Change (see, e.g., \cite[Corollary 2.23]{Kuznetsov:Hyperplane}) we deduce that
\[
\LL(F)|_C = (p_C)_*\left( p_X^*F\otimes\EE |_C\right).
\]
Using the projection formula again gives
\[
\chi(C,\LL(F)|_C)=\chi(S,F\otimes(p_X)_*\EE |_C)=\chi(w_Z\cdot\Phi_{\EE}(\OO_C))=\Im Z(\Phi_{\EE}(\OO_C)).
\]
\end{Prf}

The basic properties of $\ell_{\sigma,\EE}$ are in \cite[Lemma 8.1.2]{HL:Moduli}.
In particular, we recall the following: Let $\NN$ be a vector bundle on $S$ of rank $n$.
Then
\begin{equation}\label{eq:VectorBundle}
\ell_{\sigma,\EE\otimes p_S^*\NN}=n\cdot \ell_{\sigma,\EE}.
\end{equation}

By Theorem \ref{thm:nefdivisor}, we have a well-defined positive divisor class on a fine moduli space of stable complexes.
To extend this to the case when a universal family exists only \'etale locally on a coarse moduli
space, we recall the following definition from \cite{Mukai:BundlesK3}.

\begin{Def}\label{def:quasi-univ}
Let $T$ be an algebraic space of finite-type over $\C$.
\begin{enumerate}
\item A flat family $\EE$ on $T\times X$ is called a \emph{quasi-family} of objects in
$\MMM_{\sigma}(v)$ if, for all closed points $t\in T$, there exist an integer $\rho>0$ and an
element $E\in\MMM_{\sigma}(v)(\C)$ such that $\EE |_{\{t\}\times X} \cong E^{\oplus \rho}$. If $T$ is connected, the positive integer $\rho$ does not depend on $t$ and it is called the \emph{similitude} of $\EE$.
\item Two quasi-families $\EE$ and $\EE'$ on $T \times X$ are called \emph{equivalent} if there exist vector bundles $V$ and $V'$ on $T$ such that $\EE'\otimes p_T^*V \cong \EE \otimes p_T^*V'$.
\item A quasi-family $\EE$ is called \emph{quasi-universal} if, for every scheme $T'$ and for any quasi-family $\TT$ on $T \times X$, there exists a unique morphism $f\colon T'\to T$ such that $f^*\EE$ and $\TT$ are equivalent.
\end{enumerate}
\end{Def}

Let $\EE$ be a quasi-family of objects in $\MMM_{\sigma}(v)$ of similitude $\rho$ over a proper
algebraic space $T$.  We can define a divisor class $\ell_{\sigma}$ on $T$ by $\ell_{\sigma} :=
\frac{1}{\rho} \cdot \ell_{\sigma, \EE}$; by equation \eqref{eq:VectorBundle} it only depends on the
equivalence class of the quasi-universal family.

\begin{Rem} \label{rem:quasidivisor}
By \cite[Theorem A.5]{Mukai:BundlesK3}, if $\MMM_{\sigma}(v)$ consists only of stable (and
therefore, simple) complexes, and
if $\MMM_{\sigma}(v)$ is a
$\mathbb{G}_m$-gerbe over an algebraic space $M_\sigma(v)$ of finite-type over $\C$ (i.e., over its coarse moduli
space), then there exists a quasi-universal family on $M_{\sigma}(v)\times X$, unique up to
equivalence.

Therefore, the above construction produces a well-defined divisor class $\ell_\sigma$ on the coarse
moduli space $M_\sigma(v)$. 
Theorem \ref{thm:nefdivisor} holds similarly for $\ell_\sigma$;
in particular, it has the same positivity properties.
\end{Rem}

\section{Review: Moduli spaces for stable sheaves on K3 surfaces}\label{sec:ReviewGieseker}

In this section we give a summary on stability for sheaves on K3 surfaces.
We start by recalling the basic lattice-theoretical structure, given by the Mukai lattice.
We then review slope and Gieseker stability, the existence and non-emptiness for moduli spaces of semistable sheaves, and the structure of their N\'eron-Severi groups.
Finally, we mention briefly how all of this generalizes to twisted K3 surfaces.

\subsection*{The algebraic Mukai lattice}
Let $X$ be a smooth projective K3 surface.
We denote by $H^*_{\alg}(X,\Z)$ the algebraic part of the whole cohomology of X, namely
\begin{equation}\label{eq:AlgebraicMukaiLattice}
H^*_\alg(X,\Z) = H^0(X,\Z) \oplus \mathrm{NS}(X) \oplus H^4(X,\Z).
\end{equation}

Let $v \colon K_{\num}(X) \xrightarrow{\sim} H^*_{\alg}(X, \Z)$ be the Mukai vector given by $v(E) = \ch(E) \sqrt{\td(X)}$.
We denote the Mukai pairing $H^*_{\alg}(X, \Z) \times H^*_{\alg}(X, \Z) \to \Z$ by $(\blank, \blank)$; it can be
defined by $(v(E), v(F)) := - \chi(E, F)$.
According to the decomposition \eqref{eq:AlgebraicMukaiLattice}, we have
\[
\left( (r,c,s),(r',c',s')\right) = c.c' - rs' - r's,
\]
for $(r,c,s),(r',c',s')\in H^*_\alg(X,\Z)$.
%The \emph{algebraic Mukai lattice} consists of the data $\left( H^*_\alg(X,\Z), (\blank,\blank) \right)$.

Given a Mukai vector $v\in H^*_{\alg}(X, \Z)$, we denote
its orthogonal complement by
\[
v^\perp:=\left\{w\in H^*_{\alg}(X, \Z)\colon (v,w)=0 \right\}.
\]
We call a Mukai vector $v$ \emph{primitive} if it is not divisible in $H^*_\alg(X,\Z)$.

\subsection*{Slope stability}
Let $\omega,\beta\in\NS(X)_\Q$ with $\omega$ ample.
We define a slope function
$\mu_{\omega, \beta}$ on $\Coh X$ by
\begin{equation} \label{eq:muomegabeta}
\mu_{\omega, \beta}(E) = 
\begin{cases}
\frac{\omega.(c_1(E) - \beta)}{r(E)} & \text{if $r(E) > 0$,} \\
+\infty & \text{if $r(E) = 0$.}
\end{cases}
\end{equation}
This gives a notion of slope stability for sheaves, for which Harder-Narasimhan filtrations exist (see \cite[Section 1.6]{HL:Moduli}).
We will sometimes use the notation $\mu_{\omega,\beta}$-stability.

Set-theoretical moduli spaces of torsion-free slope semistable sheaves were constructed by Le Potier and Li (see \cite{LePotier:Determinant,JunLi:Uhlenbeck} and \cite[Section 8.2]{HL:Moduli}).

\subsection*{Gieseker stability}
Let $\omega,\beta\in\NS(X)_\Q$ with $\omega$ ample.
We define the \emph{twisted Hilbert polynomial} by
\[
P(E,m) := \int_X e^{-\beta} . (1,m\omega,\frac{m^2\omega^2}{2}) . v(E),
\]
for $E\in\Coh(X)$.
This gives rise to the notion of $\beta$-twisted $\omega$-Gieseker stability for sheaves, introduced first in \cite{MatsukiWenthworth:TwistedVariation}.
When $\beta=0$, this is nothing but Gieseker stability.
We refer to \cite[Section 1]{HL:Moduli} for basic properties of Gieseker stability.

\subsection*{Moduli spaces of stable sheaves}
Let $\omega,\beta\in\NS(X)_\Q$ with $\omega$ ample.
We fix a Mukai vector $v\in H^*_{\alg}(X,\Z)$.
We denote by $\MMM_{\omega}^{\beta}(v)$ the moduli stack of flat families of $\beta$-twisted $\omega$-Gieseker semistable sheaves with Mukai vector $v$.
By the work of Mumford, Gieseker, Maruyama, and Simpson among others (see \cite[Section 4]{HL:Moduli} and \cite{MatsukiWenthworth:TwistedVariation}), there exists a projective variety $M_{\omega}^{\beta}(v)$ which is a coarse moduli space parameterizing $S$-equivalence classes of semistable sheaves.
The open substack $\MMM_{\omega}^{\beta,s}(v)\subseteq \MMM_{\omega}^{\beta}(v)$ parameterizing stable sheaves is a $\mathbb{G}_m$-gerbe over the open subset $M_{\omega}^{\beta, s}(v)\subseteq M_{\omega}^{\beta}(v)$.
When $\beta=0$, we will denote the corresponding objects by $\MMM_\omega(v)$, etc.

The following is the main result on moduli spaces of stable sheaves on K3 surfaces.
In its final form it is proved by Yoshioka in \cite[Theorems 0.1 \& 8.1]{Yoshioka:Abelian} (see also \cite[Section 2.4]{KLS:SingSymplecticModuliSpaces}, where the condition of positivity is discussed more in detail);
it builds on previous work by Mukai and Kuleshov, among others.
We start by recalling the notion of positive vector, following \cite[Definition 0.1]{Yoshioka:Abelian}.

\begin{Def}\label{def:YoshiokaPositive}
Let $v_0=(r,c,s)\in H^*_{\alg}(X, \Z)$ be a primitive class.
We say that $v_0$ is \emph{positive} if $v_0^2\geq-2$ and
\begin{itemize}
\item either $r>0$,
\item or $r=0$, $c$ is effective, and $s\neq0$,
\item or $r=c=0$ and $s>0$.
\end{itemize}
\end{Def}

\begin{Thm}[Yoshioka]\label{thm:YoshiokaNonEmptyness}
Let $v\in H^*_{\alg}(X, \Z)$.
Assume that $v=mv_0$, with $m\in\Z_{>0}$ and $v_0$ a primitive positive vector.
Then $M_{\omega}^{\beta}(v)$ is non-empty for all $\omega,\beta$.
\end{Thm}

\begin{Rem}\label{rmk:SmoothnessFactorialityModuliSpace}
We keep the assumptions of Theorem \ref{thm:YoshiokaNonEmptyness}.
We further assume that $\omega$ is \emph{generic}\footnote{We refer to \cite{OGrady:ModuliVectBundles} for the notion of a generic polarization; it always
exists when $v_0$ is positive.} with respect to $v$.
\begin{enumerate}
\item\label{enum:KLS} By \cite{OGrady,KLS:SingSymplecticModuliSpaces, PeregoRapagnetta:Factoriality}, $M_{\omega}^{\beta}(v)$ is then a normal irreducible projective variety with $\Q$-factorial singularities.
\item\label{enum:DefoEquivalent} If $m=1$, then by \cite{Yoshioka:Abelian}, based on previous work by Mukai, O'Grady, and Huybrechts among others, $M_{\omega}^{\beta,s}(v)=M_{\omega}^{\beta}(v)$ is a smooth projective irreducible symplectic manifold of dimension $v^2+2$, deformation equivalent to the Hilbert scheme of points on a K3 surface.
\end{enumerate}
\end{Rem}

%Remark \ref{rmk:SmoothnessFactorialityModuliSpace},\eqref{enum:DefoEquivalent} can be used to give a precise description of the N\'eron-Severi group of the moduli space.

Let us briefly recall some relevant properties of the Beauville-Bogomolov form on $\NS(M)$ for a 
smooth projective irreducible symplectic manifold $M$.
It is a bilinear form $\NS(M) \times \NS(M) \to \R$. If $\varrho \in H^0(M, \Omega^2_M)$
is a global non-degenerate two-form, then there is a constant $c$ such that 
\[
(D_1, D_2) = c\int_M D_1 D_2 (\varrho \bar \varrho)^{\frac 12 \dim M -1}.
\]
Its associated quadratic
form is denoted by $q(D) = (D, D)$. It determines the volume of $D$ by
$c' \int_M D^{\dim M} = q(D)^{\frac 12 \dim M}$, for some constant $c' \in\R_{>0}$.
In particular, a nef divisor $D$ is big if and only if
$q(D) > 0$ (see also \cite[Corollary 3.10]{Huybrechts:compactHyperkaehlerbasic}).

As mentioned in the footnote to Definition \ref{def:determinant}, in the case of a K3 surface
we will always consider a dual version to the Donaldson morphisms.

\begin{Def}\label{def:MukaiHom}
Let $v\in H^*_{\alg}(X, \Z)$ be a positive primitive vector and let $\omega\in \mathrm{NS}(X)_\Q$ be an ample divisor which is generic with respect to $v$.
We define the \emph{Mukai homomorphism}
$ \theta_v \colon v^\perp \to \mathrm{NS}(M_{\omega}^{\beta}(v))$ by
\begin{equation} \label{eq:MukaiMukai}
 \theta_v(w).C = \frac 1{\rho}(w, \Phi_{\EE}(\OO_C)) 
\end{equation}
where $\EE$ is a quasi-universal family of similitude $\rho$.
\end{Def}

\begin{Rem} \label{rem:DonaldsonMukai}
With the same arguments as in Proposition \ref{prop:comparison}, we can identify the Mukai
homomorphism with the Donaldson morphism of Definition \ref{def:determinant} as follows:
given a class $w \in K_{\num}(X)$, we have
\[
\theta_v(v(w)) = -\lambda_{\EE}(w^*).
\]
\end{Rem}

The following result is proved in \cite[Sections 7 \& 8]{Yoshioka:Abelian} (see also \cite[Section 1.5]{GoettscheNakajimaYoshioka:KThDonaldson}):

\begin{Thm}[Yoshioka]\label{thm:YoshiokaSheaves}
Let $v\in H^*_{\alg}(X, \Z)$ be a positive primitive vector with $v^2 \ge 0$.
Let $\omega\in \mathrm{NS}(X)_\Q$ be an ample divisor which is generic with respect to $v$.
Then the Mukai homomorphism induces an isomorphism
\begin{itemize}
\item $\theta_v\colon v^\perp \xrightarrow{\sim} \mathrm{NS}(M_{\omega}^{\beta}(v))$, if $v^2>0$;
\item $\theta_v\colon v^\perp/v \xrightarrow{\sim} \mathrm{NS}(M_{\omega}^{\beta}(v))$, if $v^2=0$.
\end{itemize}
Under this isomorphism, the quadratic Beauville-Bogomolov form for $\mathrm{NS}(M_{\omega}^{\beta}(v))$
coincides with the quadratic form of the Mukai pairing on $X$.
\end{Thm}

\subsection*{Twisted K3 surfaces}
The results in the previous sections can be generalized to twisted K3 surfaces\footnote{For the basic theory of twisted K3 surfaces, we refer, for example, to \cite{Caldararu:Thesis}}.
Let $(X,\alpha)$ be a twisted K3 surface, with $\alpha\in\mathrm{Br}(X)$.

We denoted by $\Coh(X,\alpha)$ the category of $\alpha$-twisted coherent sheaves on $X$.
A twisted Chern character for elements of $\Coh(X,\alpha)$ has been defined\footnote{To be precise, the twisted Chern character and the twisted Mukai lattice depend on the choice of a B-field lift $\beta_0\in H^2(X,\Z)$ of $\alpha$. See also \cite[Remark 3.3]{HMS:generic_K3s} for further comments on this.} in \cite{HuybrechtsStellari:Twisted}.
It takes values in the algebraic part of the \emph{twisted Mukai lattice}; this will be denoted by $H^*_\alg(X,\alpha,\Z)$.

Given $\omega,\beta\in\NS(X)_\Q$ with $\omega$ ample, we can define a notion of stability as in the untwisted case.
Moduli of $\beta$-twisted $\omega$-Gieseker semistable $\alpha$-twisted sheaves share the same properties as untwisted sheaves (see \cite{Yoshioka:TwistedStability,Lieblich:Twisted}).
In particular, the twisted version of Theorem \ref{thm:YoshiokaNonEmptyness} is proved in \cite[Theorem 3.16]{Yoshioka:TwistedStability}, while Theorem \ref{thm:YoshiokaSheaves} (and Remark \ref{rmk:SmoothnessFactorialityModuliSpace},\eqref{enum:DefoEquivalent}) is proved in \cite[Theorem 3.19]{Yoshioka:TwistedStability}. Remark \ref{rmk:SmoothnessFactorialityModuliSpace},\eqref{enum:KLS} can be proved in a similar way.

\section{Review: Stability conditions and moduli spaces for objects on K3 surfaces}
\label{sec:reviewK3}

In this section we give a brief review of Bridgeland's results on stability conditions for K3
surfaces in \cite{Bridgeland:K3}, and of results by Toda, Yoshioka and others related to moduli spaces of
Bridgeland-stable objects.

\subsection*{Space of stability conditions for a K3 surface}
Let $X$ be a smooth projective K3 surface.
Fix $\omega,\beta\in\NS(X)_\Q$ with $\omega$ ample.

Let $\TT(\omega,\beta) \subset \Coh X$ be the subcategory of sheaves whose HN-filtrations factors (with respect to slope-stability)
have $\mu_{\omega, \beta} > 0$, and $\FF(\omega,\beta)$ the subcategory of sheaves with
HN-filtration factors satisfying $\mu_{\omega, \beta} \le 0$.
Next, consider the abelian category
\begin{equation*} %\label{eq:AK3}
\AA(\omega,\beta):=\left\{E\in\Db(X):\begin{array}{l}
\bullet\;\;\HH^p(E)=0\mbox{ for }p\not\in\{-1,0\},\\\bullet\;\;
\HH^{-1}(E)\in\FF(\omega,\beta),\\\bullet\;\;\HH^0(E)\in\TT(\omega,\beta)\end{array}\right\}
\end{equation*}
and the $\C$-linear map
\begin{equation} \label{eq:ZK3}
Z_{\omega,\beta} \colon K_{\num}(X)\to\C,\qquad E\mapsto(\exp{(\beta+\sqrt{-1}\omega)},v(E)).
\end{equation}
If $Z_{\omega,\beta}(F)\notin\R_{\leq0}$ for all spherical sheaves $F\in\Coh(X)$ (e.g., this holds
when $\omega^2>2$), then by \cite[Lemma 6.2, Prop.\ 7.1]{Bridgeland:K3},
the pair $\sigma_{\omega,\beta}=(Z_{\omega,\beta},\AA(\omega,\beta))$ defines a stability condition.
For objects $E \in \AA(\omega, \beta)$, we will denote their phase with respect to $\sigma_{\omega,
\beta}$ by $\phi_{\omega, \beta}(E) = \phi(Z(E)) \in (0, 1]$.
By using the support property, as proved in \cite[Proposition 10.3]{Bridgeland:K3}, we can extend the above and define stability conditions $\sigma_{\omega,\beta}$, for $\omega,\beta\in\NS(X)_\R$.

Denote by $U(X)\subset\Stab(X)$ the open subset consisting of the stability conditions
$\sigma_{\omega,\beta}$ just constructed up to the action of $\widetilde{\GL}_2(\R)$. It can also be
characterized as the open subset $U(X) \subset \Stab(X)$ consisting of stability conditions for
which the skyscraper sheaves $k(x)$ of points are stable of the same phase.
Let $\Stab^{\dagger}(X) \subset \Stab(X)$ be the connected component containing $U(X)$.
Let $\PP(X)\subset H^*_{\alg}(X)_\C$ be the subset consisting of vectors whose real and imaginary
parts span positive definite two-planes in $H^*_{\alg}(X)_\R$ with respect to the Mukai pairing.
It has two connected components, corresponding to the induced orientation of the two-plane.
Choose $\PP^+(X)\subset\PP(X)$ as the connected component containing the vector
$(1,i\omega,-\omega^2/2)$, for $\omega\in\mathrm{NS}(X)_\R$ the class of an ample divisor.
Furthermore, let $\Delta(X):=\{ s \in H^*_{\alg}(X, \Z) \colon s ^2=-2\}$ be the set of spherical 
classes, and, for $s \in\Delta$,
\[
s ^\perp_\C:=\stv{ \Omega\in H^*_{\alg}(X)_\C}{(\Omega,s)=0}.
\]
Finally, set
\[
\PP_0^+(X):=\PP^+(X) \setminus \underset{s\in\Delta(X)}{\bigcup}s^\perp_\C\subset H^*_\alg(X)_\C.
\]

Since the Mukai pairing $(\blank,\blank)$ is non-degenerate, we can define
$\eta(\sigma)\in H^*_{\alg}(X)_\C$ 
for a stability condition $\sigma=(Z,\PP)\in\Stab^\dagger(X)$ by 
\[
\ZZ (\sigma) (\blank) = \left(\blank,\eta(\sigma)\right).
\]

\begin{Thm}[Bridgeland]\label{thm:BridgelandK3}
The map $\eta\colon \Stab^\dagger(X)\to H^*_{\alg}(X)_\C$ is a covering map onto its image $\PP^+_0(X)$.
\end{Thm}

The proof of Theorem \ref{thm:BridgelandK3} relies on an explicit description of the boundary
$\partial U(X)$ of $U(X)$, see \cite[Theorem 12.1]{Bridgeland:K3}:

\begin{Thm}\label{thm:BoundaryU}
Suppose that $\sigma=(Z,\PP)\in\partial U(X)$ is a generic point of the boundary of $U(X)$.
Then exactly one of the following possibilities holds.
\begin{enumerate}
\item[$(A^+)$] There is a rank $r$ spherical vector bundle $A$ such that the only stable factors of
the objects $\{k(x)\colon x\in X\}$ in the stability condition $\sigma$ are $A$ and $\ST_A(k(x))$.
Thus, the Jordan-H\"older filtration of each $k(x)$ is given by
\[
0\to A^{\oplus r} \to k(x) \to \ST_A(k(x)) \to 0.
\]
\item[$(A^-)$] There is a rank $r$ spherical vector bundle $A$ such that the only stable factors of
the objects $\{k(x)\colon x\in X\}$ in the stability condition $\sigma$ are $A[2]$ and $\ST_A^{-1}(k(x))$.
Thus, the Jordan-H\"older filtration of each $k(x)$ is given by
\[
0\to \ST_A^{-1}(k(x)) \to k(x) \to A^{\oplus r}[2] \to 0.
\]
\item[$(C_k)$] There are a nonsingular rational curve $C\subset X$ and an integer $k$ such that $k(x)$ is stable in the stability condition $\sigma$, for $x\notin C$, and such that the Jordan-H\"older filtration of $k(x)$, for $x\in C$, is given by
\[
0 \to \OO_C(k+1) \to k(x) \to \OO_C(k)[1] \to 0.
\]
\end{enumerate}
\end{Thm}

A \emph{generic point of the boundary} is a stability condition
which lies on one wall only (in the sense of Proposition \ref{prop:chambers}). 
We also recall that $\ST_A$ denotes the spherical twist functor of \cite{Seidel-Thomas:braid} associated to
the spherical object $A$.

\begin{Rem} \label{rmk:Cknegative}
In the boundary of type $(C_k)$, the Mukai vectors of the stable factors of
$k(x)$ span a negative semi-definite plane in $H^*_{\alg}(X)_\R$.
\end{Rem}

\begin{Rem}\label{rmk:twisted}
As proven in \cite[Section 3.1]{HMS:generic_K3s}, the
results stated in this section extend without any difference to the case of twisted K3 surfaces:
Let $(X,\alpha)$ be a twisted K3 surface, $\alpha\in\mathrm{Br}(X)$.
Following \emph{loc. cit.}, we define the following objects analogously to the untwisted case:
$\Stab^\dagger(X,\alpha)$, $U(X,\alpha)$, $\PP_0^+(X,\alpha)$.
The statement corresponding to Theorem \ref{thm:BridgelandK3} is \cite[Proposition 3.10]{HMS:generic_K3s}.
Only Theorem \ref{thm:BoundaryU} is not treated explicitly in \cite[Section 3.1]{HMS:generic_K3s}.
However, the only geometric statement used in the proof in \cite[Section 12]{Bridgeland:K3} is 
\cite[Proposition 2.14]{Mukai:BundlesK3}, which states that a spherical torsion-free sheaf on a
K3 surface is automatically locally free. Mukai's proof carries over without change.
\end{Rem}

\subsection*{Moduli stacks of semistable objects}\label{subsec:ModuliStacks}
%Define a $2$-functor $\GGG\colon \mathrm{Sch}_\C\to{\rm Grp}$
%by mapping a $\C$-scheme $S$ to the groupoid $\GGG(S)$, whose objects consist of those
%$\EE\in\mathrm{D}_{S\text{-}\mathrm{perf}}(S\times X)$ which satisfy $\Ext^i(\EE_s,\EE_s)=0$, for
%all $i<0$ and all closed points $s\in S$, and whose morphisms are isomorphisms in
%$\mathrm{D}_{S\text{-}\mathrm{perf}}(S\times X)$.
%Here $\EE_s$ is the derived restriction of $\EE$ to $\Db(\{s\}\times X) \cong \Db(X)$.
%The main theorem in \cite{Lieblich:mother-of-all} (generalizing results in \cite{Inaba}) shows that $\GGG$ is an Artin stack, locally of finite type over $\C$.

Fix $\sigma=(Z,\AA)\in\Stab(X)$, $\phi\in\R$, and $v\in H^*_{\alg}(X,\Z)$.
We let $\mathfrak{M}_{\sigma}(v,\phi)$ be the moduli stack of flat families of
$\sigma$-semistable objects of class $v$ and phase $\phi$:
its objects are given by complexes $\EE\in\mathrm{D}_{S\text{-}\mathrm{perf}}(S\times X)$
whose restrictions $\EE_s$ belong to $\PP(\phi)$ and have Mukai vector $v$, for all closed points
$s\in S$.  We will often omit $\phi$ from the notation; in fact, by acting with an element of $\C$, as in Remark \ref{rmk:GroupAction}, and by using Lemma \ref{lem:Yukinobualgebraic}, we can always assume $\phi=1$ and $\sigma$ algebraic.

Based on results in \cite{Inaba, Lieblich:mother-of-all} on the stack of objects in $\Db(X)$,
the following theorem is proved in \cite[Theorem 1.4 and Section 3]{Toda:K3}:

\begin{Thm}[Toda]\label{thm:TodaThmA}
Let $X$ be a K3 surface and let $\sigma\in\Stab^\dagger(X)$.
Then $\sigma$-stability is an open property and $\mathfrak{M}_{\sigma}(v,\phi)$ is an Artin stack of finite type over $\C$.
\end{Thm}

Let $\MMM^s_{\sigma}(v,\phi)\subseteq\MMM_{\sigma}(v,\phi)$ be the open substack parameterizing stable objects.
Inaba proved in \cite{Inaba} that $\MMM^s_{\sigma}(v,\phi)$ is a $\mathbb{G}_m$-gerbe over a
symplectic algebraic space $M^s_{\sigma}(v,\phi)$. Toda's proof is based on the following statement,
which will also need directly:

\begin{Lem}\label{lem:properness}
Fix $\phi\in\R$ and $v\in H^*_{\alg}(X, \Z)$.
\begin{enumerate}
\item The moduli stack $\MMM_\sigma(v,\phi)$ satisfies the valuative criterion of universal
closedness.
\item\label{case:ProperToda} Assume that $\MMM_\sigma(v,\phi)=\MMM^s_\sigma(v,\phi)$. Then the coarse moduli space $M_{\sigma}(v,\phi)$ is a proper algebraic space.
\end{enumerate}
\end{Lem}

\begin{Prf}
As remarked above, we can assume that $\phi=1$, $Z(v)=-1$, and that $\sigma$ is algebraic.
As a consequence, $\AA$ is Noetherian. 
In this case, \cite[Theorem 4.1.1]{Abramovich-Polishchuk:t-structures} implies
the lemma.
\end{Prf}

\subsection*{Moduli spaces of semistable objects}

We generalize the results in Section \ref{sec:ReviewGieseker} to Bridgeland stability.
The key fact is a comparison between Bridgeland and Gieseker stability when the polarization is ``large''.

\begin{Thm}[{\cite[Proposition 14.1]{Bridgeland:K3} and \cite[Section 6]{Toda:K3}}] \label{thm:BridgelandToda}
Let $v\in H^*_\alg(X,\Z)$, and let $\beta\in\NS(X)_\Q$, $H\in\NS(X)$ be classes with $H$ ample and $\mu_{H, \beta}(v) > 0$.
If we set $\omega=tH$, then $M_{\sigma_{\omega, \beta}}(v) = M_H^\beta(v)$ for $t\gg0$. 
\end{Thm}

We will give a precise bound for $t$ in Corollary \ref{Cor:amplecone}.
The following generalizes Theorem \ref{thm:YoshiokaNonEmptyness}:

\begin{Thm}[Toda, Yoshioka]\label{thm:nonempty}
Let $v\in H^*_{\alg}(X, \Z)$.
Assume that $v=mv_0$, with $m\in\Z_{>0}$ and $v_0$ a primitive vector with $v_0^2 \ge -2$.
Then $\MMM_\sigma(v,\phi)(\C)$ is non-empty for all $\sigma = (Z, \AA) \in\Stab^\dagger(X)$ and
all $\phi \in \R$ with $Z(v) \in \R_{>0}\cdot e^{i\phi \pi}$.
\end{Thm}

\begin{Prf}
Since we are interested in semistable objects, we can assume that $v=v_0$ is primitive.
Also, since being semistable is a closed condition on $\Stab(X)$, we can assume that $\sigma$ is
generic with respect to $v$, so that every $\sigma$-semistable object of class $v$ is stable.
Then the Joyce invariant $J(v)$ of \cite{Toda:K3} is the motivic invariant of the 
proper coarse moduli space $M_{\sigma}(v)$.

By \cite[Theorem 1.4]{Toda:K3}, $J(v)$ does not depend on $\sigma$, and it is invariant under autoequivalences of $\Db(X)$.
Hence, up to acting by the shift functor, tensoring with a line bundle, and the spherical twist 
$\ST_{\OO}$, we can assume that $v$ is positive,
and, by using Theorem \ref{thm:BridgelandToda}, that $J(v)$ is equal to the motivic invariant of the moduli space $M_H(v)$ of Gieseker
stable sheaves on $X$ with Mukai vector $v$, for a generic polarization $H$.
Since $v$ is positive, Theorem \ref{thm:YoshiokaNonEmptyness} shows that $M_H(v)$ is non-empty.
Hence, $J(v)$ is non-trivial, and so $\MMM_\sigma(w,\phi)(\C)$ is non-empty for all $\sigma$.
\end{Prf}

By Theorem \ref{thm:nonempty} and \cite{Inaba:Symplectic}, we get the following corollary, which generalizes Remark \ref{rmk:SmoothnessFactorialityModuliSpace},\eqref{enum:DefoEquivalent}.

\begin{Cor}\label{cor:nonempty}
Let $v\in H^*_{\alg}(X, \Z)$ be a primitive vector with $v^2 \ge -2$, and let $\sigma\in\Stab^\dagger(X)$ be a generic stability condition with respect to $v$.
Then $M_{\sigma}(v)$ is non-empty, consists of stable objects, and it is a smooth proper symplectic algebraic space of dimension $v^2+2$.
\end{Cor}

Finally, we need the following re-writing of Theorem \ref{thm:YoshiokaSheaves}.
The Mukai homomorphism, as defined in Definition \ref{def:MukaiHom}, is well-defined for Bridgeland stability as well, and denoted in the same way, $\theta_v \colon v^\perp \to \mathrm{NS}(M_{\sigma}(v))$.

\begin{Thm}\label{thm:Yoshioka}
Let $(X,\alpha)$ be a twisted K3 surface.
Let $v\in H^*_{\alg}(X,\alpha,\Z)$ be a primitive vector with $v^2 \ge 0$.
Let $\sigma\in\Stab^\dagger(X,\alpha)$ be a generic stability condition with respect to $v$.
Assume that there exist a K3 surface $X'$, a Brauer class $\alpha'\in\mathrm{Br}(X')$, a polarization
$H'\in\mathrm{NS}(X')$, and a derived equivalence $\Phi\colon \Db(X,\alpha)\to\Db(X',\alpha')$ such that
\begin{enumerate}
\item $v'=\Phi(v)$ is positive,
\item $H'$ is generic with respect to $v'$, and
\item $M_{\Phi(\sigma)}(v')$ consists of twisted $H'$-Gieseker stable sheaves on $(X',\alpha')$.
\end{enumerate}
Then $M_{\sigma}(v)$ is an irreducible symplectic projective manifold, and the Mukai homomorphism induces an isomorphism
\begin{itemize}
\item $\theta_v\colon v^\perp \xrightarrow{\sim} \mathrm{NS}(M_{\sigma}(v))$, if $v^2>0$;
\item $\theta_v\colon v^\perp/v \xrightarrow{\sim} \mathrm{NS}(M_{\sigma}(v))$, if $v^2=0$.
\end{itemize}
Under this isomorphism, the quadratic Beauville-Bogomolov form for $\mathrm{NS}(M_{\sigma}(v))$
coincides with the quadratic form of the Mukai pairing on $X$.
\end{Thm}

\begin{Prf}
Since everything is compatible with Fourier-Mukai equivalences, this follows from \cite{Orlov:representability,Canonaco-Stellari}.
\end{Prf}

\section{K3 surfaces: Projectivity of moduli spaces}
\label{sec:ProjK3}

Let $X$ be a smooth projective K3 surface, and let $v\in H^*_{\alg}(X,\Z)$.

In the recent preprint \cite{MYY2}, Minamide, Yanagida, and Yoshioka proved the following: if $\NS(X)\cong\Z$
and $\sigma \in \Stab^\dagger(X)$ is a generic stability condition with respect to $v$,
then there exist another K3 surface $Y$, a Brauer class $\alpha\in\mathrm{Br}(Y)$, and a derived equivalence
$\Phi\colon \Db(X)\to\Db(Y,\alpha)$ such that the moduli stack $\MMM_{\sigma}(v)$ is isomorphic to a moduli stack of (twisted) Gieseker semistable sheaves on $(Y,\alpha)$ via $\Phi$.

In this section, we improve their argument, and we remove the assumption on the rank of the
N\'eron-Severi group.  As a consequence, the divisor class $\ell_{\sigma}$ will give an ample divisor on
the coarse moduli space.

We write $v=mv_0\in H^*_{\alg}(X, \Z)$, where $m\in\Z_{>0}$, and $v_0=(r,c,s)$ is primitive with $v_0^2
\ge -2$.  We start by examining the cases in which $v_0^2\leq0$.

\begin{Lem}\label{lem:spherical}
Assume that $v_0^2=-2$.
Then, for all $\sigma\in\Stab^\dagger(X)$ generic with respect to $v$, $\MMM_\sigma(v)$ admits a coarse moduli space $M_{\sigma}(v)$ consisting of a single point.
\end{Lem}

\begin{Prf}
Corollary \ref{cor:nonempty} shows that for all generic $\sigma\in\Stab^\dagger(X)$, the stack
$\MMM_{\sigma}(v_0)=\MMM^s_{\sigma}(v_0)\neq\emptyset$ is a $\mathbb{G}_m$-gerbe over a point.  The
corresponding object $E_0$ is spherical, and in particular admits no non-trivial self-extensions.
If $m>1$, then $v^2 < -2$ shows that there cannot exist any stable object with vector $v$. By
induction, every semistable object with Mukai vector $v$ must be of the form $E_0^{\oplus m}$.
\end{Prf}

\begin{Lem}\label{lem:isotropic}
Assume that $v_0^2=0$. Let $\sigma\in\Stab^\dagger(X)$ be a generic stability condition with respect to $v$.
\begin{enumerate}
\item\label{case:vprimit} For $m=1$, $M_{\sigma}(v_0)$ is a smooth projective K3 surface, and there exist a class $\alpha\in\mathrm{Br}(M_{\sigma}(v_0))$ and a derived equivalence
\[
\Phi_{\sigma,v_0}\colon \Db(X)\xrightarrow{\sim}\Db(M_{\sigma}(v_0),\alpha).
\]
\item\label{case:vnonprimit} For $m>1$, a coarse moduli space $M_{\sigma}(v)$ exists and
\[
M_{\sigma}(v) \cong \mathrm{Sym}^m\left( M_{\sigma}(v_0) \right).
\]
\end{enumerate}
\end{Lem}

\begin{Prf}
Corollary \ref{cor:nonempty} shows the non-emptiness.
The fact that $M_{\sigma}(v_0)$ is a smooth projective K3 surface and the derived equivalence is a
classical result of Mukai and C{\u{a}}ld{\u{a}}raru \cite{Mukai:BundlesK3, Caldararu:NonFineModuliSpaces} for stable sheaves.
This can be generalized to stable complexes as follows.
Again by Corollary \ref{cor:nonempty}, $M_{\sigma}(v_0)$ is smooth projective symplectic surface.
By \cite{Lieblich:mother-of-all}, it also comes equipped with a torsion class in its Brauer group.
The equivalence $\Phi_{\sigma,v_0}$ follows now from \cite{BondalOrlov:Main, Bridgeland:EqFMT}.
This shows \eqref{case:vprimit}.

The proof of \eqref{case:vnonprimit} follows now as in \cite{MYY2}.
Indeed, clearly $\MMM_{\sigma}(v)\neq\emptyset$, and the derived equivalence $\Phi$ maps any complex in $\MMM_{\sigma}(v)(\C)$ in a torsion sheaf on $M_{\sigma}(v_0)$ of dimension $0$ and length $m$.
\end{Prf}

We can now prove Theorem \ref{thm:ProjK3}, based on an idea of Minamide, Yanagida, and Yoshioka.  By
Lemma \ref{lem:spherical} and Lemma \ref{lem:isotropic}, we can restrict to the case $v^2>0$.

The following result is proved in \cite[Sections 4.1 and 3.4]{MYY2} for abelian surfaces, and for
K3 surfaces of Picard rank one. For the convenience of the reader, we give a self-contained proof
for arbitrary K3 surfaces:

\begin{Lem}\label{lem:MYY}
Let $\sigma = (Z_\sigma, \AA_\sigma)$ be a generic stability condition with respect to $v$,
lying inside a chamber $\CC$ with respect to $v$.
Then $\CC$ contains a dense subset of 
stability conditions $\tau = (Z_{\tau}, \AA_\tau)$ for which there
exists a primitive Mukai vector $w$ with $w^2=0$ such that:
\begin{enumerate}
\item \label{enum:isotropic}
$Z_{\tau}(w)$ and $Z_{\tau}(v)$ lie on the same ray in the complex plane.
\item \label{enum:stable}
All $\tau$-semistable objects with Mukai vector $w$ are stable, and
$M_{\tau}(w)$ is a smooth projective K3 surface.
\end{enumerate}
\end{Lem}

\begin{Prf}
Let us consider claim \eqref{enum:isotropic}. We may assume $Z_{\sigma}(v) = -1$ and restrict
our attention to stability conditions $\tau$ with $Z_{\tau}(v) = -1$.
Let $Q \subset H^*_{\alg}(X)_\R$ be the 
quadric defined by $w^2 = 0$. Due to the signature of the Mukai pairing, there is a real
solution $w_r$ to the pair of equations $\Im Z_{\sigma}(w) = 0$ and $w^2 = 0$.
Since $Q$ has a rational point, rational points are dense in $Q$, i.e., there exists
$w_q \in H^*_{\alg}(X)_\Q$ arbitrarily close to $w_r$ with $w_q^2 = 0$. If $w_q$ is sufficiently
close, and since $w_q$ must be linearly independent of $v$, there will be $\tau = (Z_\tau,
\AA_\tau)$ nearby $\sigma$ such that $\Im Z_\tau(v) = \Im Z_\tau(w_q) = 0$ and
$\Re Z_\tau = \Re Z_\sigma$. Replacing $w_q$ by the unique primitive integral class $w \in \R \cdot w_q$
with $\Re Z_{\tau}(w) < 0$ finishes the proof of the first claim.

It remains to show that claim \eqref{enum:stable} holds,
after possibly replacing $w$ and a further deformation of $\tau$.
Note that small deformations of $\tau$ in a codimension one submanifold of $\Stab(X)$ will
keep property \eqref{enum:isotropic} intact. If this contains a stability condition
generic with respect to $w$, our claim follows from Lemma \ref{lem:isotropic}.
Otherwise, we can assume that $\tau$ is on a generic point of a wall, and that
for $u \in H^*_{\alg}(X, \Z)$, the complex number $Z(u)$ has the same phase as
$Z(v)$ and $Z(w)$ if and only if $u$ is a linear combination of $v$ and $w$.

Using the Fourier-Mukai transform associated to $M_{\rho}(w)$ for $\rho$ nearby $\tau$ and
generic, we can further assume that $w = (0, 0, 1)$ is the Mukai vector of a point, and 
that $\rho$ is in a generic boundary point of the geometric chamber $U(X)$ as described
in Theorem \ref{thm:BoundaryU}. If $M_{\rho}(v)$ is not a fine moduli space, we need to consider $M_{\rho}(v)$ as a twisted K3 surface; see Remark \ref{rmk:twisted}.

In the case of a wall of type $(A^+)$, let $w'$ be the Mukai vector of $\ST_A(k(x))$.
Since the objects $\ST_A(k(x))$ are $\tau$-stable, the stability condition $\tau$ is generic
with respect to $w'$, we have ${w'}^2 = 0$, and 
$Z(w')$ has the same phase as $Z(v)$. The case $(A^-)$ is analogous.

If we are in case $(C_k)$, then as pointed out in Remark \ref{rmk:Cknegative}, the Mukai
pairing is negative semi-definite on the linear span
$\langle w, v(\OO_C(k+1)) \rangle$. However, since $Z(\OO_C(k+1))$ has the same phase
as $Z(v)$, this linear span is equal to the linear span 
$\langle v, w \rangle$ of $v, w$, in contradiction to $v^2>0$.
\end{Prf}

Let $w$ be the Mukai vector from Lemma \ref{lem:MYY}.
Let $Y:=M_\tau(w)$, and let $\alpha\in\mathrm{Br}(Y)$ be a Brauer class so that the choice of a (quasi-)universal family induces a derived equivalence $\Phi:\Db(X)\xrightarrow{\sim}\Db(Y,\alpha)$.
Consider the stability condition $\tau':=\Phi(\tau)\in\Stab(Y,\alpha)$.
By \cite[Section 5]{HuybrechtsStellari:Twisted}, we can assume that $\tau'\in\Stab^\dagger(Y,\alpha)$.
Then, by construction, for all $F\in\MMM_{\tau}(w)(\C)$, $\Phi(F)\cong k(y)$, for some $y\in Y$.
Therefore the skyscraper sheaves are all $\tau'$-stable with the same phase, namely $\tau'\in U(Y,\alpha)$.
Up to acting by $\widetilde\GL^+_2(\R)$, we can assume that $\tau'=\sigma_{\omega',\beta'}$, for some $\omega',\beta'\in\NS(Y)_\Q$, with $\omega'$ ample.

Since $Z_\tau(v)$ and $Z_\tau(w)$ lie on the same ray in the complex plane, we have $Z_{\omega',\beta'}(\Phi(v)) \in \R_{<0}$.
%If we write $-\Phi(v) = (r, c, s)$, this is equivalent to $\omega.c - r \omega.\beta = 0$.
Note that by the construction of Lemma \ref{lem:MYY}, the stability condition $\tau'$ is
still generic with respect to $\Phi(v)$.
Since $w$ does not lie in a wall for $v$, $k(y)$ is not a stable factor (with respect to $\tau'$) for $\Phi(E)$, for all $E\in\MMM_{\tau}(v)(\C)$.
By definition of the category $\AA(\omega',\beta')$, this shows that $\Phi(E)[-1]$ is a $\alpha$-twisted locally-free sheaf on $Y$, which is $\mu_{\omega',\beta'}$-semistable.

We claim that $\omega'$-slope (semi)stability for sheaves of class $-\Phi(v)$ is equivalent to
twisted $\omega'$-Gieseker (semi)stability: indeed, assume that a sheaf $E'$ with
$v(E') = -\Phi(v)$ is slope-semistable.
If $F' \subset E'$ is a saturated subsheaf of the same slope, then
$F'[1]$ is a subobject of $E'[1]$ in $\PP_{\omega', \beta'}(1)$; since $\tau'$ is generic
with respect to $\Phi(v)$, this
means that $v(F')$ is proportional to $v(E')$; hence the twisted Hilbert polynomial of $F'$ is
proportional to the twisted Hilbert polynomial of $E'$, and this will hold independently of the
twist $\beta'$. In particular, twisted Gieseker stability on $Y$
for $-\Phi(v)$ is equivalent to untwisted Gieseker-stability.
This shows that $\Phi\circ[-1]$ induces an isomorphism of stacks
\begin{equation} \label{eq:isomtoGieseker}
\MMM_{\tau}(v)\xrightarrow{\sim}\MMM_{\omega'}(-\Phi(v)),
\end{equation}
where, as in Section \ref{sec:ReviewGieseker}, $\MMM_{\omega'}(-\Phi(v))$ is the moduli stack of Gieseker semistable sheaves on $(Y,\alpha)$.
Moreover, the isomorphism preserves $S$-equivalence classes.
Hence, a coarse moduli space for $\MMM_{\tau}(v)$ exists, since it exists for $\MMM_{\omega'}(-\Phi(v))$, and it is a normal irreducible projective variety with $\Q$-factorial singularities, by Remark \ref{rmk:SmoothnessFactorialityModuliSpace},\eqref{enum:KLS}.
This concludes the proof of the first part of Theorem \ref{thm:ProjK3}.

We can now show the second part of Theorem \ref{thm:ProjK3}, namely that $\ell_\sigma$ is well-defined on the coarse moduli space $M_{\sigma}(v)$ and it is ample.

\begin{Lem}\label{lem:PositiveInters}
Let $\sigma=(Z,\PP)\in\Stab^\dagger(X)$ be such that $Z(v)=-1$, and let $w_{\sigma}:=\Im(\eta(\sigma))$.
Then $w_{\sigma}^2>0$.
\end{Lem}

\begin{Prf}
This follows directly from Theorem \ref{thm:BridgelandK3}, since $\eta(\sigma)\in\PP^+_0(X)$.
\end{Prf}

We first deal with the case when $v$ is primitive.
Since $\sigma$ generic with respect to $v$, $\MMM_{\sigma}(v)=\MMM^s_{\sigma}(v)$ is a $\mathbb{G}_m$-gerbe over $M_\sigma(v)$.
Moreover, by the first part of Theorem \ref{thm:ProjK3} and Remark \ref{rmk:SmoothnessFactorialityModuliSpace},\eqref{enum:KLS}, $M_\sigma(v)$ is a smooth projective irreducible symplectic manifold.
Hence, by Remark \ref{rem:quasidivisor}, the divisor class $\ell_{\sigma}$ is well-defined on $M_{\sigma}(v)$.

%\begin{Lem}\label{lem:GNY}
%Let $v\in H^*_{\alg}(X, \Z)$ be a primitive vector with $v^2 \ge 2$, and $\sigma\in\Stab^\dagger(X)$ be generic with respect to $v$.
%Then for all $u\in v^\perp_\R$ we have $q(\theta_v(u)) = u^2$, 
%where $q(\theta_v(u))$ denotes the quadratic Beauville-Bogomolov form on $M_{\sigma}(v)$.
%\end{Lem}

%\begin{Prf}
%By the first part of Theorem \ref{thm:ProjK3}, this follows directly from Theorem \ref{thm:Yoshioka}.
%\end{Prf}

\begin{Cor}
Let $v\in H^*_{\alg}(X, \Z)$ be a primitive vector with $v^2 \ge 2$, and $\sigma\in\Stab^\dagger(X)$ be generic with respect to $v$.
Then the divisor $\ell_{\sigma}$ is ample.
\end{Cor}

\begin{Prf}
By Theorem \ref{thm:Yoshioka}, $q(\ell_{\sigma}) = w_{\sigma}^2$.  By Lemma \ref{lem:PositiveInters}, $w_\sigma^2 >
0$, and so $\ell_{\sigma}$ is big and has the strong positivity property of Theorem \ref{thm:main1}.  As a symplectic projective manifold, $M_{\sigma}(v)$ has trivial
canonical bundle; so the Base Point Free Theorem \cite[Theorem 3.3]{KollarMori} implies that (a multiple of)
$\ell_{\sigma}$ is globally generated, and hence ample (see also \cite[Proposition 6.3]{Huybrechts:compactHyperkaehlerbasic}).
\end{Prf}

The case in which $v$ is not primitive is more delicate, since we
do not have a version of Theorem \ref{thm:Yoshioka} available.
Instead, we have to use an explicit comparison with determinant line bundles and rely on the
GIT construction for dealing with properly semistable objects; we use \cite[Section 8.1]{HL:Moduli}
as a reference for the classical construction.

By the openness and convexity of the ample cone, it is sufficient to prove the ampleness
of $\ell_\sigma$ for a dense subset of stability conditions in a given chamber.
We can therefore assume that $\sigma$ satisfies the properties of
the stability condition $\tau$ in Lemma \ref{lem:MYY}; 
let $\Phi$ be the induced derived equivalence $\Phi \colon \Db(X) \to \Db(Y, \alpha)$.

We will first assume $\alpha = 0$.
By \eqref{eq:isomtoGieseker}, $\MMM_{\Phi(\sigma)}(-\Phi(v))$ consists of $\omega'$-Gieseker semistable sheaves on $Y$, where $\omega'$ is a generic polarization.
By \cite[Theorem 8.1.5]{HL:Moduli} and \cite[Th\'eor\`eme 5 \& Proposition 6]{LePotier:StrangeDuality}, we know that $\ell_{\Phi(\sigma)}$ defines a divisor class on the coarse moduli space $M_{\Phi(\sigma)}(-\Phi(v))$; this class depends only on $\Phi(\sigma)$, it has the positive property as in Theorem \ref{thm:main1}, and it is compatible with $\ell_{\Phi(\sigma),\EE}$ via pull-back.
Since, by \cite{Orlov:representability,Canonaco-Stellari}, the equivalence $\Phi$ is of Fourier-Mukai type, and the construction of $\ell_{\sigma}$ is compatible with the
convolution of the Fourier-Mukai kernels, $\ell_{\sigma}$ gives a well-defined class on the coarse moduli space $M_{\sigma}(v)$ as well.

We write $-\Phi(v) = (r, c, s)$,
Let $\LL_0, \LL_1$ be as defined in \cite[Definition 8.1.9]{HL:Moduli}; 
after identifying $h$ of \cite[Section 8.1]{HL:Moduli} with $\omega'$,
then in our notation we have $\LL_0 = \theta_{-\Phi(v)}((-r, 0, s))$ and
$\LL_1 = \theta_{-\Phi(v)}((0, r\omega', \omega'.c))$.
It is immediate to check that, up to rescaling and the functor $\Phi\circ[-1]$, $\ell_{\sigma}$ coincides with the class $\LL_1$.
By Theorem \ref{thm:main1}, $\ell_{\sigma}$ is nef.
By \cite[Theorem 8.1.11 \& Remark 8.1.12]{HL:Moduli}, the line bundle $\LL_0\otimes\LL_1^{\otimes m}$ is ample for $m\gg0$.
Moreover, $\LL_0\otimes\LL_1^{\otimes m}$ for $m \gg 0$ are (up to rescaling) induced by a stability
conditions arbitrarily close to $\sigma$.
Hence we have found a dense subset of stability conditions for which $\ell_{\sigma}$ is ample.

Finally, in case $\alpha \neq 0$, one can use Proposition 2.3.3.6 and Lemma 2.3.2.8 of
\cite{Lieblich:Twisted} to reduce to the case $\alpha = 0$.
This finishes the proof of Theorem \ref{thm:ProjK3}.

\section{Flops via wall-crossing}
\label{sec:flops}

In this section, we will first discuss the possible phenomena at walls in $\Stab(X)$, 
and then proceed to prove Theorem \ref{thm:MMP} .

Let $X$ be a smooth projective K3 surface, let $v$ be a primitive Mukai vector with $v^2 \ge -2$.
Consider a wall $W \subset \Stab(X)$ with respect to $v$ in the sense of Proposition~\ref{prop:chambers}.

Let $\sigma_0=(Z_0,\AA_0) \in W$ be a generic point on the  wall.
Let $\sigma_+=(Z_{+},\AA_{+}),\sigma_{-}=(Z_{-},\AA_{-})$ be two algebraic stability conditions 
in the two adjacent chambers.
By the results of the previous section, the two moduli spaces $M_{\sigma_{\pm}}(v)$ are non-empty, irreducible symplectic projective manifolds.
If we choose (quasi-)universal families $\EE_{\pm}$ on $M_{\sigma_{\pm}}(v)$ of $\sigma_{\pm}$-stable objects, we obtain
in particular (quasi-)families of $\sigma_0$-semistable objects.
Hence, Theorem \ref{thm:nefdivisor} gives us nef divisor classes
$\ell_{\sigma_0,\EE_{\pm}}$ on $M_{\sigma_{\pm}}(v)$.

There are several possible phenomena at the wall, depending on the codimension of the locus
of strictly $\sigma_0$-semistable objects, and depending on whether there are curves $C \subset
M_{\sigma_{\pm}}(v)$ of $S$-equivalent objects with respect to $\sigma_0$, i.e., curves with
$\ell_{\sigma_0,\EE_{\pm}}.C = 0$. We call the wall $W$
\begin{enumerate}
\item a \emph{fake wall} there are no curves in $M_{\sigma_{\pm}}(v)$ of objects that are $S$-equivalent
to each other with respect to $\sigma_0$,
\item a \emph{totally semistable wall}, if $M^s_{\sigma_0}(v)=\emptyset$,
\item a \emph{flopping wall}, if $W$ is not a fake wall and
$M^s_{\sigma_0}(v)\subset M_{\sigma_{\pm}}(v)$ has complement of codimension at least two, 
\item a \emph{bouncing wall}, if there is an isomorphism $M_{\sigma_{+}}(v) \cong M_{\sigma_{-}}(v)$ that maps
$\ell_{\sigma_0,\EE_{+}}$ to $\ell_{\sigma_0,\EE_{-}}$, and 
there are divisors $D_\pm \subset M_{\sigma_{\pm}}(v)$ that are covered by curves of
objects that are $S$-equivalent to each other with respect to $\sigma_0$.
\end{enumerate}
Note that a wall can be both fake and totally semistable.
In the case of a fake wall, $W$ does not get mapped to a wall of the nef cone. In the case
of a bouncing wall, the map $l_{+} \colon \overline{\CC}_{+} \to N^1(M_{\sigma_{+}}(v))$ sends $W$ to a boundary
of the nef cone of $M_{\sigma_{+}}(v) = M_{\sigma_{-}}(v)$; and so does $l_{-}$. Hence the image of a path crossing the wall $W$ under
$l_{\pm}$ will bounce back into the ample cone once it hits the boundary of the nef cone
in $N^1$.  We will see examples of every type of wall in Section \ref{sec:Hilbert}.

We should point out that the behavior at fake walls and bouncing walls can exhibit different
behaviors than the possibilities observed in \cite{Alastair-Ishii} in a different context: in general,
the two universal families over $M_{\sigma_+}(v) , M_{\sigma_-}(v)$ do not seem to be related via a derived autoequivalence
of the moduli space $M_{\sigma_+}(v) = M_{\sigma_-}(v)$.

We can assume that $\sigma_0$ is algebraic, $Z_0(v)=-1$, and $\phi=1$.
By Theorem \ref{thm:Yoshioka} and Lemma \ref{lem:PositiveInters}, $\ell_{\sigma_0,\EE_{\pm}}$ has positive self-intersection.
Since both $M_{\sigma_{\pm}}(v)$ have trivial canonical bundles, we can apply the Base Point Free Theorem
\cite[Theorem 3.3]{KollarMori}, which shows that $\ell_{\sigma_0,\EE_{\pm}}$ are both semi-ample.

We denote the induced contraction morphism (cf.~\cite[Theorem 2.1.27]{Laz:Positivity1}) by
\[
\pi_{\sigma_{\pm}} \colon M_{\sigma_{\pm}}(v) \to Y_{\pm},
\]
where $Y_{\pm}$ are normal irreducible projective varieties.
We denote the induced ample divisor classes on $Y_{\pm}$ by $\ell_{0,\pm}$.
Note that $\pi_{\sigma_{\pm}}$ is an isomorphism if and only if the wall $W$ is a fake wall,
a divisorial contraction if $W$ is a bouncing wall, and a small contraction if $W$ is a flopping wall.

%%We write, $w_{\sigma}:=\Im(\eta(\sigma))$, as in Lemma \ref{lem:PositiveInters}..

We would like to say that $Y_{+} = Y_-$, and that they are (an irreducible component of) the coarse
moduli space of $\sigma_0$-semistable objects. The best statement we can prove in general is the
following:

\begin{Prop}\label{prop:CoarseModuli}
The spaces $Y_{\pm}$ have the following universal property: For any proper irreducible scheme
$S$ over $\C$, and for any family $\EE\in\MMM_{\sigma_0}(v)(S)$ such that there exists a
closed point $s\in S$ for which $\EE_s=\EE |_{\{s\}\times X}\in\MMM_{\sigma_{\pm}}(v)(\C)$, there
exists a finite morphism $q\colon T\to S$ and a natural morphism $f_{q^*\EE}\colon T\to Y_{\pm}$.
\end{Prop}

\begin{Prf}
We prove the statement only for $Y_+$; the proof for $Y_{-}$ is analogous.
Let $S$ be a proper scheme, and let $\EE$ be a family as above.
We can assume $S$ is normal.
By Toda's result, Theorem \ref{thm:TodaThmA}, there exists an open subset $S'\subseteq S$ such that $\EE_s$ is $\sigma_{+}$-stable, for all $s\in S'$.
By the universal property for $M_{\sigma_{+}}(v)$, there exists a natural morphism $f_{\EE}'\colon S'\to M_{\sigma_+}(v)$.
This induces a rational morphism $f_{\EE}\colon S \dashrightarrow Y_{+}$.

Consider a resolution of singularities for $f_{\EE}$,
\[
\xymatrix{& \widetilde{S}\ar[dl]_{c}\ar[dr]^{g} & \\
S\ar@{-->}[rr]^{f_{\EE}} && Y_{+}.
}
\]
Then, the family $\widetilde{\EE}\colon =(c\times\id)^*\EE$ on $\widetilde{S}$ gives rise to a divisor class $\ell_{\sigma_0,\widetilde{\EE}}$ on $\widetilde{S}$ such that
\[
\ell_{\sigma_0,\widetilde{\EE}} = g^*\ell_{0,+}.
\]
Since $\ell_{0,+}$ is ample, $\ell_{\sigma_0,\widetilde{\EE}}$ is semi-ample.
On the other hand, by Theorem \ref{thm:nefdivisor}, a curve $C\subseteq\widetilde{S}$ satisfies $\ell_{\sigma_0,\widetilde{\EE}}.C=0$ if and only if $C$ parameterizes properly $\sigma_0$-semistable objects, generically with the same Jordan-H\"older filtration.
But every curve in a fiber of $c$ has this property.
Hence, up to considering its Stein factorization, the morphism $g$ factorizes through $f_{\EE}$, as wanted.
\end{Prf}

If we can explicitly describe $\sigma_0$-semistable objects, Proposition \ref{prop:CoarseModuli}
shows that $Y_{+}$ and $Y_{-}$ are actually irreducible components of a coarse moduli space for
$\MMM_{\sigma_0}(v)$.  We will see this in some examples in Sections \ref{sec:K3sheaves} and
\ref{sec:Hilbert}.

\begin{Prf} (Theorem \ref{thm:MMP})
It remains to prove assertion \eqref{enum:flop} of Theorem \ref{thm:MMP}:
in this case, $\MMM_{\sigma_0}^s(v)$ is non-empty, and we can restrict to the case where
$\ell_{\sigma_0,\EE_{\pm}}$ is not ample.
By openness of stability, all objects in $\MMM_{\sigma_0}^s(v)(\C)$ are stable with respect to $\sigma_{\pm}$.
Write $M_{\sigma_{\pm}}^0(v)$ for the open subsets of $M_{\sigma_{\pm}}(v)$ consisting of those objects.
By assumption, we also have $\mathrm{codim}(M_{\sigma_{\pm}}^0(v),M_{\sigma_{\pm}}(v))\geq2$. (Note that since
$M_{\sigma_{\pm}}(v)$ are smooth and symplectic, the two conditions
$\mathrm{codim}(M_{\sigma_{+}}(v) \setminus M_{\sigma_{+}}^0(v),M_{\sigma_{+}}(v))\geq2$ and
$\mathrm{codim}(M_{\sigma_{-}}(v) \setminus M_{\sigma_{-}}^0(v),M_{\sigma_{-}}(v))\geq2$ are equivalent.)

Consider the birational map
\[
f_{\sigma_0} \colon M_{\sigma_{+}}(v) \dashrightarrow M_{\sigma{-}}(v)
\]
induced by the isomorphism $M^0_{\sigma_+}(v)\xrightarrow{\sim}M^0_{\sigma_-}(v)$.

Since $\mathrm{codim}(M_{\sigma_{\pm}}(v) \setminus M_{\sigma_{\pm}}^0(v),M_{\sigma_{\pm}}(v))\geq2$, and since $M_{\sigma_{\pm}}(v)$ are projective,
numerical divisor classes on $M_{\sigma_{\pm}}(v)$ are determined by their intersection numbers with curves
contained in $M_{\sigma_{\pm}}^0(v)$.
Since we can choose (quasi-)universal families $\EE_{\pm}$ on $M_{\sigma_{\pm}}(v)$ that agree 
on the open subset $M_{\sigma_{\pm}}^0(v)$, this implies that the maps
$l_{\pm} \colon \overline{\CC}^{\pm} \to \NS(M_{\sigma_{\pm}}(v))$ are identical, up to analytic continuation and
identification of the N\'eron-Severi groups via $f_{\sigma_0}$;
more precisely, we have the following equality in $\NS(M_{\sigma_{+}}(v))$:
\begin{equation} \label{eq:flopcomparison}
f_{\sigma_0}^*\ell_{\sigma_{-},\EE_-} = \ell_{Z_{-},\EE_+},
\end{equation}
where the RHS is given by
\[
\ell_{Z_{-},\EE_+}\colon  [C]\mapsto \Im \left(-\frac{Z_{-}(\Phi_{\EE_{+}}(\OO_C))}{Z_{-}(v)}\right),
\]
for all curves $C\subset M_{\sigma_{+}}(v)$.
Since $\ell_{\sigma_0,\EE_+}$ is not ample, $\ell_{Z_{-},\EE_+}$ is big and not nef.
Hence, the map $f_{\sigma_0}$ does not extend to an isomorphism $M_{\sigma_+}(v)\xrightarrow{\sim}
M_{\sigma_-}(v)$. On the other hand, the comparison \eqref{eq:flopcomparison} implies
\[
f_{\sigma_0}^*\ell_{\sigma_0,\EE_-} = \ell_{\sigma_0,\EE_+}.
\]
As a consequence,  we have $Y_+=Y_{-}$, and the following diagram
commutes:
\begin{equation*}%\label{eq:CommDiagram}
\xymatrix{ M_{\sigma_{+}}(v)\ar@{-->}[rr]^{f_{\sigma_0}}\ar[dr]_{\pi_{\sigma_+}} && M_{\sigma_{-}}(v)\ar[dl]^{\pi_{\sigma_-}}\\
& Y_+=Y_{-} &
},
\end{equation*}
\end{Prf}

\section{Stable sheaves on K3 surfaces}
\label{sec:K3sheaves}

In this section we discuss the three main theorems for moduli space of
stable sheaves on a K3 surface $X$; for surfaces with Picard group of rank one, some of these examples can also be
deduced by \cite[Section 4.3]{MYY2}. We will see that our results, when combined with well-known
methods for explicit wall-crossing computations, already capture a large amount of their geometry.
The section is organized as follows: after providing some auxiliary results, we
discuss the relation to Lagrangian fibrations; we then study moduli spaces of vector bundles in
general, and with complete results for some rank 2 cases; finally, we give a general bound for the
ample cone in terms of the Mukai lattice. We start with by recalling the simplest possible case:

\begin{Ex}\label{ex:SkyscrK3surfaces}
The simplest case is a primitive vector $v$ with $v^2=0$.
Then Lemma \ref{lem:isotropic} and Theorem \ref{thm:BoundaryU} give a complete picture of the possible
wall-crossing phenomena.  For a generic stability condition, the moduli space is a fixed smooth
projective K3 surface $Y$ with a Brauer class $\alpha$, such that $(Y, \alpha)$ is derived equivalent to $X$. The possible walls are derived equivalent
to the cases given in Theorem \ref{thm:BoundaryU}: In the cases $(A^+)$ and $(A^-)$, we have a totally
semistable fake wall. In the case $(C_k)$, we get a bouncing wall: the contraction induced by the
wall is the divisorial contraction of rational $(-2)$-curves. After we cross the wall, the moduli
space is still isomorphic to $Y$, but the universal family gets modified by applying the spherical
twist at a line bundle supported on $C$; in $\NS(Y)$, this has the effect of a reflection at
$[C]$.
\end{Ex}

\subsection*{Auxiliary results}
We first give an explicit formula for the Mukai vector $w_{\sigma}$ associated to a stability condition.

\begin{Lem}\label{lem:ExplicitFormula}
Let $X$ be a smooth projective K3 surface.
Let $v=(r,c,s)\in H^*_{\alg}(X, \Z)$ be a primitive Mukai vector with $v^2 \ge -2$, and
let $\sigma_{\omega,\beta}\in U(X)$ be a generic stability condition with respect to $v$.
Then the divisor class $\ell_{\sigma_{\omega,\beta}}\in N^1(M_{\sigma_{\omega,\beta}}(v))$ is a
positive multiple of $\theta_v(w_{\sigma_{\omega,\beta}})$, where
$w_{\sigma_{\omega,\beta}}=(R_{\omega,\beta}, C_{\omega,\beta}, S_{\omega,\beta})$ is given by
\begin{align*}
& R_{\omega,\beta}= c.\omega-r\beta.\omega\\
& C_{\omega,\beta}= \left(c.\omega-r\beta.\omega\right)\beta + \left(s-c.\beta + r\,\frac{\beta^2-\omega^2}{2} \right) \omega\\
& S_{\omega,\beta}= c.\omega\, \frac{\beta^2-\omega^2}{2} + s \beta.\omega - (c.\beta) \cdot (\beta.\omega).
\end{align*}
\end{Lem}

\begin{Prf}
Using the Definition of $\ell_{\sigma_{\omega,\beta}}$ in equation \eqref{eq:firstellDef}, and
the compatibility of $\theta_v$ with the Mukai pairing given in equation \eqref{eq:MukaiMukai}, we see
that the vector is given by
\[
w_{\sigma_{\omega, \beta}} = \Im \frac{e^{i\omega + \beta}}{-(e^{i\omega+ \beta}, v)}
\sim_{\R^+} -\Im \bigl(\overline{(e^{i\omega + \beta}, v)} \cdot e^{i\omega + \beta}\bigr).
\]
(Here and in the following $\sim_{\R^+}$ will mean that the vectors are positive scalar multiples of
each other.) Then the claim follows immediately from
\[
e^{i\omega+ \beta} = \left(1, \beta, \frac{\beta^2-\omega^2}2\right) + i\left(0, \omega,
\omega.\beta\right).
\]
\end{Prf}

If we write $\omega = t\cdot H$, for an ample divisor $H\in\mathrm{NS}(X)$, we can let $t$ go to zero or $+\infty$.
If we take the limit $t\to 0$ up to rescaling, we obtain a vector $w_{0\cdot H,\beta}$ with components
\begin{align*}
& R_{0\cdot H,\beta}= c.H-r\beta.H\\
& C_{0\cdot H,\beta}= \left(c.H-r\beta.H\right)\beta + \left(s-c.\beta + r\,\frac{\beta^2}{2} \right) H\\
& S_{0\cdot H,\beta}= c.H\, \frac{\beta^2}{2} + s \beta.H - (c.\beta) \cdot (\beta.H).
\end{align*}

If we similarly take the limit $t\to +\infty$, we obtain a
vector $w_{\infty\cdot H,\beta}$ with components
\begin{align*}
& R_{\infty\cdot H,\beta}= 0\\
& C_{\infty\cdot H,\beta}= -r\, \frac{H^2}{2} H\\
& S_{\infty\cdot H,\beta}= -c.H\,\frac{H^2}{2}.
\end{align*}

We will also use the following two observations several times; for Lemma \ref{lem:MinimalImPart} see, e.g., \cite[Section 7.2]{BMT:3folds-BG}; for Lemma \ref{lem:localP2} see \cite[Lemma 5.9]{localP2}.

\begin{Lem}\label{lem:MinimalImPart}
Let $\sigma=(Z,\AA)\in\Stab(X)$ be a stability condition such that
\[
\gamma:=\mathrm{inf}\left\{ \Im(Z(E))>0\colon E\in\AA\right\}>0.
\]
Then an object $E\in\AA$ with $\Im(Z(E))= \gamma$ is $\sigma$-stable if and only if $\Hom(\PP(1),E)=0$.
\end{Lem}

The previous lemma applies in particularly when $\Im(Z)\in \Z \cdot \gamma$, for some constant $\gamma>0$.
In this case, if an object $E\in\AA$ with $\Hom(\PP(1),E)=0$ and $\Im(Z(E))=2\gamma$ is not $\sigma$-stable,
then it must be destabilized by a short exact sequence $A\to E\to B$ where $A$ and $B$ are $\sigma$-stable with $\Im(Z)=\gamma$.

\begin{Lem}\label{lem:localP2}
Let $E \in \Db(X)$ and $\sigma \in \Stab(X)$ be a stability condition such that 
$E$ is $\sigma$-semistable.
Assume that there is a Jordan-H\"older filtration
$M^{\oplus r} \into E \onto N$ of $E$ such that $M, N$ are $\sigma$-stable, $\Hom(E, M) = 0$, and
$[E]$ and $[M]$ are linearly independent classes in $K_{\num}(X)$.
Then $\sigma$ is in the closure of the set of stability conditions where $E$ is stable.
\end{Lem}

\subsection*{Lagrangian fibrations}

\begin{Ex}\label{ex:rk0}
In the previous notation, assume $v=(0,c,s)$, for $v$ a primitive vector, $v^2\geq0$, and $c$ a non-zero effective divisor.
We assume that $H$ is a generic polarization with respect to $v$.
Then, by Theorem \ref{thm:BridgelandToda}, for $t$ sufficiently large, $M_{\sigma_{t\cdot H,\beta}}(v)=M_H^\beta(v)$ is a Lagrangian fibration.
The semi-ample nef divisor associated to this fibration is given by
$w_{\infty\cdot H,\beta} \sim_{\R^+} (0,0,-1)$.

The fact that $M_H^\beta(v)$ is a Lagrangian fibration can be seen by using the divisor $\theta_v(w_{\infty\cdot H,\beta})$ as follows.
By Le Potier's construction, see \cite[Section 1.3]{LePotier:StrangeDuality}, for all $x\in X$, we can construct a section $s_x\in H^0(M_H^\beta(v),\theta_v(w_{\infty\cdot H,\beta}))$ via its zero-locus
\[
V(s_x) = \stv{E\in M_H^\beta(v)}{\Hom(E,k(x))\neq0}.
\]
For any $x$ not contained in the support of $E$, the section $s_x$ does not vanish at the point $[E]
\in M_H^\beta(v)$; therefore $\theta_v(w_{\infty\cdot H,\beta})$ is globally generated by the sections $\{s_x\}_{x\in X}$.
The induced morphism contracts the locus of sheaves with fixed support, and thus the image has lower
dimension.  By Matsushita's Theorem \cite{Matsushita:Fibrations,Matsushita:Addendum}, the morphism
is a Lagrangian fibration.
\end{Ex}

\begin{Rem}
The previous example shows a general phenomenon for nef divisors obtained as an image of a wall in the space of Bridgeland stability conditions.
Indeed, by Lemma \ref{lem:PositiveInters} and Theorem \ref{thm:Yoshioka}, a divisor $D$ coming from a wall in $\Stab(X)$ must have $q(D)>0$.
To obtain a nef divisor $D$ with $q(D)=0$ (which conjecturally corresponds to a Lagrangian
fibration),
we necessarily have to look at ``limit points'' in $\Stab(X)$, for example $w_{0\cdot H,\beta}$, or $w_{\infty\cdot H,\beta}$.
We will use these limit points in Examples \ref{ex:Sawon} and \ref{ex:AB}.
\end{Rem}

\subsection*{Nef cones} In this subsection, we will use wall-crossing for moduli spaces of vector
bundles; in particular, in Example \ref{ex:Rank2}, we will see that for rank two,
our main Theorems with an explicit wall-crossing analysis can be sufficient to determine the nef
cone of the moduli space.

\begin{Ex}\label{ex:Dragos}
Let $X$ be a K3 surface with $\mathrm{Pic}(X)=\Z\cdot H$, for $H$ an ample line bundle with
$H^2=2d$, $d\geq1$.
Let $v=(r,cH,s)$ be a primitive Mukai vector, with $r,c,s\in\Z$, $r\geq 0$, $v^2 \ge -2$.
We assume that there exist $A,B\in\Z$, $A>0$, such that $Ac-Br=1$.

Consider the family of stability conditions $\sigma_{t, \frac BA}:=\sigma_{\omega,\beta}$ on $\Db(X)$, with $\omega = t\cdot H$ and $\beta := \frac{B}{A}\cdot H$, for $t>0$.
As long as $\sigma_{t, \frac BA}$ exists, the moduli space $M_{\sigma_{t, \frac BA}}(v)$ is the moduli space of Gieseker
stable sheaves $M_H(v)$: Indeed, we have
\[
\Im(Z_{t, \frac BA}(\blank))\in \frac{2td}{A}\cdot\Z,
\]
and $\Im(Z_{t, \frac BA}(v))=\frac{2td}{A}$.
So Lemma \ref{lem:MinimalImPart} shows that Gieseker-stable sheaves are $\sigma_{t, \frac BA}$-stable.

We distinguish two cases, according to whether $\frac{dB^2+1}{A}$ is integral or not.
Its relevance is explained by the fact that $w= (A, B\cdot H, \frac{dB^2+1}{A}) \in H^*_{\alg}(X)_\Q$
is a class with $w^2 = -2$ and $\Im Z_{t, \frac BA}(w) = 0$; since
$A, B$ are coprime, there exists an integral class with these two properties 
if and only if $\frac{dB^2+1}{A}$ is integral. 

{\bf Case 1:} $\frac{dB^2+1}{A}\notin\Z$.
Then there exists no spherical object with $\Im(Z_{t, \frac BA}(v))=0$.
By \cite[Proposition 7.1]{Bridgeland:K3}, all values of $t>0$ produce a stability condition.
This gives an explicit region of the ample cone of $M_H(v)$:
\[
\langle \theta_v(w_{\sigma_{t, \frac BA}})\colon t>0\rangle \subset \mathrm{Amp}(M_H(v)).
\]
An explicit computation is in Example \ref{ex:Sawon}.

{\bf Case 2:} $\frac{dB^2+1}{A}\in\Z$.
Then there exists a stable spherical vector bundle $U$ satisfying $\Im(Z_{t, \frac BA}(U))=0$.
We let $t_0>0$ be such that $\Re(Z_{t_0}(U))=0$.
Then $t>t_0$ produces a line segment in the ample cone of $M_H(v)$:
\[
\langle \theta_v(w_{\sigma_{t, \frac BA}})\colon t>t_0\rangle \subset \mathrm{Amp}(M_H(v)).
\]
The question now becomes to understand when $\Hom(U,F)\neq0$, for $F$ a Gieseker stable sheaf with Mukai vector $v$.
An explicit computation is in the following example.
\end{Ex}

\begin{Ex}\label{ex:Rank2}
In the notation of the previous Example \ref{ex:Dragos}, we take
\begin{equation*}
d=1,\quad v=(2,H,s)\,\,\, (s\leq0),\quad A=1,\quad B=0.
\end{equation*}
Then, the spherical vector bundle $U$ is nothing but $\OO_X$, and $t_0 = 1$.
Up to rescaling, the vector $w_{\sigma_{t, \frac BA}}$ becomes
\[
w_{\sigma_{t, 0}} = \left(2t,(-2t^3+st)H,-2t^3 \right).
\]
We will see that wall-crossing along this path will naturally lead to contractions of
Brill-Noether loci, i.e., loci of sheaves $F$ where $h^0(F)$ is bigger than expected.
These loci and contractions have been studied in \cite{Yoshioka:BN_for_sheaves}. 
We distinguish 3 cases.

{\bf Case 1:} $s=0$. We claim that the nef cone $\mathrm{Nef}(M_H(v))$ is generated by
\[
\theta_v(w_{0\cdot H,0}) \simpos \theta_v(1,0,0)\quad \text{ and }\quad \theta_v(w_{\infty\cdot H,0})
\simpos \theta_v(0,-H,-1).
\]
First of all, observe that any torsion sheaf $T \in \MMM_H(0, H, -2)(\C)$ is a line bundle of 
degree -1 on a curve of genus 2; it follows that there is a short exact sequence
\[
0 \to \OO_X^{\oplus 2} \to \ST_{\OO_X}^{-1}(T) \to T \to 0.
\]
It easy to see that $F := \ST_{\OO_X}^{-1}(T)$ is slope-stable with $v(F) = (2, H, 0)$; hence
$\ST_{\OO_X}^{-1}$ induces an injective morphism $M_H(0, H, -2) \to M_H(v)$,
which must be an isomorphism (as they have the same dimension).
Hence every $F \in \MMM_H(v)(\C)$ is of this form, and
 $\Hom(\OO_X,F) = \C^2$, for all $F\in \MMM_H(v)(\C)$.

To compute how the divisor class $\ell_{\sigma_{t, 0}}$ varies when we cross $t=1$, we will use
Lemma \ref{lem:localP2}.
For $0<t<1$, we consider the stability condition $\overline{\sigma}_{t,0}$ in the boundary of $U(X)$ of type $(A^+)$ (see Theorem \ref{thm:BoundaryU}).
The heart $\AA$ for $\overline{\sigma}_{t,0}$ can be explicitly described (see, e.g.,
\cite[Proposition 2.7]{Yoshioka:StabilityFM} or \cite[Proposition 5.6]{localP2}).
In particular, $\PP(1)$ is generated by $k(x)$ for $x \in X$, by $\OO_X$, and by all objects of the
form $G[1]$, where $G$ is any $\mu$-semistable sheaves of slope $0$ satisfying $\Hom(\OO_X, G) = 0$.
Hence, both $\OO_X$ and, by Lemma \ref{lem:MinimalImPart}, any $T \in M_H(0, H, -2)$
are $\overline{\sigma}_{t, 0}$-stable for all $0<t<1$.
Similarly, the short exact sequence
\[
0 \to T \to \ST_{\OO_X}(T) \to \OO_X^{\oplus 2} \to 0
\]
and Lemma \ref{lem:MinimalImPart} show that $\ST_{\OO_X}(T) = \ST_{\OO_X}^2(F)$ is 
$\overline{\sigma}_{t, 0}$-stable for all $0<t<1$.

In particular, $M_{\overline{\sigma}_{t, 0}}(v)$ for $0 < t < 1$ is isomorphic to
$M_H(v) = M_{\sigma_{t, 0}}(v)$ for $1 < t$. The universal families are related by an application
of $\ST_{\OO_X}^2$; as this acts trivially on the $K$-group, the two families induces the
same Mukai homomorphism $v^\perp \to N^1(M_H(v))$.

To understand the wall between the two corresponding chambers, 
we now consider the path $\sigma_{t, -\epsilon}$, where $\epsilon > 0$ is sufficiently small such
that $\OO_X$ and all $T \in \MMM_H(0, H, -1)(\C)$ are both
$\sigma_{\frac 12, -\epsilon}$-stable and $\sigma_{2, -\epsilon}$-stable. Note that the subcategory
$\AA_{t, -\epsilon}$ does not depend on $t$; it is then straightforward to check that 
$\OO_X$ and all $T$ are also $\sigma_{t, -\epsilon}$-stable for all $t \in [\frac 12, 2]$: indeed,
the imaginary part of $Z_{t, -\epsilon}(w)$ for any Mukai vector $w$ is of the form
$t \cdot \mathrm{const}$, and the real part is of the form $\mathrm{const} +\mathrm{const}
\cdot t^2$. Then the inequality $\phi_{t, -\epsilon}(w) \le \phi_{t, -\epsilon}(w')$ is equivalent
to an equation of the form $\mathrm{const} \cdot t^2 \ge \mathrm{const}$.

Let $t_0 \in [\frac 12, 2]$ be such that $\OO_X$ and $T \in \MMM_H(0, H, -1)(\C)$ have the same phase
with respect to $\sigma_{t_0, -\epsilon}$.
Lemma \ref{lem:localP2} shows that $F = \ST_{\OO_X}^{-1}(T)$ is stable for $t > t_0$,
and that $\ST_{\OO_X}^2(F) = \ST_{\OO_X}(T)$ is stable for $t < t_0$.  This is a totally semistable
and fake wall.

For $t\to0$, the contraction induced by $w_{0\cdot H,0}$ is precisely the Jacobian fibration induced
by $\ST_{\OO_X}$.  The wall at $\beta=1/2\cdot H$ corresponds instead to the Uhlenbeck
compactification: the corresponding divisorial contraction is induced precisely by $w_{\infty\cdot
H,0}$ (see also \cite{Jason:Uhlenbeck}).

{\bf Case 2:} $s=-1$. The nef cone $\mathrm{Nef}(M_H(v))$ is generated by
\[
\theta_v(w_{H,0})=\theta_v(2,-3H,-2)\quad \text{ and }\quad \theta(w_{\infty\cdot H,0})=\theta_v(0,-H,-1).
\]

Similarly to Case 1, the Riemann-Roch Theorem and stability show $\Hom(\OO_X,F)\neq0$,
for all $F\in \MMM_H(v)(\C)$.
We can use a similar argument as before to find a wall near the singular point
$\sigma_{1,0}$ where the Jordan-H\"older filtration of $F$ is given by
\[
H^0(F) \otimes \OO_X \xrightarrow{\ev} F \to \cone(\ev).
\]
There is no stable object with Mukai vector $v$ with respect to $\sigma_{t_0,\epsilon}$, hence we are still in the case of a totally semistable wall.
Unlike in the previous case, we do have curves of $S$-equivalent objects that get contracted by
$w_{H, 0}$: there is a $\P^1$ parameterizing extensions 
\begin{equation}\label{eq:Columbus020512}
0\to \OO_X\to F\to I_\Gamma(H)\to 0,
\end{equation}
for any zero-dimensional subscheme $\Gamma\subset X$ of length $4$ contained in a curve
$C\in |H|$.

{\bf Case 3:} $s\leq -2$. The nef cone $\mathrm{Nef}(M_H(v))$ is generated by
\[
\theta_v(w_{H,0})=\theta_v(2,(-2+s)H,-2)\quad \text{ and }\quad \theta(w_{\infty\cdot
H,0})=\theta_v(0,-H,-1).
\]

Indeed, in this case, we will always have both stable objects at $\sigma_{t_0,\epsilon}$ (by a dimension count), and strictly semistable ones (corresponding to extensions as in \eqref{eq:Columbus020512}, with $\Gamma\subset X$ of length $3-s$).
\end{Ex}

\subsection*{General bound for the ample cone}
Finally, we proceed to give an explicit bound for the walls of the ``Gieseker chamber'' for
any Mukai vector $v$, i.e., the chamber for which Bridgeland stability of objects of class $v$ is
equivalent to $\beta$-twisted Gieseker stability. In principle, this has been well-known, as all the
necessary arguments are already contained in \cite[Proposition 14.2]{Bridgeland:K3}; see also 
\cite[Section 6]{Toda:K3}, \cite[Proposition 4.1]{large-volume}, \cite[Section 2]{Minamide-Yanagida-Yoshioka:wall-crossing},
\cite[Theorem 4.4]{LoQin:miniwalls}; the most explicit results can be found in
\cite[Sections 2 and 3]{Maciocia:walls} (with regards to a slightly different form of the central
charge) and \cite{Kawatani:Gieseker_vs_Bridgeland}; what follows is essentially a short summary of
Kawatani's argument. Corollary \ref{Cor:amplecone} deduces a general bound for the ample cone from
this analysis.

We want to give a bound that is as explicit as possible for the form of the central
charge given in \eqref{eq:ZK3}. 
Fix a class $\beta \in \NS(X)_\Q$, and let $\omega$ vary on a ray in the ample cone.
Given a class $v \in H^*_\alg(X,\Z)$ with positive rank and slope, Bridgeland and Toda showed
that for $\omega \gg 0$, stable objects of class $v$ are exactly the twisted-Gieseker stable
sheaves, see Theorem \ref{thm:BridgelandToda}.
We want to give an explicit bound in terms of $\omega^2$ and $\beta, v$ that only depends
on the Mukai lattice $H^*_\alg(X, \Z)$.

\begin{Def} \label{def:discrepancy}
Given divisor classes $\omega, \beta$ with $\omega$ ample, 
and given a class $v = (r, c, s) \in H^*_\alg(X,\Z)$ with $v^2 \ge -2$, we 
write $(r, c_\beta, s_\beta) = e^{-\beta} (r, c, s)$ and define its
\emph{slope} $\mu_{\omega, \beta}(v) = \frac{\omega.c_\beta}r$ as in Section \ref{sec:ReviewGieseker}, equation \eqref{eq:muomegabeta}, and its
\emph{discrepancy} $\delta_{\omega, \beta}(v)$ by
\begin{align} 
\delta_{\omega, \beta}(v)
& = -\frac {s_\beta}r + 1 + \frac 12 \frac{\mu_{\omega, \beta}(v)^2}{\omega^2}
\end{align}
\end{Def}

Observe that rescaling $\omega$ will rescale $\mu_{\omega, \beta}$ by the same factor, while
leaving $\delta_{\omega, \beta}$ invariant.  A torsion-free
sheaf $F \in \Coh X$ is $\beta$-twisted Gieseker stable if
for every subsheaf $G \subset F$ we have
\begin{align*}
\mu_{\omega, \beta}(G) &\le \mu_{\omega, \beta}(F), \quad \text{and} \\
\mu_{\omega, \beta}(G) &= \mu_{\omega, \beta}(F) \Rightarrow
\delta_{\omega, \beta}(G) > \delta_{\omega, \beta}(F).
\end{align*}

Combining the Hodge Index theorem with the assumption $v^2 \ge -2$
shows 
\begin{align*} %\label{eq:deltag0}
\delta_{\omega, \beta}(v) &\ge -\frac{s_\beta}r + 1 + \frac{c_\beta^2}{2r^2}
= \frac{v^2 + 2}{2 r^2} + \left(1 - \frac 1{r^2}\right) \ge 0.
\end{align*}

Given a class $v$ with $r > 0$, we can write the central charge of equation
\eqref{eq:ZK3} as
\begin{equation} \label{eq:Zrewrite}
\frac 1r Z_{\omega, \beta}(v) =
i \mu_{\omega, \beta}(v) + \frac{\omega^2}2 - \frac{s_\beta}r
= i \mu_{\omega, \beta}(v) + \frac{\omega^2}2 - 1 - \frac {\mu_{\omega, \beta}(v)^2}{2 \omega^2} +
\delta_{\omega, \beta}(v)
\end{equation}

We now fix a class $v \in H^*_\alg(X, \Z)$ with
$r(v) > 0$ and $\mu_{\omega, \beta}(v) > 0$.
\begin{Lem} \label{lem:bounddelta}
Assume $\omega^2 > 2$.
Any class $w \in H^*_\alg(X, \Z)$
with $r(w) > 0$, $0 < \mu_{\omega, \beta}(w) < \mu_{\omega, \beta}(v)$ such that
the phase of $Z_{\omega, \beta}(w)$ is bigger or equal to the phase of
$Z_{\omega, \beta}(v)$ satisfies
$\delta_{\omega, \beta}(w) < \delta_{\omega, \beta}(v)$.
\end{Lem}
\begin{Prf}
By equation \eqref{eq:Zrewrite}, it is evident that 
decreasing $\delta_{\omega, \beta}(v)$ while keeping $\mu_{\omega, \beta}(v)$ fixed will increase the phase
of the complex number $Z_{\omega, \beta}(v)$. The same equation also shows that objects
with fixed $\delta_{\omega, \beta}$ lie on a parabola, symmetric to the real axis, which intersects
the positive real axis; in particular, increasing $\mu_{\omega, \beta}(v)$ while keeping
$\delta_{\omega, \beta}(v)$ fixed will also increase the phase of $Z_{\omega, \beta}(v)$; see also
fig.~\ref{fig:deltabound}.
\end{Prf}

\begin{figure}
\begin{centering}
\definecolor{zzttqq}{rgb}{0.27,0.27,0.27}
\definecolor{qqqqff}{rgb}{0.33,0.33,0.33}
\definecolor{uququq}{rgb}{0.25,0.25,0.25}
\definecolor{xdxdff}{rgb}{0.66,0.66,0.66}
\begin{tikzpicture}[line cap=round,line join=round,>=triangle 45,x=1.0cm,y=1.0cm]
\draw[->,color=black] (-2.3,0) -- (3.5,0);
\foreach \x in {-2,2}
\draw[shift={(\x,0)},color=black] (0pt,2pt) -- (0pt,-2pt);
\draw[->,color=black] (0,-1.5) -- (0,2.5);
\foreach \y in {,2}
\draw[shift={(0,\y)},color=black] (2pt,0pt) -- (-2pt,0pt);
\clip(-2.3,-1.5) rectangle (3.5,2.5);
\fill[color=zzttqq,fill=zzttqq,fill opacity=0.1] (0,0) -- (0.97,0.84) -- (-2.4,0.84) -- (-2.4,0) --
cycle;
%\draw [samples=50,rotate around={-270:(3,0)},xshift=3cm,yshift=0cm] plot (\x,\x^2/2/0.5);
\draw [samples=50,rotate around={-270:(3,0)},xshift=3cm,yshift=0cm] plot (\x,\x*\x);
\draw (1.54,2.24) node[anchor=north west] {$ \frac 1r Z(v) $};
\draw (1.45,1.25)-- (0,0);
\draw (1.46,-0.04) node[anchor=north west] {$\delta = \textrm{const}$};
\draw [line width=0.4pt,domain=-2.3:3.5] plot(\x,{(--1.25-0*\x)/1});
\draw [domain=-2.3:3.5] plot(\x,{(--0.84-0*\x)/1});
\draw (-0.96,1.88) node[anchor=north west] {$ \mu_{\omega, \beta}(v) $};
\draw (-0.98,1.42) node[anchor=north west] {$ \mu_{\omega, \beta}(w) $};
\draw (-0.96,0.94) node[anchor=north west] {$ \frac 1r Z(w) $};
\draw [color=zzttqq] (0,0)-- (0.97,0.84);
\draw [color=zzttqq] (0.97,0.84)-- (-2.4,0.84);
\draw [color=zzttqq] (-2.4,0.84)-- (-2.4,0);
\draw [color=zzttqq] (-2.4,0)-- (0,0);
\begin{scriptsize}
\fill [color=xdxdff] (1.45,1.25) circle (1.5pt);
\fill [color=uququq] (0,0) circle (1.5pt);
\fill [color=qqqqff] (0.16,0.4) circle (1.5pt);
\end{scriptsize}
\end{tikzpicture}
\caption{Destabilizing subobjects must have smaller $\delta$}
\label{fig:deltabound}
\end{centering}
\end{figure}
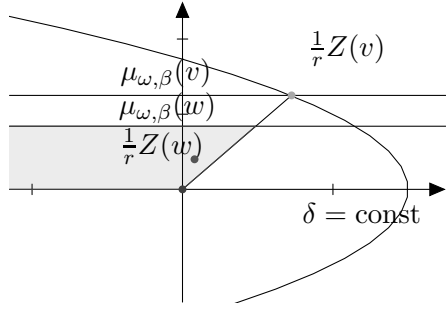

\begin{Def} \label{def:Dv}
Define $D_v \subset H^*_\alg(X, \Z)$ as the subset
\[
\stv{w}
{0 < r(w) \le r(v), w^2 \ge -2, 
0 < \mu_{\omega, \beta}(w) < \mu_{\omega, \beta}(v),
\delta_{\omega, \beta}(w) < \delta_{\omega, \beta}(v)}.
\]
\end{Def}
The set $D_v$ is finite: the Hodge Index theorem and
$r(w)^2\delta_{\omega, \beta}(w) < r(v)^2 \delta_{\omega, \beta}(v)$
bound the norm of the orthogonal projection of $c_\beta(w)$ to $\omega^{\perp} \subset H^{1, 1}_\alg(X)_\R$;
the inequality $0 <  c_{\beta}(w) < r(v) c_{\beta}(v)$ bounds  the projection of $c_\beta(w)$
to $\R \cdot v$; and, finally, $w^2 \ge -2$ and $\delta_{\omega, \beta}(w) < \delta_{\omega,
\beta}(w)$ give bounds for $s_\beta(w)$.
We also observe that $D_v$ does not change when we rescale $\omega$ in the ray
$\R_{>0}\cdot \omega$.

\begin{Def} \label{def:mumax}
We define $\mu^{\max}(v)$ by
\[
\mu^{\max}(v) := \max
\stv{\mu_{\omega, \beta}(w)}{w \in D_v} \cup \left\{ \frac{r(v)}{r(v)+1}\cdot \mu_{\omega, \beta}(v)
\right\}.
\]
\end{Def}

\begin{Lem} \label{lem:Giesekerchamber}
Let $E$ be a $\beta$-twisted Gieseker stable sheaf with $v(E) = v$.
If $\omega^2 > 2 + \frac{2\mu^{\max}(v)}{\mu_{\omega, \beta}(v) - \mu^{\max}(v)}\delta_{\omega,\beta}(v)$,
then $E$ is $Z_{\omega, \beta}$-stable.
\end{Lem}

\begin{Prf}
Consider a destabilizing short exact sequence
$A \into E \onto B$ in $\AA(\omega, \beta)$ with $\phi_{\omega, \beta}(A) \ge \phi_{\omega,
\beta}(E)$.
By the long exact cohomology sequence, $A$ is a sheaf. Consider the
HN-filtration of $A$ with respect to $\mu_{\omega, \beta}$-slope stability in
$\Coh X$, and let $A_1, \dots, A_n$ be its HN-filtration factors.
Since $A \in \AA(\omega, \beta)$ we have $\mu_{\omega, \beta}(A_i) > 0$ for all $i$.
Since the kernel of $A \to E$ lies in $\FF(\omega, \beta)$, we also have
$\mu_{\omega, \beta}(A_i) \le \mu_{\omega, \beta}(A_1) \le \mu_{\omega, \beta}(v)$.

By the see-saw property, we can choose an $i$ such that 
$\phi_{\omega, \beta}(A_i) \ge \phi_{\omega, \beta}(v)$.

First assume $\mu_{\omega, \beta}(A_i) = \mu_{\omega, \beta}(v)$, in which case $i =
1$. Consider the composition $g \colon A_1 \into A \to E$ in $\Coh X$. If $g$ is not injective,
then $\ker g$ has the same slope $\mu_{\omega, \beta}(\ker g) = \mu_{\omega, \beta}(v)$. Since $\ker
g \into A$ factors via $\HH^{-1}(B) \into A$, this is a contradiction to $\HH^{-1}(B) \in \FF(\omega,
\beta)$. However, if $g$ is injective, $A_1 \subset E$ is a subsheaf with
$\mu_{\omega, \beta}(A_1) = \mu_{\omega, \beta}(E)$ and, by assumption 
and equation \eqref{eq:Zrewrite}, 
$\delta_{\omega, \beta}(A_1) \le \delta_{\omega, \beta}(E)$. This contradicts the assumption
that $E$ is $\beta$-twisted Gieseker stable.

We have thus proved $\mu_{\omega, \beta}(A_i) < \mu_{\omega, \beta}(v)$.
Let $w \in H^*_\alg(X, \Z)$ be the primitive
class such that $v(A_i)$ is a positive integer multiple of $w$.
We claim that in fact $\mu_{\omega, \beta}(w) = \mu_{\omega, \beta}(A_i) \le \mu^{\max}(v)$.
In case $r(w) \le r(v)$, this
follows from Lemma \ref{lem:bounddelta} and the definition of the set $D_v$. In case
$r(w) \ge r(v) + 1$, we observe 
\[ 
\omega.c_\beta(w) \le \omega.c_\beta(A_i) = \Im Z (A_i) \le \Im Z(A) \le \Im Z (E) 
= \omega.c_\beta(v)
\]
to conclude $\mu_{\omega, \beta}(w) \le \frac{r(v)}{r(v)+1}\cdot \mu_{\omega, \beta}(v)$.

We conclude the proof with a simple geometric argument, see also fig.~\ref{fig:Gieseker}:
By equation \eqref{eq:Zrewrite}, the phase of $Z(w)$ is less than or equal to the phase $\phi(z)$ of
\[
z:= i \mu^{\max}(v) + \frac {\omega^2}2 - 1 - \frac{\mu^{\max}(v)^2}{2\omega^2}.
\]
We have $\Im \frac{\mu_{\omega, \beta}(v)}{\mu^{\max}(v)} z = \Im \frac 1{r(v)} Z(v)$ and
\begin{align*}
\Re \frac{\mu_{\omega, \beta}(v)}{\mu^{\max}(v)} z &=
\frac {\omega^2}2 - 1 
+ \frac{\mu_{\omega, \beta}(v) - \mu^{\max}(v)}{\mu^{\max}(v)} \left(\frac{\omega^2}2 - 1\right)
- \frac{\mu_{\omega, \beta}(v) \mu^{\max}(v)}{2\omega^2} \\
& > \frac {\omega^2}2 - 1 - \frac {\mu_{\omega, \beta}(v)^2}{2 \omega^2} + \delta_{\omega, \beta}(v)= \Re \frac 1{r(v)} Z(v).
\end{align*}
and thus $\phi(z) < \phi_{\omega, \beta}(v)$. This leads to the contradiction
\[
\phi_{\omega, \beta}(E) \le \phi_{\omega, \beta}(A_i) \le \phi(z) < \phi_{\omega, \beta}(E).
\]
\end{Prf}

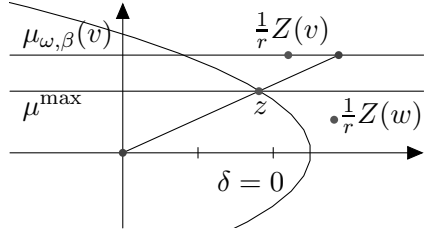
\begin{figure}
\begin{centering}
\definecolor{uququq}{rgb}{0.25,0.25,0.25}
\definecolor{qqqqff}{rgb}{0.33,0.33,0.33}
\begin{tikzpicture}[line cap=round,line join=round,>=triangle 45,x=1.0cm,y=1.0cm]
\draw[->,color=black] (-1.5,0) -- (4,0);
\foreach \x in {,2}
\draw[shift={(\x,0)},color=black] (0pt,2pt) -- (0pt,-2pt);
\draw[->,color=black] (0,-1) -- (0,2);
\foreach \y in {}
\draw[shift={(0,\y)},color=black] (2pt,0pt) -- (-2pt,0pt);
\clip(-1.5,-1) rectangle (4,2);
\draw [samples=50,rotate around={-270:(2.5,0)},xshift=2.5cm,yshift=0cm] plot (\x,\x*\x);
\draw [domain=-1.5:4] plot(\x,{(--1.3-0*\x)/1});
\draw (-1.47,1.9) node[anchor=north west] {$\mu_{\omega, \beta}(v)$};
\draw (1.12,-0.1) node[anchor=north west] {$\delta = 0$};
\draw [domain=-1.5:4] plot(\x,{(--0.82-0*\x)/1});
\draw (1.60,0.82) node[anchor=north west] {$z$};
\draw (-1.47,0.9) node[anchor=north west] {$\mu^{\textrm{max}}$};
\draw (2.87,1.3)-- (0,0);
\draw (2.75,0.86) node[anchor=north west] {$\frac 1r Z(w)$};
\draw (1.60,2.0) node[anchor=north west] {$\frac 1r Z(v)$};
\begin{scriptsize}
\fill [color=qqqqff] (2.2,1.3) circle (1.5pt);
\fill [color=qqqqff] (1.81,0.82) circle (1.5pt);
\fill [color=uququq] (0,0) circle (1.5pt);
\fill [color=uququq] (2.87,1.3) circle (1.5pt);
\fill [color=qqqqff] (2.81,0.44) circle (1.5pt);
\end{scriptsize}
\end{tikzpicture}
\caption{Phases of $Z(w)$, $z$ and $Z(v)$}
\label{fig:Gieseker}
\end{centering}
\end{figure}

\begin{Cor} \label{Cor:amplecone}
Let $v \in H^*_{\alg}(X, \Z)$ be a primitive Mukai vector with $v^2\geq 2$.
Let $\omega, \beta \in \NS(X)_\Q$ be generic with respect to $v$, and such that
$r(v) > 0$ and $\omega.c_\beta(v) > 0$.
Let $M_\omega^\beta(v)$ be the moduli space of $\beta$-twisted Gieseker
stable sheaves. If $\mu^{\max}(v)$ is as given in Definition \ref{def:mumax}
and $\omega$ satisfies 
$\omega^2 > 2 + \frac{2\mu^{\max}(v)}{\mu_{\omega, \beta}(v) - \mu^{\max}(v)}\delta_{\omega,\beta}(v)$, then
\[
\theta_v(w_{\sigma_{\omega, \beta}}) \subset \mathrm{Amp} (M_\omega^\beta(v)).
\]
\end{Cor}
This bound does not depend on the specific K3 surface $X$, other than on its lattice
$H^*_{\alg}(X, \Z)$. We could also make it completely independent of $X$ by considering the
full lattice $H^*(X, \Z)$ instead of $H^*_{\alg}(X, \Z)$ in Definition \ref{def:Dv}.

\section{Hilbert scheme of points on K3 surfaces}
\label{sec:Hilbert}

In this section, we study the behavior of our nef divisor
at walls for the Hilbert scheme of points on a K3 surface of Picard rank one. Its walls in $\Stab(X)$
 have been described partly in \cite[Section 3]{Aaron-Daniele}, and many of the arguments
are identical to the case of $\P^2$ treated in \cite{ABCH:MMP}.

We will see that our nef divisor can be deformed in one direction to recover the Hilbert-Chow
morphism in Example \ref{ex:HilbChow}. In the other direction, we study the first wall
in Example \ref{ex:firstwall}, leading to new results on the Mori cone,
see Proposition \ref{prop:Hilbnefcone} and Remark \ref{rem:HT_ER_conjecture}. Under more
specific numerical constraints, we can continue further to recover a known
case and a new case of a well-known conjecture on Lagrangian fibrations, see Theorem
\ref{thm:MarksushevichSawon} and Theorem \ref{thm:TBHTHS}.

Let $X$ be a K3 surface surface with $\mathrm{Pic}(X)=\Z\cdot H$, where $H$ is an
ample line bundle.
 We set $H^2=2d$ for some $d \in \Z, d \geq 1$. 
As in the previous section, we write $\sigma_{t,b} := \sigma_{\omega, \beta}$ 
for $\omega = tH, t > 0$ and $\beta = bH$. 
Let $v=(1,0,1-n)$. Theorem \ref{thm:BridgelandToda} implies
$M_{\sigma_{t, b}}(v) = \Hilb^n(X)$ for $t \gg 0$ and $b < 0$.
We let $\II$ be the universal family of ideals on $\Hilb^n(X)$.
We will denote by $\tilde H \subset \Hilb^n(X)$ the divisor of subschemes intersecting a given curve
$C$ in the linear system $\abs{H}$, and by $2B \subset \Hilb^n(X)$ the divisor of non-reduced
subschemes; $\tilde H$ and $B$ are a basis for $\NS(\Hilb^n(X))$. 

\begin{Ex} \label{ex:HilbChow}
One wall of the Gieseker chamber is always $b = 0$: In this case, for all $I_Y \in \Hilb^n(X)$, $I_Y[1]\in\AA(H,\beta=0)$,
$Z(I_Y[1]) \in \R_{<0}$, and the following 
short exact sequence in $\PP(1)$ makes $I_Y[1]$ strictly semistable:
\[
0 \to \OO_Y \to I_Y[1] \to \OO_X[1] \to 0
\]
Further, considering the filtrations of $\OO_Y, \OO_{Y'}$ by skyscraper sheaves of points,
we see that $I_Y[1]$ and $I_{Y'}[1]$ are $S$-equivalent if and only if $Y, Y'$ define the same point
in the Chow variety. It follows that the corresponding nef divisor $\ell_{\sigma_{t, 0},\II}$ contracts exactly
the curves that are contracted by the Hilbert-Chow morphism, and $\ell_{\sigma_{t, 0},\II} \simpos \tilde H$.

If $\sigma_{t, \epsilon}$ is a stability condition across the Hilbert-to-Chow wall, $\epsilon>0$, then the moduli space
$M_{\sigma_{t, \epsilon}}(v) \cong \Hilb^n(X)$ is unchanged, but the universal family is changed: the
object $I_Y$ is replaced by its derived dual $\RlHom(O_X, I_Y)[1]$. This change affects the
map $l$ in such a way that the image of a path crossing the Hilbert-to-Chow wall in $\Stab(X)$ will bounce
back into the ample cone once it reaches the ray $\R_{>0} \cdot \tilde H \subset \NS(\Hilb^n(X))_\R$.
\end{Ex}

\begin{Ex} \label{ex:firstwall}
We consider the path $\sigma_{t, b}$ with $b = -1 - \epsilon$
as $t \in (0, +\infty)$ varies;  here $\epsilon > 0$ is fixed and sufficiently small, and such that
there exists no spherical object $U$ with $\Im Z_{t, b}(U) = 0$.
Let $t_0$ be such that $Z_{t_0, b}(\OO(-H))$, and
$Z_{t_0, b}(v)$ have the same phase. A direct computation shows $t_0 = \sqrt{\frac 1d} + O(\epsilon)$,
and that for $t > t_0$, the phase of $Z_{t, b}(\OO(-H))$ is bigger than the phase of
$Z_{t, b}(v)$.

We claim that for $t > t_0$, any $I_Y \in \Hilb^n(X)$ is stable.
Consider any destabilizing subobject $A \subset I_Y$ in $\AA_{t, b} = \AA(\omega = H, \beta = -H -
\epsilon H)$; by the long exact cohomology sequence, $A$ is a torsion-free sheaf.
Indeed, if we denote by $B$ the quotient $I_Y/A$ in $\AA_{t,b}$, we have
\[
0 \to \HH^{-1}(B) \to \HH^0(A)=A \to I_Y \to \HH^0(B) \to 0,
\]
and $\HH^{-1}(B)$ is torsion-free by definition of the category $\AA_{t,b}$.

As in the proof of Lemma \ref{lem:Giesekerchamber}, consider any slope-semistable sheaf $A_i$ appearing
in the HN-filtration of $A$ with respect to ordinary slope-stability.
If we write $v(A_i) = (r, cH, s)$, we have $r > 0$ and 
\[
\Im Z_{t, b}(A_i) = t \cdot 2d \cdot\bigl(c + r  + r \epsilon\bigr)
\]
We must have
\[
\Im Z_{t, b}(A_i) < \Im Z_{t, b}(I_Y) = t\cdot2d \cdot (1 + \epsilon)
\]
and hence $c + r < 1$, or $c \le -r$.
This implies
\[ \Im Z_{t, b}(A_i) \le t \cdot 2d \cdot r \epsilon. \]
Unless $A_i \cong O(-H)^{\oplus r}$, the stable factors of $A_i$ have rank at least two, and
thus $\delta_{t, b}(A_i) \ge \frac 34$, where $\delta_{t, b} = \delta_{tH, bH}$ is given
in Definition \ref{def:discrepancy}.  The same computation that led
to equation \eqref{eq:Zrewrite} then shows
\[
\Re Z_{t, b}(A_i) \ge t^2d - 1 + \frac 34 + O(\epsilon).
\]
This shows that for $t> t_0$, the object $A_i$ has strictly smaller phase than
$I_Y$, and so $I_Y$ is stable for $t > t_0$.

On the other hand, whenever there is a curve $C \in \abs{H}$ containing $Y$,
the short exact sequence
\begin{equation}\label{eq:DestabilizingIdeals}
0 \to \OO(-H) \to I_Y \to \OO_C(-Y) \to 0
\end{equation}
will make $I_Y$ strictly semistable for $t = t_0$.

This wall is totally semistable when such a curve exists for any $Y$, that is
if and only if $n \le h^0(\OO(H)) = d+1$. To determine whether this is a fake wall or
corresponds to a wall in the nef cone, we would have to determine whether there exists a curve of
$S$-equivalent objects. Rather than answering this question in general, we just observe that if
there exists a curve $C \in \abs{H}$ for which there exists degree $n$ morphism
$g^1_n \colon C \to \P^1$, then the pullbacks $(g^1_n)^*(-x)$ for $x \in \P^1$ give a non-trivial family
of subschemes $Y \subset C$ for which $\OO_C(-Y)$ is constant. Thus the corresponding family
of ideal sheaves $I_Y$, fitting in the exact sequence \eqref{eq:DestabilizingIdeals}, have a filtration into semistable factors with constant filtration quotients,
and hence they are $S$-equivalent. 

By \cite[Corollary 1.4]{Lazarsfeld:BN-Petri} and general Brill-Noether theory, a $g^1_n$ exists for
any smooth curve $C \in \abs{H}$ if and only if $n \ge \frac {d+3}2$.

\begin{Prop} \label{prop:Hilbnefcone}
Let $X$ be a K3 surface with $\Pic X = \Z\cdot H$, $H^2=2d$, and $n \ge \frac {d+3}2$.
The nef cone of $\Hilb^n(X)$ is generated by $\tilde H$ and
$\tilde H - \frac{2d}{d+n} B$.
\end{Prop}

\begin{Prf}
Since ideal sheaves are $\sigma_{t,b}$-stable, for $t>t_0$, the family $\II$ is a family of $\sigma_{t_0,b}$-semistable objects.
Under the assumptions, we have shown above that there exists a curve $R \cong \P^1 \subset
\Hilb^n(X)$ parametrizing objects that are $S$-equivalent with respect to $\sigma_{t_0, b}$.
Theorem \ref{thm:main1}, applied to $\Hilb^n(X)$ with $\II$ as a family of
$\sigma_{t_0, b}$-semistable objects, implies that $\ell_{\sigma_{t_0, b},\II}$ is a nef divisor 
with $\ell_{\sigma_{t_0, b},\II}.R=0$. 

The Hilbert-Chow morphism shows that $\tilde H$ is also an extremal ray of the nef cone; 
it remains to compute the class of $\ell_{\sigma_{t_0, b},\II}$.
Since $C$ intersects a general element in $\abs{H}$
in $2d$ points, and since the linear system given by $R$ will vanish at each point exactly once, we
have $R.\tilde H = 2d$. On the other hand, $R \cap 2B$ is the ramification divisor
of the map $g^1_n \colon C \to \P^1$; the Riemann-Hurwitz formula gives $R.2B = 2d + 2n$. 
Since $\ell_{\sigma_{t_0, b},\II}.R = 0$, this implies the claim.
\end{Prf}
\end{Ex}

For $n$ sufficiently big, the same result could possibly be obtained using the
technique of $k$-very ample vector bundles; see \cite[Section 3]{ABCH:MMP} and references therein.

The paper \cite{KnutsenCiliberto} discusses the existence of Brill-Noether divisors
on normalizations of curves with $\delta$ nodes in the linear system $\abs{H}$; the authors ask whether 
these produce extremal curves of the Mori cone, see \cite[Question 8.4]{KnutsenCiliberto}. 
The above proposition answers their question in the case $\delta = 0$; it would be
interesting to see whether Brill-Noether divisors on nodal curves could also produce curves of
$S$-equivalent objects for different walls in $\Stab(X)$ (in the case $n < \frac{d+3}2$).

\begin{Rem} \label{rem:HT_ER_conjecture}
The Beauville-Bogomolov form, along with the intersection pairing
\[
N_1(\Hilb^n(X)) \times N^1(\Hilb^n(X)) \to \R,
\]
induces an isomorphism $N_1 \cong \left(N^1\right)^\vee \cong N^1$.
Since $\theta_v(0, -H, 0) = \tilde H$ and 
$\theta_v(1, 0, n-1) = -B$, the Beauville-Bogomolov pairing is determined by
\[ \tilde H^2 = 2d, \quad B^2 = -2n+2, \quad \text{and} \quad (H, B) = 0.\]
(See also \cite[Section 1]{HassettTschinkel:ExtremalRays}.)
Thus, the computation in the proof of Proposition \ref{prop:Hilbnefcone} shows that
the isomorphism identifies $R$ with
$\tilde H + \frac{d+n}{2n-2} B$. 
So the self intersection of $R$ (with respect to the Beauville-Bogomolov form) is given by
\[
(R, R) = \tilde H^2 + \left(\frac{d+n}{2n-2}\right)^2 B^2 
= 2d - \frac{\left(d+n\right)^2}{2n-2} = -\frac{n+3}2 + \frac{(d+1)(2n-d-3)}{2n-2}.
\]
As pointed out to us by Eyal Markman, this does not seem fully consistent with a conjectural
description of the Mori cone by Hassett and Tschinkel, see \cite[Conjecture
1.2]{HassettTschinkel:ExtremalRays}. While $n \ge \frac{d+3}2$ implies $(R, R) \ge -\frac{n+3}2$
in accordance with their conjecture, in general the Mori cone is smaller than predicted. Let $h, b$ be the
primitive curve classes on the rays dual to $\tilde H, B$; they are characterized by
$h.\tilde H = 2d, b.B = 1$ and $h.B = b.\tilde H = 0$. Our extremal curve is give by
$R = h + (d+n)b$. If we let $\overline{R} = h + (d+n+1)b$, then 
\begin{align*}
(\overline{R},\overline{R}) = 2d - \frac{\left(d+n+1\right)^2}{2n-2} &> -\frac{n+3}2 \quad
\text{for $n \gg 0$} \\
\overline{R}.(\tilde H - \epsilon B) & > 0.
\end{align*}
However, since $\overline{R}.\left(\tilde H - \frac{2d}{d+n} B\right) < 0$, the class 
$\overline{R}$ cannot be contained in the Mori cone, in contradiction to \cite[Conjecture
1.2]{HassettTschinkel:ExtremalRays}. The smallest example is the case 
$d = 2$ and $n=5$ and $\overline{R} = h + 8b$, which had been obtained earlier
by Markman, \cite{Eyal:private}.
\end{Rem}

\begin{Ex}\label{ex:Sawon}
The previous example considered the case where $n$ is large compared to the genus.  Let us now
consider a case where the number of points is small compared to the genus: $d=k^2(n-1)$ for
some integers $k\geq 2$.

With the notation as in the previous examples, we now consider the path
of stability conditions $\sigma_{t, -\frac 1k}$ for $t > 0$. Then we are in the situation
of Example \ref{ex:Dragos} with $A = k$ and $B = -1$; more specifically, since
$\frac{d+1}{k}$ is not an integer, we are in Case 1,
and the moduli space $M_{\sigma_{t, -\frac 1k}}(v)$ is isomorphic to the Hilbert scheme
$\mathrm{Hilb}^n(X)$ for all $t>0$.
Markushevich and Sawon proved the following result:
\begin{Thm}[{\cite{Markushevich:Lagrangian, Sawon:LagrangianFibrations}}] \label{thm:MarksushevichSawon}
Let $X$ be a K3 surface with $\Pic X = \Z \cdot H$, $H^2 = 2d$, and 
$d = k^2(n-1)$ for integers $k \ge 2, n$.
\begin{enumerate}
\item $\mathrm{Nef}(\mathrm{Hilb}^n(X))$ is generated by
\[
\theta(w_{\infty\cdot H,-1/k})=\theta(0,-H,0) = \tilde H \quad
\text{ and }\quad \theta(w_{0\cdot H, -1/k})=\theta(k,-H,(n-1)k) = \tilde H - k B.
\]
\item All nef divisors are semi-ample. The morphism induced by $w_{\infty\cdot H,-1/k}$ is the Hilbert-to-Chow morphism, while the one induced by $w_{0\cdot H, -1/k}$ is a Lagrangian fibration.
\end{enumerate}
\end{Thm}
The fact that $w_{\infty \cdot H,-1/k}$ induces the Hilbert-to-Chow morphism follows in our
setup simply from the equality $w_{\infty \cdot H,-1/k}=w_{\sigma_{t\cdot H,0}}$ and Example \ref{ex:HilbChow}.

To reprove the existence of the Lagrangian fibration, we can proceed exactly as in
\cite{Sawon:LagrangianFibrations}, except that Bridgeland stability guides and simplifies the
arguments: indeed, since $w_{0\cdot H, -\frac 1k}^2 = 0$, Lemma
\ref{lem:isotropic} shows that the moduli space $Y:=M_{\sigma_{t, -\frac 1k}}(-w_{0\cdot H, -1/k})$ is
a smooth K3 surface.
Let $\Phi$ denote the induced Fourier-Mukai transform
$\Phi \colon \Db(X)\to\Db(Y,\alpha)$; then $\Phi(v)$ has rank $0$.
Since $Z_{t, -\frac 1k}(-w_{0\cdot H, -1/k}) \in \R_{<0}$,
skyscraper sheaves on $Y$ are $\Phi(\sigma_{t, -\frac 1k})$-stable of phase $1$; hence the stability condition
$\Phi(\sigma_{t, -\frac 1k})$ is again of the form $\sigma_{\widehat{H}_t,\widehat{\beta}_t}$ constructed at the beginning of Section \ref{sec:reviewK3}, up to the action of $\widetilde{\GL}^+_2(\R)$.
Since $\Im Z_{t H, -\frac 1k H}(v)$ is minimal, this has to be true also for $\Im Z_{\widehat{H}_t,\widehat{\beta}_t}(\Phi(v))$.
But then, since the rank of $\Phi(v)$ is zero, this implies that $\widehat{\beta}_t=\widehat{\beta}$ is constant in $t$ and $\widehat{H}_t = u(t) \widehat{H}$, for some function $u(t)$.
We claim that for $t\mapsto 0$, we have $u(t)\mapsto\infty$.
Indeed, for $t\mapsto 0$, we must have either $u(t)\mapsto 0$, or $u(t)\mapsto\infty$, and if $u(t)\mapsto 0$, then, by Lemma \ref{lem:ExplicitFormula}, $w_{0\cdot \widehat{H},\widehat{\beta}}\neq(0,0,1)$, which is a contradiction.

Hence, for $t\mapsto 0$, via the equivalence $\Phi$, we are in the Gieseker chamber for $Y$.
It follows that the moduli space $M_{\sigma_{t, b}}(v)\cong M_{\Phi(\sigma_{t, b})}(\Phi(v))$ is isomorphic to
a moduli space of twisted Gieseker stable sheaves of rank 0; 
as is well-known and discussed in Example \ref{ex:rk0}, the latter admits a Lagrangian fibration.
\end{Ex}

\begin{Ex}\label{ex:AB}
The Hilbert scheme of $n$ points admits a divisor $D$ with $q(D) = 0$ if and only if
$h^2d=k^2(n-1)$ for integers $h, k$. The ``Tyurin-Bogomolov-Hassett-Tschinkel-Huybrechts-Sawon
Conjecture''
would imply that in this case, the Hilbert scheme admits a birational model with a Lagrangian
fibration; we refer to \cite{Verbitsky:HyperkaehlerSYZ} for some discussion of the history of
the conjecture, and \cite{Beauville:problemlist} for some context.

We now consider the first unknown case: 
\begin{Thm} \label{thm:TBHTHS}
Let $X$ be a K3 surface with $\Pic X = \Z \cdot H$ and
$H^2 = 2d$. Assume that there is an odd integer $k$ with $d = \frac{k^2}4(n-1)$ for some
integer $n$. Then:
\begin{enumerate}
\item \label{enum:movable}
The movable cone $\mathrm{Mov}(\mathrm{Hilb}^n(X))$ is generated by
\begin{align*}
& \theta(w_{\infty\cdot H,-\frac 2k \cdot H})=\theta(0,-H,0)=\tilde H\text{, and}\\
& \theta(w_{0\cdot H,-\frac 2k\cdot H})=\theta(k,-2\cdot H,(n-1)k)=2 \tilde H - k B.
\end{align*}
\item \label{enum:TBHTHS}
The morphism induced by $w_{\infty\cdot H,-2/k}$ is the Hilbert-to-Chow morphism, while the one induced by $w_{0\cdot H, -2/k}$ is a Lagrangian fibration on a minimal model for $\Hilb^n(X)$.
\item \label{enum:allMMP}
All minimal models for $\mathrm{Hilb}^n(X)$ arise as moduli spaces of stable objects in $\Db(X)$ and their birational transformations are induced by crossing a wall in $\Stab^\dagger(X)$.
\end{enumerate}
\end{Thm}

\begin{Prf}
We consider the family of stability conditions $\sigma_{t, -\frac 2k}$
for $t > 0$.
As in the previous case, the stability condition $\sigma_{t, -\frac 2k}$ exists for all $t > 0$,
since $\frac{4d+1}k\notin\Z$.
We will study the wall-crossing for the moduli spaces
$M_{\sigma_{t, -\frac 2k}}(v)$ for $v = (1, 0, 1-n)$.

Proceeding as in Example \ref{ex:Sawon}, we consider the smooth projective K3 surface
$Y := M_{\sigma_{t, -\frac 2k}}(-w_{0\cdot H,-2/k\cdot H})$ and the induced Fourier-Mukai transform 
$\Phi \colon \Db(X) \to \Db(Y, \alpha)$. 
The assumption $d = \frac{k^2}4(n-1)$ implies that $\Phi(v)$ has rank 0, and the same computation
as above shows that $t \mapsto 0$ on $X$ corresponds to $t \to +\infty$ via $\Phi$; in particular,
for $t \ll 1$ the stability condition $\sigma_{t, -\frac 2k}$ is in the Gieseker chamber
of $Y$ for $\Phi(v)$; thus $M_0 := M_{\sigma_{\epsilon, -\frac 2k}}(v)$
admits a Lagrangian fibration, as discussed in Example \ref{ex:rk0}.

This also shows that this path meets walls only at finitely many points $t_1, \dots, t_m \in \R$.
Denote the moduli stacks of stable objects by  $\MMM_i := \MMM_{\sigma{t_i + \epsilon, - \frac
2k}}(v)$, and their coarse moduli spaces by $M_i$. By the results in Section \ref{sec:ProjK3}, 
each $M_i$ is a smooth irreducible projective variety of dimension $2n$.
We will first prove by descending induction on $i$ that $M_i$
is birational to the Hilbert scheme $\Hilb^n(X) = M_{m+1}$; this will prove
claims \eqref{enum:movable} and \eqref{enum:TBHTHS}.

Let $(Z_i, \AA) = \sigma_i = \sigma_{t_i, -\frac 2k}$ be a stability condition on one of the walls.
We want to show that the set of $\sigma_i$-stable objects of class $v$ is non-empty.

Note that $\Im(Z_t)\in \frac 1k 2dt\cdot\Z$ and $\Im(Z_t(v))= \frac 1k 4dt$ for all $t \in \R$.
Thus, if an object $E\in\MMM_{i+1}(\C)$ is strictly $\sigma_i$-semistable, then it fits into a short exact
sequence $A \into E \onto B$
in $\AA$ where $A, B$ are stable of the same phase, and satisfy 
\begin{equation} \label{eq:ImAB}
\Im(Z_i(A))=\Im(Z_i(B))= \frac 1k 2dt_i.
\end{equation}
(In fact we have $Z_i(A) = Z_i(B) = \frac 12 Z_i(v)$, and there are only finitely many classes
$v(A), v(B)$ that are possible, see e.g. \cite[Lemma 3.15]{Toda:K3}.)
Mukai's Lemma, see \cite[Corallary 2.8]{Mukai:BundlesK3} and \cite[Lemma 5.2]{Bridgeland:K3}, shows
that 
\begin{equation} \label{eq:Mukaiineq}
\mathrm{ext}^1(E,E)\geq\mathrm{ext}^1(A,A)+\mathrm{ext}^1(B,B).
\end{equation}
Since $A \neq B$ are stable and of the same phase, we have $\Hom(A, B) = 0 = \Hom(B, A) = \Ext^2(A, B)$
and thus $-(v(A), v(B)) = \chi(A, B) = - \mathrm{ext^1}(A, B)$; in particular, $\mathrm{ext}^1(A, B)$ is constant
as $A, B$ vary in their moduli space of stable objects.
Since $A, B, E$ are simple, the bilinearity of $\chi$ then gives
\begin{equation} \label{eq:simpleineq}
2 \mathrm{ext}^1(A,B) = 2 + \mathrm{ext}^1(E,E)-\mathrm{ext}^1(A,A)-\mathrm{ext}^1(B,B).
\end{equation}

We can distinguish two cases:
\begin{enumerate}
\item\label{case:AllSemist} $(v(A), v(B)) = \mathrm{ext}^1(A,B)=1$.
\item\label{case:Mukai} $(v(A), v(B)) = \mathrm{ext}^1(A,B)\geq2$.
\end{enumerate}

In case \eqref{case:Mukai}, the point on $M_{i+1}$ corresponding to $E$ lies on a rational curve:
by Lemma \ref{lem:localP2}, each of the extensions parameterized by $\P \Ext^1(A, B)$ is 
$\sigma_{t_i + \epsilon, -\frac 2k}$-stable.
As $M_{i+1}$ is $K$-trivial, it is not covered by rational curves;
this means that a generic object $E \in \MMM_{i+1}(\C)$ cannot be destabilized
by objects of classes $v(A), v(B)$. 

We will now proceed to show that case \eqref{case:AllSemist} cannot appear. This will imply
that there is an open subset $U \subset M_{i+1}$ of objects that are still stable on the 
wall at $\sigma_i$, and thus $U$ gives a common open subset of $M_i, M_{i+1}$, inducing
a birational map.

So assume we are in case \eqref{case:AllSemist}. Then $\mathrm{ext}^1(A', B') = \C$ will hold
for any pair $A' \in \MMM_{\sigma_i}(v(A))(\C)$, $B' \in \MMM_{\sigma_i}(v(B))(\C)$.
By Lemma \ref{lem:localP2}, the objects $E'$ corresponding to the unique extension
$A' \into E' \onto B'$ will be $\sigma_{t_i + \epsilon, -\frac 2k}$-stable; thus 
the universal family of extensions over
$\MMM_{\sigma_i}(v(A)) \times \MMM_{\sigma_i}(v(B))$ induces an injective morphism
\[
M_{\sigma_i}(v(A)) \times M_{\sigma_i}(v(B)) \into M_{i+1}
\]
However, equations \eqref{eq:Mukaiineq} and \eqref{eq:simpleineq} also show that in 
case \eqref{case:AllSemist} we have
\[
\mathrm{ext}^1(E,E)=\mathrm{ext}^1(A,A)+\mathrm{ext}^1(B,B),
\]
in other words the above
morphism is an injective morphism between projective varieties of the same dimension. As each 
$M_{i+1}$ is an irreducible holomorphic symplectic variety, this is a contraction unless
one of the moduli spaces on the left is a point, i.e., unless $A$ or $B$ is a spherical object.

Let $\xi$ be the Mukai vector of this spherical object. Then $\xi^2 = -2$ and the assumption
of case \eqref{case:AllSemist} implies $(v-\xi, \xi) = 1$, and so $(v, \xi) = -1$.
If $u = (0, \frac Hk, 1-n)$, then $\Im Z_t(\blank) = t (u, \blank)$, and so  equation
\eqref{eq:ImAB} implies $(u, \xi)  = \frac 12 (u, v) = \frac 12 (n-1)$. By Lemma
\ref{lem:snotexist}, such a class $\xi \in H^*_{\alg}(X, \Z)$ does not exist. This finishes
the proof of claims \eqref{enum:movable} and \eqref{enum:TBHTHS}.

It remains to prove claim \eqref{enum:allMMP}. By the previous part, wall-crossing induces a chain
of birational maps
\[
M_0 \dashrightarrow M_1 \dashrightarrow \dots \dashrightarrow M_n = \Hilb^n(X)
\]
As each $M_i$ is a smooth $K$-trivial variety, the moduli spaces are isomorphic outside of
codimension two, and thus we can canonically identify their N\'eron-Severi groups:
\[
\NS(M_0) = \NS(M_1) = \dots = \NS(\Hilb^n(X)).
\]
Theorem \ref{thm:nefdivisor} produces
maps $l_i \colon I_i \to N^1(\Hilb^n(X))$, where $I_0 = (0, t_0]$, $I_1 = [t_0, t_1]$, etc.
By Lemma \ref{lem:continuous} below, the maps $l_i, l_{i+1}$ agree on the overlap $t_i$; hence they
produce a continuous path $l \colon (0, +\infty) \to \NS(\Hilb^n(X))$ with starting point at
a Lagrangian fibration, and endpoint at a divisorial contraction. Since $\Hilb^n(X)$ has
Picard rank two, this path must hit the nef cone of any minimal model of $\Hilb^n(X)$.
\end{Prf}
\end{Ex}

\begin{Lem}  \label{lem:snotexist}
Let $X$ be K3 surface with assumptions as in Theorem \ref{thm:TBHTHS}. 
Let $u, v \in H^*_{\alg}(X, Z)$ be given by
$v = (1, 0, 1-n)$ and $u = (0, \frac Hk, 1-n)$. Then there exists no class
$\xi \in H^*_{\alg}(X, Z)$ satisfying the following three equations:
\begin{align*}
(u, \xi) = \frac{n-1}2 \quad \text{and} \quad
(v, \xi)  = -1 \quad \text{and} \quad
\xi^2  = -2.
\end{align*}
\end{Lem}
\begin{Prf}
Observe that $v^2 = 2n- 2$, that $(u, v) = n-1$ and $u^2 = \frac{2d}{k^2} = \frac 12 (n-1)$. 
Hence the orthogonal complement of $u$ and $v$ is given by
$\bigl(\R \cdot u + \R \cdot v \bigr)^\perp = \R \cdot (2u-v) \subset H^*_{\alg}(X)_\R$.
As $\xi_0 = (\frac 12, 0, \frac 12(n+1))$ satisfies the linear equations of the lemma, any solution
to the linear equations above is given by $\xi = \xi_0 + \alpha (2u - v)$. Since the rank of
$\xi$ must be integral, we have $\alpha \in \frac 12 + \Z$.

Using $\xi_0 ^2 = -\frac 12(n+1)$ and $(\xi_0, 2u-v) = n$ as well as $(2u-v)^2 = 0$, we obtain
\[
\xi^2 = -2 \Leftrightarrow (4 \alpha - 1) \cdot n = -3.
\]
As $4\alpha \in \Z$, this is a contradiction to $n \ge 5$.
\end{Prf}

\begin{Lem} \label{lem:continuous}
Let $X$ be a projective K3 surface, and let $v \in H^*_{\alg}(X, \Z)$ be a primitive class with $v^2
\ge -2$. With notation as in Theorem \ref{thm:MMP}, assume that there
there exists a $\sigma_0$-stable objects of class $v$; identify 
the N\'eron-Severi groups of $M_{\sigma_{\pm}}(v)$ by extending the birational morphism $M_{\sigma_+}(v)
\dashrightarrow M_{\sigma_-}(v)$ induced by the common open subset $M_{\sigma_0}^s(v)$ to an isomorphism
outside of codimension two.

Under this identification, we have an equality $\ell_{\sigma_{0},\EE_+} = \ell_{\sigma_0,\EE_-}$ in the previous identification $N^1(M_{\sigma_+}(v)) = N^1(M_{\sigma_-}(v))$. 
\end{Lem}

The proof is based on the notion of \emph{elementary modification} in the derived category,
introduced by Abramovich and Polishchuk.

\begin{Prf}
First note that since $M_{\sigma_\pm}(v)$ are $K$-trivial, there exists a common open subset $V \subset M_{\sigma_\pm}(v)$
that contains $M_{\sigma_0}^s(v)$ and has complement of codimension 2; this gives the
identification of the N\'eron-Severi groups.

We can normalize the stability condition $\sigma_0 = (Z_0, \PP_0)$ such that $\MMM_{\sigma_\pm}(v)$ parameterizes
families of $\sigma_0$-semistable objects of phase 1; we can also assume that $\sigma_0$ is
algebraic.

Since $M_{\sigma_\pm}(v)$ are projective, $\ell_{\sigma_0, \EE_\pm}$ are determined by their degrees on smooth curves that
are contained in $V$ and intersect $M_{\sigma_0}^s(v)$; let $C \subset V$ be such a curve,
and let $U \subset C$ be the open subset $C \cap M_{\sigma_0}^s(v)$. 
Since the Brauer group of a curve is trivial, there exists a universal family
$\EE_{+,C}$ on $C$, corresponding to a lift $C \to \MMM_{\sigma_+}(v)$ of the map
$C \into V \into M_{\sigma_+}(v)$.
Similarly, we obtain a family $\EE_{-,C}$ via the embedding $C \into V \into M_{\sigma_-}(v)$. We can choose
$\EE_{\pm,C}$ such that $\EE_{+,C}|_U \cong \EE_{-,C}|_U$. Using further twists by a line bundle,
we can assume that this isomorphism extends to a morphism $\varphi \colon \EE_{+,C} \to \EE_{-,C}$
(this follows, e.g., from \cite[Proposition 2.2.3]{Lieblich:mother-of-all}).

We can think of $\EE_{\pm,C}$ as flat families of objects in $\PP_0(1)$.  It is proved in \cite[Lemma
4.2.3]{Abramovich-Polishchuk:t-structures} that the morphism $\varphi$ can be given by a sequence
\[
\EE_{+,C} = \EE_0 \to \EE_1 \to \dots \to \EE_k = \EE_{-,C}
\]
of ``elementary modifications'' $\EE_{i-1} \to \EE_{i}$: this means that there is a point $c_i \in C
- U$, an object $Q_i \in \PP_0(1)$ such that $\EE_{i-1}$ fits into the short exact sequence
\[
\EE_{i-1} \into \EE_i \onto (i_{c_i})_* Q_i
\]
(This is proved in \cite[Section 4.2]{Abramovich-Polishchuk:t-structures} in the case of a smooth
affine curve, and $U$ the complement of a point, but the proof extends directly to our case.)

Since $\Im Z_0(Q_i) = 0$, we obtain
\[
\ell_{\sigma_0, \EE_-}.C = \Im Z_0((p_X)_* \EE_{-,C}) = \Im Z_0((p_X)_* \EE_{+,C}) + \sum_i \Im Z_0 (Q_i)
= \ell_{\sigma_0, \EE_+}.C
\]
as claimed.
\end{Prf}

\begin{Ex}
Finally, let us point out that in the case of the Hilbert scheme of $n$ points on an 
abelian surface $T$ of Picard rank one, Yanagida-Yoshioka and Meachan (in the case of a principally polarized
abelian surface) have obtained examples with an infinite series of
walls inside the subset $U(T)$ of ``geometric'' stability conditions; see \cite[Section
4]{Meachan:thesis} and \cite[Section 5 and 6]{YY:abeliansurfaces}.
Their examples have in common that each series of walls contains an infinite sequence
of bouncing walls, corresponding to divisorial contractions.
Between two bouncing walls, there may be finitely many flopping walls; the corresponding 
line bundle $\ell_{\sigma}$ will traverse finitely many birational models of $\Hilb^n(T)$.
\end{Ex}

\bibliography{all}                      % .bib-Datei
\bibliographystyle{alphaspecial}     % .bst-Datei

\end{document}